\documentclass[oneside, 11pt]{article}
\usepackage{setspace}
\usepackage{titlesec}
\usepackage[utf8]{inputenc}
\usepackage{tabularx}	
\usepackage{etoolbox}
\usepackage{mathtools}
\usepackage{amsmath}
\usepackage{amssymb} 
\usepackage{amsfonts}
\usepackage{mathrsfs} 
\usepackage{amsthm,pifont}
\usepackage{graphicx}
\usepackage{bbm}
\usepackage[colorlinks=true, pdfstartview=FitV, linkcolor=blue,
citecolor=blue, urlcolor=blue, pagebackref=false]{hyperref}
\usepackage[nottoc,numbib]{tocbibind}
\usepackage{xcolor}
\definecolor{truepurp}{RGB}{160,15,55}

\usepackage{tikz}
\usepackage{pgfplots,pgf}

\usepackage{cite}
\usepackage{braket}
\usepackage{leftidx,tensor}
\usepackage{dsfont}
\usepackage{calc}
\usepackage{caption}
\usepackage{subcaption}
\usepackage{color}
\usepackage[printonlyused,nohyperlinks]{acronym}
\usetikzlibrary{decorations.markings}
\usetikzlibrary{shapes.geometric}
\usetikzlibrary{patterns}
\tikzstyle{vertex}=[circle, draw, inner sep=0pt, minimum size=6pt]
\newcommand{\vertex}{\node[vertex]}

\usepackage[utf8]{inputenc}
\usepackage[english]{babel}
\usepackage{enumitem}

\allowdisplaybreaks

\newtheorem{theorem}{Theorem}[section]
\makeatletter
\patchcmd{\ttlh@hang}{\parindent\z@}{\parindent\z@\leavevmode}{}{}
\patchcmd{\ttlh@hang}{\noindent}{}{}{}
\makeatother
\titleformat*{\section}{\large\bfseries}
\titleformat*{\subsection}{\small\bfseries}
\titleformat*{\subsubsection}{\small\bfseries}
\titleformat*{\paragraph}{\small\bfseries}
\titleformat*{\subparagraph}{\small\bfseries}

\newcommand{\N}{\mathbb{N}}
\newcommand{\R}{\mathbb{R}}
\newcommand{\Z}{\mathbb{Z}}
\newcommand{\E}{\mathbb{E}}
\newcommand{\p}{\mathbb{P}}

\newcommand{\md}{\ensuremath{\mathrm{d}}}
\newcommand{\eps}{\varepsilon}

\newcommand{\cC}{\mathcal{C}}

\newcommand{\cF}{\mathcal{F}}

\newcommand{\cP}{\mathcal{P}}

\newcommand{\mz}{\mathbf{0}}
\newcommand{\pre}{\text{PRE}}
\newcommand{\T}{\textbf{T}}
\newcommand{\bp}{\mathbf{P}}
\newcommand{\be}{\mathbf{E}}
\newcommand{\ov}{\text{OV}}

\newtheorem{lemma}[theorem]{Lemma}
\newtheorem{remark}[theorem]{Remark}
\newtheorem{proposition}[theorem]{Proposition}

\newtheorem{claim}[theorem]{Claim}

\newtheorem{conjecture}[theorem]{Conjecture}

\usepackage{geometry}
\begin{document}

	\title{\vspace{-1cm}Truncation of long-range percolation with non-summable interactions in dimensions $d\geq 3$}

	\author{Johannes B\"aumler\footnote{ \textsc{Department of Mathematics, University of California, Los Angeles}.\\ \text{  \ \ \ \ \ \ } E-Mail: \href{mailto:jbaeumler@math.ucla.edu}{jbaeumler@math.ucla.edu}
	}
	}

	\maketitle
	
	\vspace{-3mm}
	
	\begin{center}
		\parbox{13cm}{ \textbf{Abstract.} 
			Consider independent long-range percolation on $\mathbb{Z}^d$ for $d\geq 3$. Assuming that the expected degree of the origin is infinite, we show that there exists an $N\in \N$ such that an infinite open cluster remains after deleting all edges of length at least $N$. For the isotropic case in dimensions $d\geq 3$, we show that if the expected degree of the origin is at least $10^{400}$, then there exists an infinite open cluster almost surely. We also use these results to prove corresponding statements for the long-range $q$-states Potts model.
	}
	\end{center}
	
	\let\thefootnote\relax\footnotetext{{\sl MSC Class}: 82B43, 60K35 }
	\let\thefootnote\relax\footnotetext{{\sl Keywords}: Long-range percolation, phase transition, truncation, non-summable interaction, long-range Potts model}

	\hypersetup{linkcolor=black}
	\tableofcontents
	\hypersetup{linkcolor=blue}

\section{Introduction}

We consider long-range bond percolation on $\Z^d$, where for each $x,y \in \Z^d$ the edge $\{x,y\}$ is open with probability $p_{x-y} \in \left[0,1\right)$, independent of all other edges. We write $\p$ for the associated measure, we write $x \sim y$ if the edge between $x$ and $y$ is open, and we write $x \leftrightarrow y$ if there exists a path of open edges from $x$ to $y$.  For $x\in \Z^d$, we define
\begin{equation*}
	K_x=\left\{y\in \Z^d : x \leftrightarrow y\right\}
\end{equation*}
as the {\sl open cluster} containing $x$. We are interested in the case of non-summable interactions, i.e., when
\begin{equation}\label{eq:non-summable}
	\sum_{x \in \Z^d \setminus \{\mz\}} p_x = \infty.
\end{equation}
The independence of different edges together with a Borel-Cantelli argument directly implies that $|K_\mz|=\infty$ almost surely. In particular, an infinite open cluster almost surely exists. This differs from the summable situation, in which $\sum_{x \in \Z^d \setminus \{\mz\}} p_x < \infty$. Here, an infinite cluster may or may not exist, depending on the precise specifications of the model. For the special case where $\sum_{x \in \Z^d \setminus \{\mz\}} p_x < 1$, an infinite cluster does almost surely not exist, as shown by a comparison with a sub-critical branching process. Note that by linearity of expectation, the quantity $\sum_{x \in \Z^d \setminus \{\mz\}} p_x$ is just the expected degree $\E\left[\deg(\mz)\right]$ of the origin, where the {\sl degree} of a vertex $y$ (denoted by $\deg(y)$) is the number of open edges adjacent to $y$.\\

In this paper, we study the so-called {\sl truncation question} in the non-summable case. This question asks whether long edges are necessary for an infinite cluster to exist. Say that $x \overset{\leq n}{\longleftrightarrow} y$ if there is an open path
from to using only edges of length $n$. More formally, for $x,y \in \Z^d$ we write $x \overset{\leq n}{\longleftrightarrow} y$ if there exists an open path from $x$ to $y$ that uses only the open edges $\{a,b\}$ with $\|a-b\|\leq n$. For $x\in \Z^d$, we define
\begin{equation*}
	K_x^n=\left\{y\in \Z^d : x \overset{\leq n}{\longleftrightarrow} y \right\}
\end{equation*}
as the set of all points $y \in \Z^d$ that are connected to $x$ by an open path using edges of length at most $n$. The question now is whether there exists $n \in \N$ such that $\p\left(|K_\mz^n|=\infty\right) > 0$. Note that by translation invariance and ergodicity, this is equivalent to the question of whether there exists an infinite cluster with probability one after the truncation. The truncation question in the non-summable regime has been frequently considered in previous literature \cite{alves2017note,campos2020truncation,van2016truncated,friedli2004longrange,friedli2006truncation,de2008truncated,menshikov2001note,sidoravicius1999truncated}. In this paper, we solve the problem under minimal assumptions in dimensions $d\geq 3$.\\

Throughout this paper, we assume that the collection of probabilities $\left(p_{x}\right)_{x\in \Z^d}$ is {\bf symmetric}, meaning that
\begin{align}
	\label{eq:symmetry} & p_{x} = p_{y} \text{ for all $x,y \in \Z^d$ for which $|\langle x,e_i \rangle| = |\langle y,e_i \rangle|$ for $i=1,\ldots, d$},
\end{align}
where $e_i \in \Z^d$ denotes the $i$-th standard unit vector of $\R^d$.
Condition \eqref{eq:symmetry} guarantees that the percolation measure is invariant under reflections through the hyper-planes $\left\{z \in \R^d: \langle z, e_i \rangle = 0\right\}$. Specifically, \eqref{eq:symmetry} implies that $p_x=p_{-x}$ for all $x\in \Z^d$. In addition, the percolation measure is always translationally invariant by construction.
Furthermore, we assume that $\left(p_{x}\right)_{x\in \Z^d}$ is {\bf irreducible}, meaning that 
\begin{equation}\label{eq:irreduc}
	\p\left(x \leftrightarrow y\right) > 0 \text{ for all } x,y\in \Z^d.
\end{equation}
If the irreducibility does not hold, one can always consider the problem on (possibly lower-dimensional) sublattices. \\

It is easy to see that $\p\left(|K_\mz^n|=\infty\right) = 0$ in dimension $d=1$ for all $n\in \N$, provided that $p_x < 1$ for all $x\in \Z$, so conditions \eqref{eq:non-summable} through \eqref{eq:irreduc} can not imply the existence of an infinite open cluster in dimension $d=1$ after the truncation. Contrary to that, it is conjectured that the truncation property holds, i.e., that $\p\left(|K_\mz^n|=\infty\right) > 0$ for $n$ large enough, in all dimensions $d\geq 2$, cf. \cite{friedli2006truncation}.
In this paper, we verify this conjecture for dimensions $d\geq 3$. 

\subsection{Main results} Our main results are as follows.

\begin{theorem}\label{theo:main}
	Let $d\geq 3$, and assume that $\left(p_x\right)_{x\in \Z^d}$ satisfies conditions \eqref{eq:non-summable} through \eqref{eq:irreduc}. Then there exists $n\in \N$ such that
	\begin{equation}\label{eq:main existence}
		\p \left(|K_\mz^n|=\infty\right) > 0. 
	\end{equation} 
	Furthermore,
	\begin{equation}\label{eq:main high prob}
		\lim_{n\to \infty} \p \left(|K_\mz^n|=\infty\right) = 1.
	\end{equation} 
\end{theorem}

Furthermore, we also study the same problem in the {\sl isotropic case}, in which the symmetry-condition \eqref{eq:symmetry} is replaced by the stronger condition \eqref{eq:isotropy}. We say that $T \in \{-1,0,1\}^{d \times d}$ is a signed permutation matrix if each row and each column of $T$ contains exactly one non-zero entry. The alternative condition asks that
\begin{equation}\label{eq:isotropy}
	p_x = p_y \text{ if } x=Ty \text{ for some signed permutation matrix $T$}.
\end{equation}
Note that condition \eqref{eq:isotropy} is stronger than condition \eqref{eq:symmetry}. The difference is that condition \eqref{eq:symmetry} requires that $p_x$ is not affected when flipping the sign of one (or more) entries of $x$, whereas condition \eqref{eq:isotropy} requires that $p_x$ is invariant under changing the sign of coordinates and under interchanging the different entries of $x$. In particular, condition \eqref{eq:isotropy} guarantees that there are no distinguished directions among the $d$ coordinate directions. The main example of condition \eqref{eq:isotropy} is the {\sl isotropic} case, in which $p_x = p_y$ for all $x,y \in \Z^d$ for which $\|x\|_q = \|y\|_q$, where $\|\cdot\|_q$ denotes the usual $q$-norm on $\R^d$. Assuming condition \eqref{eq:isotropy}, we obtain the following results.

\begin{theorem}\label{theo:isotropic}
	Let $d\geq 3$, let $T(d)=10^{26}$ for $d\geq 4$, and let $T(3)=10^{400}$. Let $(p_x)_{x\in \Z^d\setminus \{\mz\}}$ be such that $\infty > \sum_{x \in \Z^d \setminus \{\mz\}} p_x > T(d)$, and such that \eqref{eq:irreduc} and \eqref{eq:isotropy} hold. Then
	\begin{equation}\label{eq:isotrop infinite cluster pos prob}
		\p \left( |K_\mz|=\infty \right) > 1 - \frac{1}{e} > 0
	\end{equation}
	and
	\begin{equation}\label{eq:isotrop finite cluster low prob}
		\p \left( |K_\mz| < \infty \right) \leq \exp\left(1- \frac{\sum_{x \in \Z^d \setminus \{\mz\}} p_x}{T(d)} \right).
	\end{equation}
\end{theorem}

\noindent
Although the enormously large constants of $10^{400}$ in dimension $d=3$ and $10^{26}$ in dimensions $d\geq 4$ in Theorem \ref{theo:isotropic} are not relevant for any practical use, it is interesting that the constants do not depend on the dimension $d$ and are uniform over all probability measures $\p$ defined by $(p_x)_{x\in \Z^d\setminus\{\mz\}}$. Further significant reduction of this constant seems feasible using more precise estimates throughout the proof; we do not pursue this in this paper. It is also an interesting problem how small a constant $C$ can be such that the condition $\sum_{x\in \Z^d \setminus \{\mz\}}p_x > C$ guarantees $\p(|K_\mz|=\infty)>0$ in the isotropic case. More generally, it is conjectured by Easo and Hutchcroft \cite[Conjecture 7.3]{easo2023critical} that there exists a constant $C<\infty$ such that for all transitive graphs that are not one-dimensional, the condition $\E\left[\deg(x)\right]>C$ guarantees the existence of an infinite open cluster, where $x$ is a vertex in the graph. Special cases of this conjecture for spread-out percolation have been solved by Penrose \cite{penrose1993spread} and Spanos and Tointon \cite{spanos2024spread}. Generally, one expects the optimal constant $C$ to be much smaller than $10^{26}$. For example, for the standard integer lattice in dimensions $d \in \{2,\ldots,9\}$ it is known that the expected degree at criticality $(2 d p_c(d))$ is at most $2.4$ \cite{gomes2021upper}. 

Also, note that the probability $\p \left( |K_\mz| < \infty \right)$ in \eqref{eq:isotrop finite cluster low prob} does in general not decay sub-exponentially in the expected degree $\E\left[\deg(\mz)\right]=\sum_{x \in \Z^d \setminus \{\mz\}} p_x$, as
\begin{equation*}
	\p(|K_\mz|< \infty) \geq \p(|K_\mz|=1) = \prod_{x \in \Z^d \setminus \{\mz\}} (1-p_x) \geq \exp\Big( - c \sum_{x \in \Z^d \setminus \{\mz\}} p_x \Big)  ,
\end{equation*}
for some constant $c>0$, assuming that $\sup_x p_x \leq \tfrac{1}{2}$. So the exponential decay observed in \eqref{eq:isotrop finite cluster low prob} describes the correct asymptotic behavior of $\p(|K_\mz|< \infty)$ in $\sum_{x \in \Z^d \setminus \{\mz\}} p_x$, up to a very large constant in the exponent.\\

Theorem \ref{theo:isotropic} directly implies that if $\sum_{x \in \Z^d \setminus \{\mz\}} p_x = \infty$ and if conditions \eqref{eq:irreduc} and \eqref{eq:isotropy} hold, then there exists $n \in \N$ such that $\p \left(|K_\mz^n|=\infty\right) > 0$. While Theorem \ref{theo:isotropic} requires the stronger assumption \eqref{eq:isotropy}, compared to the assumptions of mirror-symmetry \eqref{eq:symmetry} in Theorem \ref{theo:main}, it also has stronger results, which are quantitative statements. These stronger assumptions are also necessary,  as the following discussion shows.

\begin{remark}
	The results of Theorem \ref{theo:isotropic} are not true assuming only the conditions of Theorem \ref{theo:main} $($\eqref{eq:symmetry} instead of \eqref{eq:isotropy}$)$.
\end{remark}

\begin{proof}
	To construct a counterexample, let $N\in \N$ and consider the probabilities $(p_x)_{x\in \Z^d\setminus \{\mz\}}$ defined by
	\begin{align*}
		p_x^\eps = \begin{cases}
			0.5 & \text{ if } x=k\cdot e_1 \text{ with $k\in \{-N,\ldots,N\} \setminus \{\mz\}$} \\
			\eps & \text{ if } x = \pm e_i \text{ with $i\in \{2,\ldots,d\}$}   \\
			0 & \text{ otherwise}
		\end{cases},
	\end{align*}
	where $\eps \in (0,1)$ is a parameter. Write $\p_\eps$ for the associated measure and $\E_\eps$ for the expectation under this measure.
	From the construction, it directly follows that $\E_\eps\left[\deg(\mz)\right] = \sum_{x \in \Z^d \setminus \{\mz\}} p_x = N + 2(d-1)\eps > N$ and that \eqref{eq:symmetry} and \eqref{eq:irreduc} hold for $\eps >0$. However, we will argue below that for sufficiently small $\eps = \eps(N) >0$ we have $\p_\eps \left(|K_\mz|=\infty\right)=0$. As $N \in \N$ was arbitrary this shows that the result of Theorem \ref{theo:isotropic} does not hold, assuming only \eqref{eq:symmetry} and \eqref{eq:irreduc} instead of \eqref{eq:irreduc} and \eqref{eq:isotropy}. For $\eps=0$, the percolation reduces to finite-range percolation on one-dimensional fibers of the form $\{x+k\cdot e_1: k\in \Z\}$, so in particular we have $\E_0 \left[|K_\mz|\right] < \infty$. Furthermore, we can write the expectation of $|K_\mz|$ as
	\begin{equation*}
		\E_0 \left[|K_\mz|\right] = \sum_{n\in \Z} \p_0(\mz\leftrightarrow n \cdot e_1) < \infty,
	\end{equation*}
	which implies that there exists $K\in \N_{>0}$ such that
	\begin{equation*}
		\sum_{n \in \{-K,\ldots, K\}} \sum_{k \in \Z\setminus \{-K,\ldots,K\}} \p_0(\mz\leftrightarrow n \cdot e_1) \p_0(n \cdot e_1 \sim k \cdot e_1) < 1 .
	\end{equation*}
	So when we define the set $S\subset \Z^d$ by $S=\{-K\cdot e_1,\ldots,K\cdot e_1\}$ one directly gets that
	\begin{equation*}
		\varphi_0(S) \coloneqq \sum_{x \in S} \sum_{y \in \Z^d \setminus S} \p_0 (\mz \overset{S}{\longleftrightarrow} x) \p_0 ( x \sim y) < 1,
	\end{equation*}
	where we write $\mz \overset{S}{\longleftrightarrow} x$ if there exists an open path contained in $S$ connecting $\mz$ and $x$. 
	For a fixed finite set $S \subset \Z^d$ and $x\in S, y\notin S$, the functions $\eps \mapsto \p_{\eps} (\mz \overset{S}{\longleftrightarrow} x)$ and $\eps \mapsto \p_{\eps} ( x \sim y)$ are continuous functions in $\eps$, so in particular,
	\begin{equation*}
		\varphi_{\eps}(S) \coloneqq \sum_{x \in S} \sum_{y \notin S} \p_{\eps} (\mz \overset{S}{\longleftrightarrow} x) \p_{\eps} ( x \sim y) < 1,
	\end{equation*}
	for $\eps > 0$ small enough. This directly implies that $\p_\eps \left( |K_\mz|=\infty \right)=0$, by the proof of sharpness of the phase transition by Duminil-Copin and Tassion \cite{duminil2016new, duminil2017new}.
\end{proof}

As there are no infinite clusters in one-dimensional finite-range percolation, the results of Theorems \ref{theo:main} and \ref{theo:isotropic} do not hold in dimension $d=1$. This means that dimension $d=2$ is the only case in which the truncation question as formulated above remains open. We conjecture that the results of Theorems \ref{theo:main} and \ref{theo:isotropic} also hold in dimension $d=2$.

\subsection{Related work} The non-summability in \eqref{eq:non-summable} directly ensures that each point $x\in \Z^d$ is contained in an infinite open cluster almost surely.
Whenever the percolation measure is irreducible, then there exists exactly one infinite cluster almost surely. So for each $x,y \in \Z^d$ there exists an open path from $x$ to $y$, as proven by Grimmett, Keane, and Marstrand \cite{grimmett1984connectedness} and by Kalikow and Weiss \cite{kalikow1988random}.\\
Various works have addressed the truncation problem in the non-summable case under different assumptions on the connection probabilities $(p_x)_{x\in \Z^d}$ or on the underlying graph \cite{alves2017note,campos2020truncation,van2016truncated,friedli2004longrange,friedli2006truncation,de2008truncated,menshikov2001note,sidoravicius1999truncated}. In dimension $d=2$, the truncation question was studied in \cite{sidoravicius1999truncated,de2008truncated,menshikov2001note,campos2020truncation} under further regularity conditions on the probabilities $(p_x)_{x\in \Z^d}$. Whenever $p_x$ decays polynomially with $\|x\|$ or when $\limsup_{x\to \infty} p_x>0$, the truncation question was studied in \cite{friedli2004longrange,friedli2006truncation}, where it has an affirmative answer. For the case in which $p_x=f(\|x\|_\infty)$, Friedli and de Lima showed that the truncation question has an affirmative answer in dimensions $d\geq 3$ when $\sum_{n=1}^{\infty} f(n)=\infty$ \cite{friedli2006truncation}.
In this notation, Theorem \ref{theo:main} shows that the truncation question has an affirmative answer in dimensions $d\geq 3$, whenever $\sum_{n=1}^{\infty} f(n) n^{d-1}  = \infty$. 
Analogous questions on oriented graphs were treated in \cite{alves2017note,van2016truncated}.

The truncation problem has also been studied in the summable situation. Here the question is slightly different and asks whether one can remove all sufficiently long edges of a {\sl super-critical} long-range percolation cluster while still retaining an infinite open cluster. This question was first studied by Meester and Steif \cite{meester1996continuity} who showed this for long-range percolation with exponential decay, following earlier work on finite-range percolation by Grimmett and Marstrand \cite{grimmett1990supercritical}. Later, the exponential decay was relaxed to polynomial (but summable) decay in \cite{berger2002transience,baumler2023continuity}. Similar results were also obtained for Poisson Boolean percolation by Dembin and Tassion \cite{dembin2022almost} and for inhomogeneous long-range percolation by Mönch \cite{monch2023inhomogeneous}. The uniqueness of the infinite cluster, whenever it exists, was studied by Burton and Keane \cite{burton1989density}.

\paragraph*{Outline of the paper} 
The main idea in the proofs of Theorems \ref{theo:main} and \ref{theo:isotropic} is to apply the second moment method to a quantity related to the number of open paths starting at the origin. This technique was already used for oriented short-range percolation by Cox and Durrett in dimensions $d\geq 4$ \cite{cox1983oriented} and by Benjamini, Pemantle, and Peres in dimension $d=3$ \cite{benjamini1998unpredictable}. Similar techniques were also used by Kesten in high dimensions \cite{kesten1990asymptotics}. A second-moment bound for this quantity related to the number of open paths of a certain length will then imply that for all $n\in \N$ there exists a path of length $n$ that starts at the origin with a uniform (in $n$) positive probability. This implies the existence of an infinite cluster. The main input and setup of the second-moment bound are described in Proposition \ref{prop:measure on paths}. Then, we apply Proposition \ref{prop:measure on paths} to the different situations considered in Theorems \ref{theo:main} and \ref{theo:isotropic}. In order to apply Proposition \ref{prop:measure on paths} to the different setups considered in this paper, we choose specific measures on directed random walks in $d$ dimensions, which are adapted to long-range percolation. Furthermore, we also prove upper bounds on the probability that two independent paths sampled by this measure intersect. \\

For the proofs of Theorem \ref{theo:isotropic} in dimensions $d\geq 4$, we need some inequalities for sums of independent random variables. We prove these results for sums of independent random variables in Section \ref{sec:randvar}. 
Subsequently, we adapt the notion of unpredictable paths of \cite{benjamini1998unpredictable} to long-range random paths in Section \ref{sec:unpredictable} and use it to prove Theorem \ref{theo:main} and Theorem \ref{theo:isotropic} in dimension $d=3$ in Section \ref{sec:proofmain}, respectively, Section \ref{sec:proof isot}. A major input for the proof of Theorem \ref{theo:isotropic} in dimension $d=3$ is Proposition \ref{prop:unpredictable}, which shows that the future behavior of certain random walks with arbitrary distribution of the jump-size is hard to predict.\\

In Section \ref{sec:d=2}, we provide a short argument as to why the technique of Proposition \ref{prop:measure on paths} cannot work in dimension $d=2$. In Section \ref{sec:potts}, we describe and prove analogous statements to the ones of Theorem \ref{theo:main} and Theorem \ref{theo:isotropic} for the $q$-states Potts model.

\subsection{Open problems}

The results presented in this paper give rise to two natural open problems. The first one is, whether the results of Theorem \ref{theo:main} are also true in dimension $d=2$.

\begin{conjecture}
	Let $d = 2$, and assume that $\left(p_x\right)_{x\in \Z^d}$ satisfies conditions \eqref{eq:non-summable} through \eqref{eq:irreduc}. Then there exists $n\in \N$ such that
	\begin{equation*}
		\p \left(|K_\mz^n|=\infty\right) > 0. 
	\end{equation*} 
\end{conjecture}

As mentioned earlier, this conjecture fails in dimension $d=1$, so that dimension $d=2$ is the only dimension in which this question remains open. The second open problem is to show Theorem \ref{theo:isotropic} in dimensions $d\geq 2$ with a significantly reduced constant. For example, it would be interesting to know if the following conjecture is true.

\begin{conjecture}
	Let $d\geq 2$, and let $(p_x)_{x\in \Z^d\setminus \{\mz\}}$ be such that $\sum_{x \in \Z^d \setminus \{\mz\}} p_x > 10$, and such that \eqref{eq:irreduc} and \eqref{eq:isotropy} hold. Then
	\begin{equation*}
		\p \left( |K_\mz|=\infty \right) > 0.
	\end{equation*}
\end{conjecture}

\section{Dispersion of random paths}

In this section, we study the question of how much mass the $n$-th step of a random walk $(S_n)_{n\in \N_0}$ can have on a single point. In other words, we establish upper bounds for $\p(S_n=x)$ uniformly over $x\in \Z^d$ and for different types of random walks $(S_n)_{n\in \N_0}$. We start with the case where $S_n$ is the sum of independent random variables in Section \ref{sec:randvar}. The results shown in this section, in particularly Lemma \ref{lem1}, are a main input for the proof of Theorem \ref{theo:isotropic} in dimensions $d\geq 4$. Contrary to that, the proof of Theorem \ref{theo:isotropic} in dimension $d = 3$ needs different types of random walks, which are considered in Section \ref{sec:unpredictable}. The main new input that is needed to prove Theorem \ref{theo:isotropic} in dimension $d=3$ is Proposition \ref{prop:unpredictable}.

\subsection{Sums of random variables}\label{sec:randvar}

\begin{lemma}\label{lem1}
	Let $a_1,\ldots,a_n \in \N_{>0}$, and let $Y_1,\ldots,Y_n$ be independent and identically distributed random variables with $\p(Y_k=1) = \p(Y_k=-1) = \tfrac{1}{2}$. Then, for all $x\in \Z$,
	\begin{equation}\label{eq:sqrt n bound d=1}
		\p\left( \sum_{k=1}^{n} a_k Y_k = x \right) \leq \frac{1}{\sqrt{n}} .
	\end{equation}
	Furthermore, for $a_1,\ldots,a_n \in \N_0$ with $\max_i a_i >0$ one has
	\begin{align}\label{eq:one quarter bound}
		&\p\left( \sum_{k=1}^{n} a_k Y_k > 0 \right) \geq \frac{1}{4},\\
		& \label{eq:one quarter bound1} \p\left( \sum_{k=1}^{n} a_k Y_k \neq 0 \right) \geq \frac{1}{2}, \text{ and } \\
		& \label{eq:one quarter bound2} \p\left( \sum_{k=1}^{n} a_k Y_k \geq 0 \right) \geq \frac{1}{2} .
	\end{align}
\end{lemma}
\begin{proof}
	We start with the proof of \eqref{eq:sqrt n bound d=1}.
	For $k \in \{1,\ldots,n\}$, let $\phi_k(t)=\E\left[e^{i t a_k Y_k}\right] = \cos\left( a_k t \right)$ be the characteristic function of $a_k Y_k$, and let $\phi(t) = \prod_{k=1}^{n} \phi_k(t)$ be the characteristic function of $Y \coloneqq \sum_{k=1}^{n} a_k Y_k$. By the inversion theorem for the characteristic function \cite[Theorem 15.10]{klenke2013probability} we get for all $x\in \Z$ that
	\begin{align*}
		\p(Y=x) & = \frac{1}{2\pi} \int_{0}^{2\pi} e^{-itx} \phi(t) \md t 
		= \left| \frac{1}{2\pi} \int_{0}^{2\pi} e^{-itx} \prod_{k=1}^{n} \phi_k(t) \md t \right|
		\leq  \frac{1}{2\pi} \int_{0}^{2\pi} \prod_{k=1}^{n} \left| \phi_k(t)\right| \md t \\
		&
		\leq \frac{1}{2\pi} \prod_{k=1}^{n} \left(\int_{0}^{2\pi} \left| \phi_k(t)\right|^n \md t \right)^{1/n}
		=
		\frac{1}{2\pi} \prod_{k=1}^{n} \left(\int_{0}^{2\pi} \left| \cos(a_k t)\right|^n \md t \right)^{1/n},
	\end{align*}
	where we used Hölder's inequality for the last inequality. By the $2\pi$-periodicity of the cosine, the integrals of the form $\int_{0}^{2\pi} \left| \cos(a_k t)\right|^n \md t$ are the same for all $a_k \in \N_{>0}$. Thus we get that
	\begin{equation}\label{eq:indefinition}
		\p(Y=x) \leq 	\frac{1}{2\pi} \prod_{k=1}^{n} \left(\int_{0}^{2\pi} \left| \cos(a_k t)\right|^n \md t \right)^{1/n}
		= \frac{1}{2\pi} \int_{0}^{2\pi} \left| \cos( t)\right|^n \md t = \frac{1}{\pi} \int_{-\pi/2}^{\pi/2}  \cos( t)^n \md t \eqqcolon I_n .
	\end{equation}
	Using integration by parts, we see that for $n\geq 2$
	\begin{align*}
		I_n & = \frac{1}{\pi} \int_{-\pi/2}^{\pi/2} \cos(t)^{n-1} \cos(t) \md t
		= \frac{n-1}{\pi} \int_{-\pi/2}^{\pi/2} \cos(t)^{n-2} \sin(t)^2 \md t\\
		&
		= \frac{n-1}{\pi} \int_{-\pi/2}^{\pi/2} \cos(t)^{n-2} (1-\cos(t)^2) \md t
		= (n-1) I_{n-2} - (n-1) I_n
	\end{align*}
	and thus we see that the inductive relation $I_n = \frac{n-1}{n} I_{n-2}$ holds for $n\geq 2$. For $n\geq 2$ even, this implies by induction on $n$ even that $I_n = 2^{-n} \binom{n}{n/2}$. We will now show that $2^{-n} \binom{n}{n/2} \leq \tfrac{0.87}{\sqrt{n}}$ for $n$ even. For this, we use a version of Stirling's formula due to Robbins \cite{robbins1955remark}, which directly implies that
	\begin{align}\label{eq:stirling}
		\sqrt{2\pi k} \left(\frac{k}{e}\right)^k < k! < \sqrt{2\pi k} \left(\frac{k}{e}\right)^k e^{1/12} \text{ for all $k \in \N_{>0}$.}
	\end{align}
	Using this formula, we get for $n\geq 2$ even that
	\begin{align*}
		I_n & = 2^{-n} \binom{n}{n/2} = 2^{-n} \frac{n!}{(n/2)!(n/2)!} 
		\leq 2^{-n} \frac{\sqrt{2\pi n} \left(\frac{n}{e}\right)^n e^{1/12}}{\left(\sqrt{\pi n} \left(\frac{n/2}{e}\right)^{n/2}\right)^2}\\
		& = \frac{1}{\sqrt{n}} \sqrt{\frac{2}{\pi}} e^{1/12} \leq \frac{0.87}{\sqrt{n}}.
	\end{align*}
	For $n$ odd with $n\geq 5$ we have that
	\begin{equation*}
		I_n \leq I_{n-1} \leq \frac{0.87}{\sqrt{n-1}} \leq \frac{0.87}{\sqrt{0.8n}} \leq \frac{1}{\sqrt{n}}
	\end{equation*}
	and for $n\in \{1,3\}$ one also checks that $I_n \leq \frac{1}{\sqrt{n}}$. Thus we showed that $I_n \leq \frac{1}{\sqrt{n}}$ for all $n\in \N_{>0}$ and inserting this into \eqref{eq:indefinition} finishes the proof.\\
	
	For the proof of \eqref{eq:one quarter bound}, we can assume without loss of generality that $a_1>0$. 
	By symmetry, we get that
	\begin{align}\label{eq:4 things}
		\notag &\p\left( a_1 Y_1 + \sum_{k=2}^{n} a_k Y_k > 0 \right) 
		=
		\p\left( a_1 Y_1 - \sum_{k=2}^{n} a_k Y_k > 0 \right) \\
		= &
		\p\left( -a_1 Y_1 + \sum_{k=2}^{n} a_k Y_k > 0 \right) 
		=
		\p\left( -a_1 Y_1 - \sum_{k=2}^{n} a_k Y_k > 0 \right) .
	\end{align}
	As
	$a_1\neq 0$, at least one of the four expressions of the form $ \pm a_1 Y_1 \pm \sum_{k=2}^{n} a_k Y_k$ needs to be positive. Thus the sum of all probabilities in \eqref{eq:4 things} needs to be at least 1. As all probabilities in \eqref{eq:4 things} are equal, this implies that $\p\left( a_1 Y_1 + \sum_{k=2}^{n} a_k Y_k > 0 \right) \geq \frac{1}{4}$, which shows \eqref{eq:one quarter bound}. Inequality \eqref{eq:one quarter bound1} follows from \eqref{eq:one quarter bound} and \eqref{eq:4 things}, as
	\begin{equation*}
		\p\left( \sum_{k=1}^{n} a_k Y_k \neq 0 \right) = \p\left( \sum_{k=1}^{n} a_k Y_k > 0 \right)
		+
		\p\left( \sum_{k=1}^{n} a_k Y_k < 0 \right) \geq \frac{2}{4}.
	\end{equation*}
	Inequality \eqref{eq:one quarter bound2} holds by the symmetry around $0$ of the random variable $\sum_{k=1}^{n} a_k Y_k$.
\end{proof}

\subsection{Unpredictable paths}\label{sec:unpredictable}

So far, we only discussed sums of random variables. Fur sequences of random variables $S=(S_n)_{n \in \N_0}$, which are not the sums of independent random variables, one way to measure the dependence between $S_n$ for different values of $n\in \N_0$ is the so-called {\sl predictability profile} \cite{benjamini1998unpredictable}. Informally, the predictability profile quantifies how well $S_{n+k}$ can be predicted, upon observing $S_0,\ldots,S_n$. Formally, for any sequence of discrete random variables $S=(S_n)_{n\in \N_0}$, we define the predictability profile $(\pre_k(S))_{k\in \N_0}$ by
\begin{equation}\label{eq:predictability profile}
	\pre_k(S)=\sup \p\left(S_{n+k}=x | S_0,\ldots,S_n\right),
\end{equation}
where the supremum is over all $x\in \Z^d, n\in \N_0$, and all histories $S_0,\ldots,S_n$ with positive probability. When $S_n = \sum_{i=1}^n X_i$ with i.i.d. $\mathbb{Z}$-valued random variables $X_1,\ldots$ that are non-degenerate, then the predictability profile satisfies $\pre_k(S) = \mathcal{O}(k^{-\tfrac{1}{2}})$. Furthermore, $k^{-\tfrac{1}{2}}$ is the correct asymptotic when the random variables have a finite variance.
For the proof of Theorems \ref{theo:main} and \ref{theo:isotropic} in dimension $d=3$, we need stochastic processes whose predictability profile decays strictly faster than $k^{-\tfrac{1}{2}}$. 
So in this section, we are interested in stochastic processes $(S_n)_{n\in \N_0}$ which are \textbf{not} the sum of independent random variables and have a predictability profile that decays strictly faster than $k^{-\tfrac{1}{2}}$. This question, and its application to percolation was considered by Benjamini, Pemantle, and Peres \cite{benjamini1998unpredictable}. After that, there were also improvements on the existence of sequences of random variables with a given predictability profile by Häggström and Mossel \cite{haggstrom1998nearest} and Hoffmann \cite{hoffman1998unpredictable}. We quickly review the construction by Benjamini, Pemantle, and Peres \cite[Section 4]{benjamini1998unpredictable}. The results in \cite{benjamini1998unpredictable} encompass much more than what is presented in the following. In the notation of \cite{benjamini1998unpredictable}, we only consider the special case where $b=3, \ell = 2, r=1$.\\

\noindent
Let $\T_3$ be the infinite rooted ternary tree, i.e., the rooted tree where each vertex has exactly three children. We also assume that the children of each vertex are ordered, i.e., that there is a first, second, and third child. We say that the root is at level $0$ and write $|\text{root}|=0$. If $v \in \T_3$ is a vertex with child $w$, we define $|w|=|v|+1$ as the level of $w$. By construction, there are exactly $3^N$ many vertices at level $N$. 
For each vertex $v\in \T_3$, we now attach a label $\sigma(v)\in \{-1,+1\}$ (also called {\sl spin}) to this vertex as follows. We do this inductively over all levels. Start with $\sigma(\text{root})=+1$. For any vertex $\sigma \in \T_3$ with children $w_1,w_2,w_3$, assign the first two children $w_1,w_2$ of $\sigma$ the same spin as $v$, i.e., $\sigma(v)=\sigma(w_1)=\sigma(w_2)$. For the third child $w_3$ of $\sigma$, the spin $\sigma(w_3)$ is either $+1$ or $-1$ with equal probability, independent of $\sigma(v)$ and all other spins that have been assigned so far. Using this construction, one can inductively assign spins to all levels. Let $Z_N^+$ be the number of vertices at level $N$ with spin $+1$, and let $Z_N^-$ be the number of vertices at level $N$ with spin $-1$. For $Y_N$, the total sum of spins at level $N$, defined by 
\begin{equation*}
	Y_N = \sum_{v \in \T_3 : |v|=N} \sigma(v) = Z_N^+ - Z_N^-,
\end{equation*}
it is proven in \cite[Lemma 4.1]{benjamini1998unpredictable} that
\begin{equation}\label{eq:benjamini1}
	\sup_{x\in \R} \p(Y_N=x) \leq C 2^{-N}
\end{equation}
with
\begin{align*}
	C &= 4 \left(\frac{\pi}{2}+ \sum_{k=1}^{N} \frac{2^k \pi}{2} \cos(\tfrac{\pi}{4})^{3^{k-1}} \right)
	\leq 2 \left(\pi+  2 \pi \sum_{k=1}^{\infty} 2^{k-1}  \cos(\tfrac{\pi}{4})^{3^{k-1}} \right) \leq 2 \left(\pi+  4 \pi  \right) \leq 32,
\end{align*}
where the second to last inequality can be checked using the rapid decay of $\cos(\tfrac{\pi}{4})^{3^{k-1}}$. (The true value of the sum $\sum_{k=1}^{\infty} 2^{k-1}  \cos(\tfrac{\pi}{4})^{3^{k-1}}$ is approximately $1.59\ldots$). See \cite[Equation (21)]{benjamini1998unpredictable} for the definition of the constant $C$.\\

Let $v_1,\ldots,v_{3^N}$ be the vertices in the tree at level $N$, in lexicographic ordering. Summing the spins of these vertices at level $N$ and defining $S_m = \sum_{i=1}^{m} \sigma(v_i)$ one gets a finite sequence of random variables. The distribution of $S_m$ does not depend on the choice of $N$, so using Kolmogorov's consistency theorem, one can extend this finite sequence to an infinite sequence $S=(S_m)_{m\in \N_0}$. The infinite sequence satisfies
\begin{equation}\label{eq:predict 100 old}
	\pre_k(S) \leq 6^{\frac{\log(2)}{\log(3)}} 32 k^{-\frac{\log(2)}{\log(3)}} 
	\leq 100 k^{-\frac{\log(2)}{\log(3)}},
\end{equation}
see \cite[Equation (23)]{benjamini1998unpredictable}. The important thing to notice here is that $\frac{\log(2)}{\log(3)} > \frac{1}{2}$, so the predictability profile of $S$ decays much faster than that of a simple random walk. Throughout this paper, we use the convention that $\log$ is the logarithm with base $e$. Furthermore, the process $S=(S_m)_{m\in \N_0}$ satisfies $|S_{m+1}-S_m|=1$ for all $m\in \N_0$, so the step size of the process is $1$. \\

Next, we use the results of \cite{benjamini1998unpredictable} to construct stochastic processes with more general increments.
In Proposition \ref{prop:unpredictable}, we construct a stochastic process $\tilde{S}=(\tilde{S}_m)_{m\in \N_0}$ with a rapidly decaying predictability profile, where the absolute values of the increments are not identical. This proposition will be useful to prove Theorem \ref{theo:isotropic} in dimension $d=3$. It allows us to adapt the theory of unpredictable paths to paths with a more general step-size distribution.
Before going to the statements, we introduce some notation. We write $\left[a_1,\ldots,a_n\right]$ for the \textbf{multiset} consisting of $a_1,\ldots,a_n$. Recall that a collection of random variables $(X_n)_{n\in \left\{1,\ldots,N\right\}}$ is \textbf{exchangeable} if for every permutation $\eta : \{1,\ldots,N\} \to \{1,\ldots,N\}$ the sequence of random variables $(X_1,\ldots,X_N)$ has the same distribution as $(X_{\eta(1)},\ldots, X_{\eta(N)})$.

\begin{proposition}\label{prop:unpredictable}
	Let $S=(S_m)_{m\in \N_0}$ be a process satisfying \eqref{eq:predict 100 old} such that $S_0=0$ and $|S_{m+1}-S_m| = 1$ for $m\in \N_0$. Write $S_m=\sum_{i=1}^{m} \sigma_i$. Let $a_{1},\ldots,a_{N+k}>0$, and let $X_{1},\ldots,X_{N+k}$ be exchangeable positive random variables that are furthermore independent of $(S_m)_{m\in \N_0}$ and satisfy $\left[X_{1},\ldots,X_{N+k}\right]=\left[a_{1},\ldots,a_{N+k}\right]$. Then the random variables $(\tilde{S}_{m})_{m\in \{0,\ldots,N+k\}}$ defined by
	\begin{equation}\label{eq:tilde S defini}
		\tilde{S}_{m} = \sum_{i=1}^{m} \sigma_i X_i  \ \ \text{ for } \ m=0,\ldots,N+k
	\end{equation}
	satisfy
	\begin{align}
		&\label{eq:predict 100 new} 
		\sup_{z\in \R} \p\left(\tilde{S}_{N+k}=z| \tilde{S}_0,\ldots,\tilde{S}_N\right)
		\leq 100 k^{-\frac{\log(2)}{\log(3)}} .
	\end{align}
\end{proposition}

\begin{proof}
	For $m,j\in \N_0$, let $Z_{m,m+j}^+$ be the number of spins $\sigma_i$ with $\sigma_i=+1$ and $i\in \{m+1,\ldots,m+j\}$. Analogously, let $Z_{m,m+j}^-$ be the number of spins $\sigma_i$ with $\sigma_i=-1$ and $i\in \{m+1,\ldots,m+j\}$. So in particular we have that 
	\begin{equation}\label{eq:Z difference equation}
		Z_{m,m+j}^+ +  Z_{m,m+j}^- = j \text{ and } S_{m+j} = S_m + Z_{m,m+j}^+ - Z_{m,m+j}^- 
		=
		S_m + 2 Z_{m,m+j}^+ - j .
	\end{equation}
	Conditioned on $\tilde{S}_0,\ldots,\tilde{S}_N$, one can directly deduce $X_1,\ldots,X_N$, so that 
	\begin{equation*}
		\left[X_{N+1},\ldots,X_{N+k}\right] = \left[a_{1},\ldots,a_{N+k}\right] \setminus
		\left[X_{1},\ldots,X_{N}\right] \eqqcolon \left[b_{N+1},\ldots,b_{N+k}\right]
	\end{equation*}
	As $X_{1},\ldots,X_{N+k}$ were assumed to be exchangeable, this directly implies that conditioned on $X_{1},\ldots,X_{N}$, the random variables $X_{N+1},\ldots,X_{N+k}$ are still exchangeable.
	Using that the collection of random variables $X_{N+1},\ldots,X_{N+k}$ is exchangeable, we see that for any $\tau_{N+1},\ldots,\tau_{N+k} \in \{-1,+1\}$ the distribution of $\sum_{i=N+1}^{N+k} \tau_i X_i$ just depends on $|\left\{i \in \{N+1,\ldots, N+k\} : \tau_i=+1\right\}|$, but not on the order of $\tau_{N+1},\ldots,\tau_{N+k}$. As the collection $X_{N+1},\ldots,X_{N+k}$ is furthermore
	independent of $\sigma_{N+1},\ldots, \sigma_{N+k}$, this implies that
	\begin{align*}
		& \p\left(\tilde{S}_{N+k}= z | \tilde{S}_0,\ldots, \tilde{S}_N\right)
		=
		\p\left(\sum_{i=N+1}^{N+k} \sigma_i X_i = z -\tilde{S}_N \big| \tilde{S}_0,\ldots, \tilde{S}_N\right)\\
		&
		=
		\p\Bigg(\sum_{i=N+1}^{N+Z_{N,N+k}^+} X_i - \sum_{i=N+Z_{N,N+k}^+ + 1}^{N+k} X_i = z -\tilde{S}_N \ \big| \ \tilde{S}_0,\ldots, \tilde{S}_N\Bigg)
	\end{align*}
	for all $z\in \R$.
	Conditioning not just on $\tilde{S}_0,\ldots,\tilde{S}_N$, but also on $X_{N+1},\ldots,X_{N+k}$, we see that
	\begin{align}
		& \notag \sup_{z, \tilde{S}_0,\ldots, \tilde{S}_N} \p\left(\tilde{S}_{N+k}= z | \tilde{S}_0,\ldots, \tilde{S}_N\right) \\
		& \notag
		= 
		\sup_{z, \tilde{S}_0,\ldots, \tilde{S}_N}
		\p\Bigg(\sum_{i=N+1}^{N+Z_{N,N+k}^+} X_i - \sum_{i=N+Z_{N,N+k}^+ + 1}^{N+k} X_i = z  \ \big| \ \tilde{S}_0,\ldots, \tilde{S}_N\Bigg)\\
		& \label{eq:sup einsetzen}
		\leq
		\sup_{\substack{z, \tilde{S}_0,\ldots, \tilde{S}_N, \\ X_{N+1},\ldots, X_{N+k}}}
		\p\Bigg(\sum_{i=N+1}^{N+Z_{N,N+k}^+} X_i - \sum_{i=N+Z_{N,N+k}^+ + 1}^{N+k} X_i = z  \ \big| \ \tilde{S}_0,\ldots, \tilde{S}_N, X_{N+1},\ldots,X_{N+k}\Bigg)
	\end{align}
	where the last suprema are over all $z\in \R$, histories $\tilde{S}_0,\ldots,\tilde{S}_N$, and - in the last line - all realizations of $X_{N+1},\ldots,X_{N+k}$ with  $\left[X_{N+1},\ldots,X_{N+k}\right] = \left[b_{N+1},\ldots,b_{N+k}\right]$ consisting of positive numbers. For each fixed $z \in \R$, history $\tilde{S}_0,\ldots,\tilde{S}_N$, and $X_{N+1},\ldots, X_{N+k}>0$, there is at most one value $M = M(z,X_{N+1},\ldots,X_{N+k}) \in \{0,\ldots,k\}$ such that
	\begin{equation}\label{eq:N defini}
		\sum_{i=N+1}^{N+M} X_i - \sum_{i=N+ M + 1}^{N+k} X_i = z.
	\end{equation}
	This holds as $X_{N+1},\ldots, X_{N+k}>0$. Let $M= M(z,X_{N+1},\ldots,X_{N+k}) \in \{0,\ldots,k\}$ be the unique value satisfying \eqref{eq:N defini} if it exists, and let $M=-1$ if there does not exist $M \in \{0,\ldots,k\}$ satisfying \eqref{eq:N defini}. Thus we see that
	\begin{align*}
		& \p\Bigg(\sum_{i=N+1}^{N+Z_{N,N+k}^+} X_i - \sum_{i=N+Z_{N,N+k}^+ + 1}^{N+k} X_i = z  \ \big| \ \tilde{S}_0,\ldots, \tilde{S}_N, X_{N+1},\ldots,X_{N+k}\Bigg)\\
		&=
		\p\left( Z_{N,N+k}^+ = M(z,X_{N+1},\ldots,X_{N+k})  \ \big| \ \tilde{S}_0,\ldots, \tilde{S}_N, X_{N+1},\ldots,X_{N+k}\right)\\
		&
		\overset{\eqref{eq:Z difference equation}}{=}
		\p\left( S_{N+k} = S_N + 2M(z,X_{N+1},\ldots,X_{N+k}) - k  \ \big| \ \tilde{S}_0,\ldots, \tilde{S}_N, X_{N+1},\ldots,X_{N+k}\right)\\
		&
		= \p\left( S_{N+k} = S_N + 2M(z,X_{N+1},\ldots,X_{N+k}) - k  \ \big| \ S_0,\ldots, S_N, X_{N+1},\ldots,X_{N+k}\right)\\
		&
		\leq \pre_k(S) \overset{\eqref{eq:predict 100 old}}{\leq} 100 k^{-\frac{\log(2)}{\log(3)}}.
	\end{align*}
	In the last two lines, we used that $S_{N+k}-S_N$ is independent of $X_1,\ldots,X_{N+k}$. As the calculation above holds for all $z\in \R$, histories $\tilde{S}_0,\ldots,\tilde{S}_N$, and $X_{N+1},\ldots,X_{N+k}>0$, we can insert this inequality into \eqref{eq:sup einsetzen} and get that
	\begin{align*}
		\sup_{z, \tilde{S}_0,\ldots, \tilde{S}_N} \p\left(\tilde{S}_{N+k}= z | \tilde{S}_0,\ldots, \tilde{S}_N\right) \leq 100 k^{-\frac{\log(2)}{\log(3)}}.
	\end{align*}
\end{proof}

\section{Proofs of the Theorems}\label{sec:proofs}

In this section, we prove Theorems \ref{theo:main} and \ref{theo:isotropic}. We start with the proofs of the second statements (Equations \eqref{eq:main high prob} and \eqref{eq:isotrop finite cluster low prob}) of these theorems in Section \ref{sec:3.1}. In Section \ref{sec:setup}, we describe the general setup of the proofs of Theorems \ref{theo:main} and \ref{theo:isotropic}. We then proceed to prove Theorem \ref{theo:main} in Section \ref{sec:proofmain} and Theorem \ref{theo:isotropic} in Sections \ref{sec:proof isot} and \ref{sec:boots}. Finally, in Section \ref{sec:d=2}, we show that the setup outlined in Section \ref{sec:setup} can not work in dimension $d=2$.

\subsection{Proof of \eqref{eq:main high prob} and \eqref{eq:isotrop finite cluster low prob}}\label{sec:3.1}

In this section, we prove that the percolation density converges to $1$ when the truncation length $n$ is taken to $+\infty$ in Theorem \ref{theo:main}, respectively when the expected degree $\E\left[\deg(\mz)\right]$ diverges to $+\infty$ in Theorem \ref{theo:isotropic}.
We start with the proof of \eqref{eq:main high prob}, assuming \eqref{eq:main existence}. That is, we need to show that
\begin{equation*}
	\lim_{n \to \infty} \p \left( |K_\mz^n|=\infty \right) = 1,
\end{equation*}
assuming that $\p \left( |K_\mz^N|=\infty \right)>0$ for some $N$ large enough.

\begin{proof}[Proof of \eqref{eq:main high prob} assuming \eqref{eq:main existence}]
	Let $N$ be large enough so that $\p \left( |K_\mz^N|=\infty \right)>0$. Write $\cC_\infty^{\leq N} = \left\{x \in \Z^d : x \overset{\leq N}{\longleftrightarrow} \infty \right\}$ for the infinite cluster in the environment where we truncated all edges of size larger than $N$. Note that this infinite cluster is almost surely unique, which follows from the amenability of $\Z^d$ and the uniqueness of the infinite cluster as proven by Burton and Keane \cite{burton1989density}. Furthermore, one has that $\mz \leftrightarrow x$ for all $x\in \Z^d$ almost surely, which was proven by Grimmett, Keane and Marstrand  \cite[Theorem 3]{grimmett1984connectedness}. Thus we directly get that $\mz \leftrightarrow \cC_\infty^{\leq N}$ almost surely. In particular, there exists a random, but almost surely finite, $K \in \N$ such that $\mz \overset{\leq K}{\longleftrightarrow} \cC_\infty^{\leq N}$. As the events of the form $\left\{ \mz \overset{\leq k}{\longleftrightarrow} \cC_\infty^{\leq N} \right\}$ are increasing in $k$, we get that for $n \geq N$
	\begin{equation*}
		\lim_{n\to \infty}\p \left( |K_\mz^n|=\infty \right) \geq \lim_{k \to \infty} \p \left( \mz \overset{\leq k}{\longleftrightarrow} \cC_\infty^{\leq N}  \right)
		=
		\p \left( \mz \leftrightarrow \cC_\infty^{\leq N}  \right) = 1,
	\end{equation*}
	which shows \eqref{eq:main high prob}.
\end{proof}

Next, we turn to the proof of \eqref{eq:isotrop finite cluster low prob}. For this, we first need to introduce some notation. We identify the percolation environment with an element $\omega \in \{0,1\}^E$, where we set $\omega(e)=1$ if the edge $e$ is open and $\omega(e)=0$ if the edge is closed. We write $K_\mz(\omega)$ for the open cluster containing the origin in the percolation configuration defined by $\omega$.

\begin{proof}[Proof of \eqref{eq:isotrop finite cluster low prob} assuming \eqref{eq:isotrop infinite cluster pos prob}]
	Let $\omega \in \{0,1\}^E$ be distributed according to a product measure with $\p\left(\omega(\{x,y\})=1\right)=p_{x-y}$. Assume that $\sum_{x\in \Z^d\setminus \{\mz\}}p_x > N T(d)$ for some $N\in \N_{>0}$. Let $\omega_1,\ldots,\omega_N\in \{0,1\}^E$ be independent and identically distributed percolation environments such that the components $(\omega_i(e))_{e\in E, i=1,\ldots,N}$ are independent, and such that $\p\left(\omega_i(\{x,y\})=1\right)=\frac{p_{x-y}}{N}$. The collection of probabilities $\left(\tfrac{p_x}{N}\right)_{x\in \Z^d \setminus \{\mz\}}$ satisfies all the assumptions of Theorem \ref{theo:isotropic}.
	So in particular, by \eqref{eq:isotrop infinite cluster pos prob}, we have that
	\begin{equation*}
		\p \left(|K_\mz(\omega_i)|< \infty \right) \leq \frac{1}{e} \text{ for $i=1,\ldots,N$.}
	\end{equation*}
	Next, we compare the environment $\omega$ to the environment $\max_i \omega_i \in \{0,1\}^E$ defined by
	\begin{equation*}
		(\max_i \omega_i)(e) \coloneqq \max_i \omega_i(e).
	\end{equation*}
	The claim is that $\omega$ stochastically dominates $ \max_i \omega_i$. As the different components of $\omega$, respectively $\max_i\omega_i$ are independent, it suffices to compare the marginals of the two random variables. By a union bound, we have that
	\begin{equation*}
		\p\left(\max_i\omega_i(\{x,y\}) = 1\right) \leq \sum_{i=1}^{N} \p\left(\omega_i(\{x,y\}) = 1\right) =p_{x-y} = \p \left(\omega(\{x,y\})=1\right)
	\end{equation*}
	for all $\{x,y\}\in E$,
	which implies that $\omega$ dominates $\max_i \omega_i$. This directly implies that
	\begin{equation*}
		\p \left(|K_\mz(\omega)|< \infty \right) 
		\leq 
		\p \left(|K_\mz(\max_i \omega_i)|< \infty \right).
	\end{equation*}
	If $x \in K_{\mz}(\omega_j)$ for some $j\in \{1,\ldots,N\}$, then also $x \in K_{\mz}(\max_i\omega_i)$, as $\omega_j \leq \max_i\omega_i$. As $j \in \{1,\ldots,N\}$ was arbitrary, this implies that $K_\mz(\max_i \omega_i) \supseteq \bigcup_i K_\mz( \omega_i)$, from which we further get that
	\begin{align*}
		&\p \left(|K_\mz(\omega)|< \infty \right) 
		\leq 
		\p \left(|K_\mz(\max_i \omega_i)|< \infty \right)
		\leq 
		\p \left(|K_\mz(\omega_i)|< \infty \text{ for $i=1,\ldots,N$} \right) 
		\\
		&
		=
		\p \left(|K_\mz(\omega_1)|< \infty \right)^N \leq e^{-N}.
	\end{align*}
	Thus we see that the condition $\sum_{x\in \Z^d\setminus \{\mz\}}p_x > N T(d)$ for some $N\in \N_{0}$ implies that $\p \left(|K_\mz|< \infty \right) \leq \exp(-N)$, which proves \eqref{eq:isotrop finite cluster low prob}, as
	\begin{align*}
		\p \left(|K_\mz(\omega)|< \infty \right) \leq \exp \left(-\Bigg\lceil \frac{\sum_{x\in \Z^d\setminus \{\mz\}}p_x}{T(d)} -1 \Bigg\rceil\right)
		\leq \exp \left(1- \frac{\sum_{x\in \Z^d\setminus \{\mz\}}p_x}{T(d)} \right).
	\end{align*}
\end{proof}

\subsection{The general setup}\label{sec:setup}

In this section, we describe a general setup that we use for the proofs of Theorem \ref{theo:main} and Theorem \ref{theo:isotropic}. Before going to the statements, we introduce some notation. Throughout this paper, we say that a path is a map $\gamma: I \to \Z^d$, where $I = \left[a,b\right]\cap \Z$ for some $a,b \in \R \cup \{-\infty, +\infty\}$. For a path $\gamma : I \to \Z^d$ and $x\in \Z^d$, we define the path $x+\gamma$ by $(x+\gamma)(k) = x + \gamma(k)$ for all $k\in I$. We say that the path $\gamma$ is open if the edges of the form $\{\gamma_k, \gamma_{k+1}\}$ are open for all $k\in I$ for which $\gamma_k \neq \gamma_{k+1}$ and $k+1\in I$. 
For a path $\gamma=(\gamma_k)_{k\in I}$ and $j, n \in I$ with $j\leq n$, we write $\gamma_j^n=(\gamma_j,\ldots,\gamma_n)$ for this segment of the path; If $j>n$, we define $\gamma_j^n=\emptyset$.

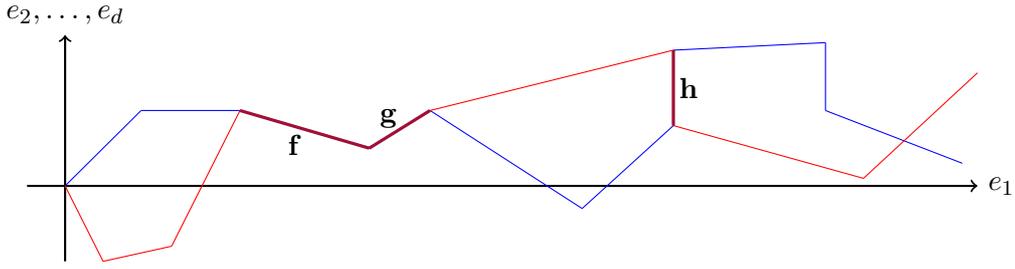
\begin{figure}[h]		
	\begin{center}
		\begin{tikzpicture}
			\draw[->, thick] (-0.5,0) -- (12,0) node[right] {$e_1$};
			\draw[->, thick] (0,-1) -- (0,2) node[above] {$e_2,\ldots,e_d$};
			
			\path[blue] 
			(0,0) edge (1,1) 
			(1,1) edge (2.3,1) 
			(2.3,1) edge (4,0.5) 
			(4,0.5) edge (4.8,1)
			(4.8,1) edge (6.8,-0.3)
			(6.8,-0.3) edge (8,0.8)
			(8,0.8) edge (8,1.8)
			(8,1.8) edge (10,1.9)
			(10,1.9) edge (10,1)
			(10,1) edge (11.8,0.3);
			
			\path[ red] 
			(0,0) edge (0.5,-1) 
			(0.5,-1) edge (1.4,-0.8) 
			(1.4,-0.8) edge (2.3,1) 
			(2.3,1) edge (4,0.5) 
			(4,0.5) edge (4.8, 1)
			(4.8,1) edge (8, 1.8)
			(8,1.8) edge (8,0.8)
			(8,0.8) edge (10.5,0.1)
			(10.5,0.1) edge (12,1.5);
			
			\path[very thick, truepurp] 
			(2.3,1) edge (4,0.5) 
			(4,0.5) edge (4.8, 1)
			(8,0.8) edge (8,1.8);
			
			\vertex[ draw=none ] at (3,0.55) {\bf f};
			\vertex[ draw=none ] at (4.25,0.9) {\bf g};
			\vertex[ draw=none ] at (8.2,1.3) {\bf h};
			
		\end{tikzpicture} 
		
		\vspace{3mm}
		
		\parbox{13cm}{\caption{Two self-avoiding paths starting at the origin. The edges which are traversed by both paths (\textbf{f}, \textbf{g}, and \textbf{h}) are drawn in purple. Both paths traverse the edges \textbf{f} and \textbf{g} in the same direction. Contrary to that, the paths traverse the edge \textbf{h} in opposite directions.}	\label{fig:paths}}
		
	\end{center}
\end{figure}

We write $\cP^n$ for the set of  self-avoiding paths of length $n$ ($n\in \N_0 \cup \{+\infty\}$) in $\Z^d$ starting at the origin $\mz$. We consider probability measures on $\cP^\infty$. For two such probability measures $\tilde{\mu},\mu$, we write $\bp_{\tilde{\mu}\times \mu}$ for their product measure and $\be_{\tilde{\mu}\times \mu}$ for the expectation with respect to the product measure. We will often consider two random paths $\left(\left(\gamma_k\right)_{k\in \N_0},\left(\varphi_k\right)_{k\in \N_0}\right)$ sampled from the measure $\bp_{\tilde{\mu}\times \mu}$, with the convention that $\left(\gamma_k\right)_{k\in \N_0}$ has distribution $\tilde{\mu}$ and $\left(\varphi_k\right)_{k\in \N_0}$ has distribution $\mu$.

For two paths $\gamma=(\gamma_k)_{k\in I}$, $\varphi=(\varphi_n)_{n\in J}$, and an edge $e=\{x,y\}$ ($x,y \in \Z^d, x \neq y$), we say that $e \in \gamma \cap \varphi$ if there exist $k \in I, n\in J$ such that $k+1 \in I, n+1 \in J$, and $\{x,y\}=\{\varphi_n,\varphi_{n+1}\}=\{\gamma_k,\gamma_{k+1}\}$. See Figure \ref{fig:paths} for an illustration. Note that we do not require that the two paths traverse the edge in the same direction or that the paths are self-avoiding. If there does not exist $e=\{x,y\}$ with $x,y \in \Z^d, x\neq y$, and $e\in \gamma \cap \varphi$, we define that $\gamma \cap \varphi = \emptyset$; note that this holds even if there exist $x\in \Z^d, k\in I, n\in J$ with $x=\gamma_k=\varphi_n$. If different edges can be open or closed, as is the case for percolation, for two (finite or infinite) paths $\gamma$ and $\varphi$, we define their (weighted) \textbf{overlap} by
\begin{equation*}
	\ov(\gamma,\varphi) \coloneqq \prod_{e\in \gamma \cap \varphi } \p \left(e \text{ open}\right)^{-1}.
\end{equation*}
If $\gamma \cap \varphi=\emptyset$, the overlap of the two paths is defined by $\ov(\gamma,\varphi)=1$. If $\p(e \text{ open})=0$ for some $e\in \gamma \cap \varphi$, then we define $\ov(\gamma,\varphi)=\infty$. However, this will only occur with probability $0$ in our construction. Also, note that since all factors $\p(e \text{ open})^{-1}$ are at least 1, the weighted overlap is well-defined in $\R_{\geq 0} \cup \{+\infty\}$, even if $|\gamma\cap \varphi|=\infty$ and the product is an infinite product.
Further, the overlap depends only on the distribution of the percolation configuration, but not on the actual realization of the percolation configuration. If one of the paths, say $\gamma$, is not self-avoiding and traverses an edge $e \in \gamma \cap \varphi$ twice, note that we do not count the edge twice in the overlap.\\

The next proposition is the essential step that makes the connection between measures on paths and the existence of infinite open percolation clusters. It follows from the Cauchy-Schwarz (respectively Paley-Zygmund) inequality. Similar versions of this proposition for short-range percolation were also used in \cite{cox1983oriented,benjamini1998unpredictable}. For the models of short-range percolation considered in \cite{cox1983oriented, benjamini1998unpredictable}, it suffices to consider the moment-generating function of $|\gamma \cap \varphi|$, as all edges are open with the same probability. Contrary to that, for long-range percolation, different edges $e \in \gamma\cap \varphi$ contribute a different multiplicative factor to the weighted overlap defined above.

\begin{proposition}\label{prop:measure on paths}
	Let $\mu$ be a probability measure on self-avoiding paths in $\Z^d$. Assume that two random paths $\gamma=(\gamma_k)_{k\in \N_0}$ and $\varphi=(\varphi_k)_{k\in \N_0}$ are independent and distributed according to $\mu$. Assume that edges $e=\{x,y\}$ of $\Z^d$ are either open or closed and that all edges are independent of each other.
	Then
	\begin{equation*}
		\p \left( |K_\mz|=\infty \right) \geq\left( \be_{\mu \times \mu} \left[ \ov(\gamma,\varphi) \right]\right)^{-1}.
	\end{equation*}
\end{proposition}

\begin{proof}
	Consider the random variable $Z_n$ defined by
	\begin{align*}
		Z_n & = \int_{\cP^\infty}  \frac{\mathbbm{1}\{\gamma_0^n \text{ open}\}}{\p(\gamma_0^n \text{ open})} \md \mu (\gamma) .
	\end{align*}
	Taking expectations and applying Fubini gives that $\E\left[Z_n\right]=1$, as $\mu$ is a probability measure. For the square of $Z_n$, we have that
	\begin{align*}
		Z_n^2 & = \int_{\cP^\infty} \int_{\cP^\infty}  \frac{\mathbbm{1}\{\gamma_0^n \text{ open}\}}{\p(\gamma_0^n \text{ open})}
		\frac{\mathbbm{1}\{\varphi_0^n \text{ open}\}}{\p(\varphi_0^n \text{ open})}
		\md \mu (\gamma) \md \mu(\varphi) .
	\end{align*}
	The events $\{\gamma_0^n \text{ open}\}$ and $\{\varphi_0^n \text{ open}\}$ are not independent, as the two paths might have joint edges. Compensating for the edges that are in both paths, we compute the second moment of $Z_n$ to be
	\begin{align}\label{eq:second moment}
		\notag \E\left[Z_n^2\right] & = \int_{\cP^\infty} \int_{\cP^\infty}  \frac{\p(\gamma_0^n \text{ open}, \varphi_0^n \text{ open})}{\p(\gamma_0^n \text{ open}) \p(\varphi_0^n \text{ open})}
		\md \mu (\gamma) \md \mu(\varphi) \\
		& 
		= \int_{\cP^\infty} \int_{\cP^\infty} \prod_{e\in \gamma_0^n \cap \varphi_0^n} \p(e \text{ open})^{-1} \md \mu (\gamma) \md \mu(\varphi) = \be_{\mu \times \mu} \left[\ov(\gamma_0^n, \varphi_0^n)\right].
	\end{align}
	Write $K_{\mz}$ for the open cluster containing the origin.
	As the paths in $\cP^\infty$ are self-avoiding and start at $\mz$, the condition $Z_n > 0$ implies that $|K_{\mz}| \geq n$. As $Z_n$ is a non-negative random variable, the Paley-Zygmund-inequality implies that
	\begin{align}\label{eq:Paley Zygmund}
		\p \left(|K_\mz| \geq n\right) \geq \p \left(Z_n>0\right) \geq \frac{\E\left[Z_n\right]^2}{\E\left[Z_n^2\right]} = \frac{1}{\be_{\mu \times \mu} \left[\ov(\gamma_0^n, \varphi_0^n)\right]} .
	\end{align}
	Taking the infimum over all $n\in \N$, we get that
	\begin{align*}
		& \p \left(|K_\mz|=\infty\right) 
		= \inf_{n\in \N} \p \left(|K_\mz|>n\right) \geq \inf_{n\in \N} \p \left(Z_n>0\right) \geq \inf_{n\in \N} \left( \left(\be_{\mu \times \mu} \left[\ov(\gamma_0^n, \varphi_0^n)\right]\right)^{-1}\right) \\
		& =   \left( \sup_{n\in \N} \be_{\mu \times \mu} \left[\ov(\gamma_0^n, \varphi_0^n)\right]\right)^{-1}
		= \Big( \be_{\mu \times \mu} \left[\ov(\gamma, \varphi)\right]\Big)^{-1}.
	\end{align*}
	The last equality follows by monotone convergence since $\ov(\gamma_0^n, \varphi_0^n)$ is non-decreasing (in $n$) and $\lim_{n \to \infty} \ov(\gamma_0^n, \varphi_0^n) = \ov(\gamma, \varphi)$.
\end{proof}

\subsection{Proof of Theorem \ref{theo:main}}\label{sec:proofmain}

We start with the proof of Theorem \ref{theo:main}. We prove this theorem simultaneously for dimension $d=3$ and dimensions $d\geq 4$, using the concept of unpredictable paths introduced in Section \ref{sec:unpredictable}. For dimensions $d\geq 4$, one can also give a different proof of Theorem \ref{theo:main} that does not use unpredictable paths as described in Section \ref{sec:unpredictable} but uses ``normal" random walks with independent increments instead. We do not follow such a separate approach for dimensions $d\geq 4$ here. \\

\noindent
From condition \eqref{eq:non-summable} it directly follows that
\begin{align*}
	\infty = \sum_{x \in \Z^d \setminus \{\mz\}} p_x 
	= \sum_{n=1}^{\infty} \ \sum_{x: \|x\|_\infty = n} p_x
	\leq \sum_{i=1}^d \sum_{n=1}^{\infty} \ \sum_{\substack{x: \|x\|_\infty = n, \\ |x_i|=n}} p_x.
\end{align*}
Thus there needs to exist $i\in \{1,\ldots,d\}$ such that
\begin{equation}\label{eq:direction 1 infty}
	\sum_{n=1}^{\infty} \
	\sum_{\substack{x: \|x\|_\infty = n, \\ |x_i|=n}} p_x = \infty.
\end{equation}
Without loss of generality, we can assume that $i=1$. Let $u=(u_1,\ldots,u_d)$ and $v = (v_1,\ldots,v_d) \in \Z^d$ be two distinct vectors with $v_1, u_1 \geq 0$, such that $\eta \coloneqq \min(p_u,p_v ) >0$, and such that the vectors $(u_2,\ldots,u_d),(v_2,\ldots,v_d)\in \Z^{d-1}$ are linearly independent. The existence of such vectors is proved in the next claim and essentially follows from symmetry \eqref{eq:symmetry} and irreducibility \eqref{eq:irreduc}. 

\begin{claim}\label{claim:existence}
	Suppose that $(p_x)_{x\in \Z^d}$ satisfy conditions \eqref{eq:symmetry} and \eqref{eq:irreduc}. Then there exist vectors $u=(u_1,\ldots,u_d)$ and $v = (v_1,\ldots,v_d) \in \Z^d$ with $v_1, u_1 \geq 0$, such that $\eta \coloneqq \min(p_u,p_v ) >0$, and such that the vectors $(u_2,\ldots,u_d),(v_2,\ldots,v_d)\in \Z^{d-1}$ are linearly independent.
\end{claim}
\begin{proof}
	Define the set $A^+ = \left\{x \in \Z^d \setminus \{\mz\} : p_x > 0\right\}$.
	By irreducibility, for each standard unit vector $e_i \in \Z^d$, there exist $\alpha_1,\ldots,\alpha_n \in \N$ and $x_1,\ldots,x_n \in A^+$ such that $\alpha_1 x_1 + \ldots + \alpha_n x_n = e_1$. From this we can easily see that
	\begin{equation*}
		\text{span}(A^+) = \left\{\sum_{i=1}^{n} \lambda_i x_i : \lambda_i  \in \R, x_i \in A^+ , n \in \N\right\} = \R^d .
	\end{equation*}
	So in particular $\text{dim}\left(\text{span}(A^+)\right)=d$. Define the linear function $H:\R^d \to \R^{d-1}$ by
	\begin{equation*}
		H\left((u_1,\ldots,u_d)\right) = (u_2,\ldots,u_d).
	\end{equation*}
	Then the linearity directly implies that
	\begin{align*}
		\R^{d-1} = H\left(\R^d\right) = H \left(\text{span}(A^+)\right) = \left\{\sum_{i=1}^{n} \lambda_i H(x_i) : \lambda_i  \in \R, x_i \in A^+ , n \in \N\right\} .
	\end{align*}
	Since $d-1 \geq 2$, there need to exist at least two vectors $\tilde{u}=(\tilde{u}_1,\ldots,\tilde{u}_d),\tilde{v}=(\tilde{v}_1,\ldots,\tilde{v}_d) \in A^+$ such that $H(\tilde{u}) = (\tilde{u}_2,\ldots,\tilde{u}_d)$ and $H(\tilde{v}) = (\tilde{v}_2,\ldots,\tilde{v}_d)$ are linearly independent. As we assumed that the symmetry described in \eqref{eq:symmetry} holds, we get that
	\begin{equation*}
		p_{(\tilde{u}_1,\tilde{u}_2,\ldots,\tilde{u}_d)} = p_{(-\tilde{u}_1,\tilde{u}_2,\ldots,\tilde{u}_d)} \quad \text{ and } \quad p_{(\tilde{v}_1,\tilde{v}_2,\ldots,\tilde{v}_d)} = p_{(-\tilde{v}_1,\tilde{v}_2,\ldots,\tilde{v}_d)}
	\end{equation*}
	and thus we can also find vectors $u=(u_1,\ldots,u_d), v = (v_1,\ldots,v_d) \in A^+$ such that $(u_2,\ldots,u_d)$ and $ (v_2,\ldots,v_d)$ are linearly independent and $u_1, v_1 \geq 0$. 
\end{proof}

We also write $u^{+}=u$ and $v^{+}=v$, and we define $u^{-},v^{-} \in \Z^d$ by flipping all but the first coordinates of $u$, respectively $v$, i.e.,
\begin{equation*}
	u^{-}=(u_1,-u_2,-u_3,\ldots,-u_d) \text{ and } v^{-}=(v_1,-v_2,-v_3,\ldots,-v_d) .
\end{equation*}
As the vectors $(u_2,\ldots,u_d)$ and $(v_2,\ldots,v_d)$ are linearly independent, we directly get that $u^{+}, u^{-}, v^{+}, v^{-}$ are four different vectors. However, note that $u^+=-u^-$ or $v^+=-v^-$ can be the case if $u_1=0$, respectively if $v_1=0$. \\

Define $Q \coloneqq \max\{\|u\|_\infty, \|v\|_\infty\}$.
Throughout this section (the rest of Section \ref{sec:proofmain}), we assume that
\begin{equation}\label{eq:leq 1}
	\sum_{\substack{x: \|x\|_\infty = n, \\ |x_1|=n}} p_x \leq 1
\end{equation}
for all $n>Q$. Indeed, if this does not hold, then define the probabilities $\left(\tilde{p}_x\right)_{x\in \Z^d \setminus \{\mz\}}$ by 
\begin{equation*}
	\tilde{p}_x = \left(1 \vee \sum_{\substack{x: \|x\|_\infty = n, \\ |x_1|=n}} p_x \right)^{-1} p_x  \ \ \ \text{ for each $x\in \Z^d\setminus\{\mz\}$ with $\|x\|_\infty = n$}
\end{equation*}
and let $\tilde{\p}$ be the probability measure where each edge $\{x,y\}$ is open with probability $\tilde{p}_{x-y}$, independent of all other edges. The probability measure $\tilde{\p}$, respectively the collection $\left(\tilde{p}_x\right)_{x\in \Z^d \setminus \{\mz\}}$ satisfies the non-summability \eqref{eq:non-summable}, symmetry \eqref{eq:symmetry}, irreducibility \eqref{eq:irreduc}, and \eqref{eq:leq 1}. Using the {\sl Harris coupling} (cf. \cite[Section 1.3]{heydenreich2017progress}), one can couple the measures $\p$ and $\tilde{\p}$ as follows. Let $E = \left\{\{x,y\}\subset \Z^d, x \neq y\right\}$ and let $(U_{e})_{e \in E}$ be i.i.d. random variables such that each random variable $U_{e}$ is uniformly distributed on $[0,1]$. Define the percolation environments $(\omega(e))_{e\in E}$ and $(\tilde{\omega}(e))_{e\in E}$ by
\begin{equation*}
	\omega(\{x,y\})= \mathbbm{1}_{\{U_{\{x,y\}} \leq p_{x-y}\}} \quad \text{ and } \quad \tilde{\omega}(\{x,y\})= \mathbbm{1}_{\{U_{\{x,y\}} \leq \tilde{p}_{x-y}\}} .
\end{equation*}
Then $\omega$ is distributed like a percolation environment sampled from $\p$ and $\tilde{\omega}$ is distributed like a percolation environment sampled from $\tilde{\p}$. Since $\tilde{p}_{x-y} \leq p_{x-y}$ by the definition of $\tilde{p}_{x-y}$ above, we also directly get that $\tilde{\omega}(\{x,y\}) \leq \omega(\{x,y\})$. This directly implies that for all $N \in \N$
\begin{equation*}
	\left|K_\mz^N(\tilde{\omega})\right| = \infty \quad \Longrightarrow \quad \left|K_\mz^N(\omega)\right| = \infty,
\end{equation*}
where $K_\mz^N(\tilde{\omega})$ $($respectively $K_\mz^N(\omega))$ is the set of vertices $x \in \Z^d$ such that there exists a path from $\mz$ to $x$ consisting only of edges $\{a,b\}$ with $\|a-b\| \leq N$ and $\tilde{\omega}(\{a,b\}) = 1$ $($respectively $\omega(\{a,b\})=1)$.
Thus we see that $\p \left( |K_\mz^N|=\infty \right) \geq \tilde{\p} \left( |K_\mz^N|=\infty \right)$ for all $N \in \N$.
So showing that $\tilde{\p}(|K_\mz^N|=\infty) >0$ for this new measure and some $N\in \N$ implies that $\p(|K_\mz^N|=\infty) > 0$ for the original measure. Because of this, we can assume that \eqref{eq:leq 1} holds for all $n>Q$. \\

\noindent
For $M>3Q$, define the set $A_M \subset \Z^d$ by
\begin{equation}\label{eq:AM defini}
	A_M = \left\{x=(x_1,\ldots,x_d)\in \Z^d :  \|x\|_\infty \leq M, x_1 \geq 3Q \right\}
\end{equation}
and a probability measure $\psi_M$ on $A_M$ by
\begin{equation*}
	\psi_M(x) = \frac{p_x}{\sum_{y\in A_M} p_y}.
\end{equation*}
Also, note that $\sum_{y\in A_M} p_y$ diverges to $+\infty$ as $M\to \infty$, by \eqref{eq:direction 1 infty} and the reflection-symmetry \eqref{eq:symmetry}. From here on, fix $M>3Q$ large enough so that
\begin{equation}\label{eq:Mconditions}
	\eps^{-1} \coloneqq \sum_{y\in A_M} p_y \geq 10^{10} \text{ and } \frac{25004\eps^{0.05}}{\eta^2} \leq \frac{1}{4} .
\end{equation}

Let $(X_n)_{n\in \N_0}$ be i.i.d. random variables with probability mass function $\psi_M$ and let $(S_n)_{n\in \N_0}$, $(S_n^\prime)_{n\in \N_0}$ be random walks with increments in $\{-1,+1\}$ satisfying \eqref{eq:predict 100 old}. Furthermore, let  $(S_n)_{n\in \N_0}$, $(S_n^\prime)_{n\in \N_0}$, and $(X_n)_{n\in \N_0}$ be independent.
We define a random path $\zeta=(\zeta_0,\ldots)$ by
\begin{align}\label{eq:gamma}
	\zeta_0=\mz \text{ and } \zeta_{k+1}-\zeta_k = 
	\begin{cases}
		X_{k/3} & \text{ if } k = 0 \mod 3 \\
		u^{+} & \text{ if } k = 1 \mod 3, S_{\frac{k+2}{3}} - S_{\frac{k-1}{3}} = +1 \\
		u^{-} & \text{ if } k = 1 \mod 3, S_{\frac{k+2}{3}} - S_{\frac{k-1}{3}} = -1\\
		v^{+} & \text{ if } k = 2 \mod 3, S_{\frac{k+1}{3}}^\prime - S_{\frac{k-2}{3}}^\prime = +1 \\
		v^{-} & \text{ if } k = 2 \mod 3, S_{\frac{k+1}{3}}^\prime - S_{\frac{k-2}{3}}^\prime = -1
	\end{cases} .
\end{align}
Note that $\langle \zeta_{k+1}-\zeta_k , e_1 \rangle\geq 0$ for all $k\in \N_0$, with a strict inequality if $k=0\mod 3$. Further, by the definition of $u^+, u^-, v^+$, and $v^-$, one sees that the path $(\zeta_k)_{k\in \N_0}$ is self-avoiding.
We write $\zeta_0^n = (\zeta_0,\ldots,\zeta_n)$ for the first $n$ steps of $\zeta$ and $\hat{\mu}$ for the distribution of the sequence $(\zeta_k)_{k\in \N_0}$ constructed in \eqref{eq:gamma}. 
From the construction, it is also clear, that for a given path segment $\zeta_0^n$, one can reconstruct the random variables $X_k$ for $k=0,\ldots,\lceil \frac{n}{3} \rceil - 1$ and the random variables $S_{k+1}-S_{k}$ for $k=0,\ldots,\lceil \frac{n-1}{3} \rceil - 1$ and $S_{k+1}^\prime-S_{k}^\prime$ for $k=0,\ldots,\lceil \frac{n-2}{3} \rceil-1$.
Let $M_1$ be the set of probability measures on $\cP^\infty$ that are the distributions of $(\zeta_{n}-\zeta_{3N})_{n\geq 3N}$, conditioned on $\zeta_0,\ldots,\zeta_{3N}$, where $(\zeta_j)_{j\in \N_0}$ has distribution $\hat{\mu}$ and $N\in \N_0$. One important feature about the measure $M_1$ is that if the distribution of a random self-avoiding path $\left(\gamma_k\right)_{k\in \N_0}$ is in $M_1$ and $K$ is a random variable such that $K=0 \mod 3$ almost surely and $K$ is measurable with respect to $\sigma(\gamma_0,\ldots,\gamma_K)$, then also the distribution of the sequence  $\left(\gamma_{K+k}-\gamma_K\right)_{k\in \N_0}$ lies in $M_1$.\\

Next, we consider two paths $\gamma=(\gamma_k)_{k\in \N_0}$ and $\varphi=(\varphi_k)_{k\in \N_0}$ that are distributed according to measures in $M_1$. We will mostly work with the convention that the two infinite paths are sampled from the measure $\bp_{\tilde{\mu}\times \mu}$, and that $\gamma$ has distribution $\tilde{\mu}$, whereas $\varphi$ has distribution $\mu$.
Define the number $W_n$ by
\begin{equation}\label{eq:Wn defini 1}
	W_n = \sup_{\tilde{\mu},\mu \in M_1, x \in \Z^d} \be_{\tilde{\mu} \times \mu} \left[\ov(x+\gamma_0^{3n},\varphi_0^\infty)\right].
\end{equation}
Note that almost surely
\begin{equation*}
	\ov(x+\gamma_0^{3n},\varphi_0^\infty) \leq \left(\eta \wedge \min\{p_y: y \in A_M, p_y>0\} \right)^{-3n},
\end{equation*}
so in particular $W_n<\infty$. Further, note that the numbers $W_n$ are non-decreasing in $n$, as for fixed $x \in \Z^d, \gamma_0^\infty, \varphi_0^\infty \in \cP^\infty$ the weighted overlap $\ov(x+\gamma_0^{3n},\varphi_0^\infty)$ is also non-decreasing in $n$.
In Lemma \ref{lem:Wn def 1} below, we derive a bootstrapping-type inequality for $W_n$, which implies that $\sup_n W_n \leq 2 < \infty$. Thus we get that $\be_{\hat{\mu} \times \hat{\mu}} \left[ \ov(\gamma_0^{3n},\varphi_0^{3n}) \right] \leq W_n \leq 2$.
Proposition \ref{prop:measure on paths} then implies that
\begin{equation*}
	\p\left(|K_\mz^M|=\infty\right) 
	\geq
	\left( \be_{\hat{\mu} \times \hat{\mu}} \left[ \ov(\gamma_0^\infty,\varphi_0^\infty) \right]\right)^{-1}
	=
	\left( \sup_{n} \be_{\hat{\mu} \times \hat{\mu}} \left[ \ov(\gamma_0^{3n},\varphi_0^{3n}) \right]\right)^{-1}
	\geq
	\frac{1}{\sup_n W_n} \geq \frac{1}{2} > 0.
\end{equation*}

\begin{lemma}\label{lem:Wn def 1}
	For the quantity $W_n$ defined in \eqref{eq:Wn defini 1} one has $W_n \leq 2$ for all $n\in \N$.
\end{lemma}

\begin{proof}
	
	We show that for all $\tilde{\mu},\mu \in M_1$ and all $x\in \Z^d$ one has
	\begin{equation*}
		\be_{\tilde{\mu} \times \mu} \left[\ov(x+\gamma_0^{3n},\varphi_0^\infty)\right] \leq 1 + \frac{W_n}{2},
	\end{equation*}
	which implies that $W_n \leq 1 + \frac{W_n}{2}$, or $W_n \leq 2$; note that one can safely rearrange these inequalities as $W_n < \infty$. For the inequality above, we distinguish the cases $x=\mz$ and $x\neq \mz$.\\
	
	\hypertarget{targetcase1}{\textbf{Case 1:}}
	We start with the case  $x\neq \mz$. Define the random variables $K=K(x+\gamma_0^{3n},\varphi_0^\infty)$ and $N=N(x+\gamma_0^{3n},\varphi_0^\infty)$ by 
	\begin{align*}
		& K(x+\gamma_0^{3n},\varphi_0^\infty) = \inf \left\{z\in \{0,\ldots,3n\} : \ z=0 \mod 3 \text{ and } x+\gamma_{\tilde{z}}=\varphi_\ell \text{ for some $\tilde{z}\leq z, \ell\in \N_0$}\right\},\\
		&
		N(x+\gamma_0^{3n},\varphi_0^\infty) = \inf \left\{\ell\in \N_0 : \ \ell=0 \mod 3 \text{ and } x+\gamma_{z}=\varphi_{\tilde{\ell}} \text{ for some $\tilde{\ell}\leq \ell, z \in \{0,\ldots,3n\}$}\right\}.
	\end{align*}
	Both $K$ and $N$ can attain the value $+\infty$ if there does not exist $z,\ell$ fulfilling the necessary requirements above. Also, note that $K<\infty$ if and only if $N<\infty$.
	
	Assume that $\tilde{z}\in \{0,\ldots,3n\}$ and $\tilde{\ell} \in \N_0$ are such that $x+\gamma_{\tilde{z}}=\varphi_{\tilde{\ell}}$ and $z\in \{0,3,6,\ldots,3n\}$ and $\ell \in 3\N_0$ are the smallest integer multiples of $3$ such that $K = z \geq \tilde{z}$ and $ \ell \geq \tilde{\ell}$. We will first argue that this implies that $N=\ell$. Assume that there exists $\hat{\ell}\leq \ell -3$ and $\hat{z} \in \{0,\ldots,3n\}$ such that $\varphi_{\hat{\ell}}=x+\gamma_{\hat{z}}$. Let $z^\prime$ and $\ell^\prime$ be the smallest integer multiple of $3$ for which $z^\prime \geq \hat{z}$ and $\ell^\prime \geq \hat{\ell}$. Then $\gamma_{z^\prime}=\gamma_{\hat{z}}+r$, where $r\in \Z^d$ is of the form $r=\sum_{i=1}^{j} v_i$, with $j\in \{0,1,2\}$ and $v_i \in \{u^+, u^-, v^+, v^-\}$. So in particular $\langle \gamma_{z^\prime}, e_1 \rangle \leq \langle \gamma_{\hat{z}}, e_1 \rangle+2Q$. As $\ell^\prime = 0 \mod 3$ and $\tilde{\ell} > \ell^\prime$, by the definition of the set $A_M$ \eqref{eq:AM defini}, one has $\langle \varphi_{\tilde{\ell}}, e_1 \rangle \geq \langle \varphi_{\ell^\prime}, e_1 \rangle+3Q$. Combining these two inequalities with the fact that $t \to \langle  x+\gamma_{t}, e_1 \rangle$ and $t \to \langle  \varphi_{t}, e_1 \rangle$ are increasing in $t$, we get that
	\begin{align*}
		\langle  x+\gamma_{z^\prime}, e_1 \rangle 
		&\leq 
		\langle  x+\gamma_{\hat{z}}, e_1 \rangle + 2Q
		=
		\langle  \varphi_{\hat{\ell}}, e_1 \rangle + 2Q
		\leq
		\langle  \varphi_{\ell^\prime}, e_1 \rangle + 2Q\\
		&
		\leq
		\langle  \varphi_{\tilde{\ell}}, e_1 \rangle -Q
		=
		\langle  x+\gamma_{\tilde{z}}, e_1 \rangle -Q
		<
		\langle  x+\gamma_{\tilde{z}}, e_1 \rangle 
		\leq
		\langle  x+\gamma_{z}, e_1 \rangle .
	\end{align*}
	As $\langle  x+\gamma_{t}, e_1 \rangle$ is increasing in $t$, this implies that $z^\prime < z$. However, this is a contradiction to the fact that $z=K$ and the fact that $K$ was defined to be the smallest integer multiple of $3$ for which there exists $\tilde{z} \in \{z-2,z-1,z\}$ and $\tilde{\ell} \in \N_0$ such that $x+\gamma_{\hat{z}}=\varphi_{\hat{\ell}}$. So we see that $N=\ell$.
	
	Next, we will argue that the path $x+\gamma_0^K$ can not have any joint edges with $\varphi_N^\infty$:
	Using that $\langle x+\gamma_z, e_1 \rangle \leq \langle x+\gamma_{\tilde{z}},e_1 \rangle +2Q$ and that $\langle\varphi_{\ell+1},e_1\rangle \geq \langle\varphi_{\ell},e_1\rangle + 3Q$ by the definition of the set $A_M$ \eqref{eq:AM defini}, we see that
	\begin{align*}
		\langle\varphi_{N+1},e_1\rangle &
		=
		\langle\varphi_{\ell+1},e_1\rangle \geq \langle\varphi_{\ell},e_1\rangle + 3Q > 
		\langle\varphi_{\ell},e_1\rangle + 2Q 
		\geq
		\langle\varphi_{\tilde{\ell}},e_1\rangle + 2Q 
		=
		\langle x+\gamma_{\tilde{z}},e_1\rangle + 2Q \\
		&
		\geq
		\langle x+\gamma_z,e_1\rangle 
		=
		\langle x+\gamma_K,e_1\rangle 
	\end{align*}
	and thus $\langle \varphi_{N+i}, e_1 \rangle > \langle x + \gamma_m, e_1 \rangle$ for all $i \in \N_{>0}$ and $m\in \{0,\ldots,K\}$. In particular, this implies that $x+\gamma_0^K$ can not have any joint edges with $\varphi_N^\infty$. By the same argument, $\varphi_0^N$ can not have any joint edges with $x+\gamma_K^\infty$. Thus, if $K<\infty$, we can write
	\begin{equation}\label{eq:overlap}
		\ov(x+\gamma_0^{3n},\varphi_0^\infty) = \ov(x+\gamma_0^{K},\varphi_0^N) \cdot 
		\ov(x+\gamma_K^{3n},\varphi_N^\infty)
	\end{equation}
	The two paths $x+\gamma_0^K$ and $\varphi_0^N$ can overlap on at most two edges that are translates of elements in $\{u^+, u^-, v^+, v^-\}$. Indeed, if the paths $x+\gamma_0^{3n}$ and $\varphi_0^\infty$ overlap on any edge $\{a,b\}=\{x+\gamma_j, x+\gamma_{j+1}\}=\{\varphi_i, \varphi_{i+1}\}$ which is not a translate of an element in $\{u^+, u^-, v^+, v^-\}$, then, by construction, $i \mod 3 = j \mod 3 = 0 \mod 3$, which implies that
	$N \leq i, K \leq j$. Thus, if $K<\infty$, we see that $\ov(x+\gamma_0^{K},\varphi_0^N) \leq \min(p_{u^+},p_{u^-},p_{v^+},p_{v^-}) = \eta^{-2}$. Plugging this into \eqref{eq:overlap} yields that
	\begin{align}\label{total problaw}
		\ov(x+\gamma_0^{3n},\varphi_0^\infty) & \leq 1 +
		\frac{1}{\eta^2} \mathbbm{1}_{\{K(x+\gamma_0^{3n},\varphi_0^\infty)<\infty\}}  \ov(x+\gamma_{K}^{3n},\varphi_{N}^\infty) .
	\end{align}
	From this formula, we see that it is of central importance to control the probability $\p(K<\infty)$. In Lemma \ref{lemma:hitting} below, we show that
	\begin{align}\label{eq:hitting different point in proof}
		\bp_{\tilde{\mu}\times \mu}\left(K\left(x+\gamma_0^{3n},\varphi_0^\infty\right) < \infty \right) \leq & \bp_{\tilde{\mu}\times \mu} \left(\exists k,m\in \N_0: x+\gamma_k=\varphi_m \right) \leq 25004\eps^{0.05} \leq \frac{1}{4} ,
	\end{align}
	where the last inequality follows from the assumption on $\eps$ \eqref{eq:Mconditions}. The assumption \eqref{eq:assumptioneps} of Lemma \ref{lemma:hitting} is satisfied in the situation considered because
	\begin{equation*}
		\bp_{\tilde{\mu}\times \mu} \left(\langle \varphi_1, e_1 \rangle = z\right) 
		= 
		\bp_{\tilde{\mu}\times \mu} \left(\langle \gamma_1, e_1 \rangle = z\right) = \frac{\sum_{y\in A_M: \langle y,e_1\rangle =z} p_y }{\sum_{y\in A_M} p_y} \overset{\eqref{eq:leq 1}}{\leq } \frac{1}{\sum_{y\in A_M} p_y} \overset{\eqref{eq:Mconditions}}{=} \eps \leq 10^{-10}
	\end{equation*}
	for all $z\in \Z$.
	From the two paths $\gamma_0^\infty$ and $\varphi_0^\infty$, on the event $K<\infty$, we define two new paths by $\left(\tilde{\gamma}_i\coloneqq\gamma_{K+i}-\gamma_K\right)_{i \in \N_0}$ and $(\tilde{\varphi}_i \coloneqq \varphi_{N+i}-\varphi_N)_{i\in \N_0}$.
	Using these two shifted paths, we can rewrite and bound the second part of the above inequality by
	\begin{align*}
		&\ov(x+\gamma_{K}^{3n},\varphi_{N}^\infty) = \ov(x-\varphi_{N}+\gamma_K-\gamma_K+\gamma_{K}^{3n},-\varphi_N+\varphi_{N}^\infty)\\
		&
		=
		\ov(x-\varphi_{N}+\gamma_K+\tilde{\gamma}_{0}^{3n-K},\tilde{\varphi}_{0}^\infty)
		\leq
		\ov(x-\varphi_{N}+\gamma_K+\tilde{\gamma}_{0}^{3n},\tilde{\varphi}_{0}^\infty).
	\end{align*}
	Note that the random variables $K$ and $N$ are measurable with respect to the $\sigma$-algebra generated by $\gamma_0,\ldots,\gamma_K,\varphi_0,\ldots,\varphi_N$. Given $\sigma\left(\gamma_0^K, \varphi_0^N\right)$, the distributions of $\left(\tilde{\gamma}_i\right)_{i\in \N_0}$ and $\left(\tilde{\varphi}_i\right)_{i\in \N_0}$ are still measures in $M_1$, and the random variable $-\varphi_{N}+\gamma_K$ is measurable with respect to $\sigma\left(\gamma_0^K, \varphi_0^N\right)$.
	Thus we see that
	\begin{equation*}
		\be_{\tilde{\mu} \times \mu} \left[ \ov(x-\varphi_{N}+\gamma_K+\tilde{\gamma}_{0}^{3n},\tilde{\varphi}_{0}^\infty) | \sigma\left(\gamma_0^K, \varphi_0^N\right)\right] \leq W_n
	\end{equation*}
	and using the tower rule implies that
	\begin{align*}
		& \be_{\tilde{\mu} \times \mu} \left[ \ov(x+\gamma_{K}^{3n},\varphi_{N}^\infty) | K<\infty \right]
		=
		\be_{\tilde{\mu} \times \mu} \left[ \be_{\tilde{\mu} \times \mu} \left[ \ov(x+\gamma_{K}^{3n},\varphi_{N}^\infty) | \sigma\left(\gamma_0^K, \varphi_0^N\right)\right] | K<\infty \right]
		\\
		&
		\leq
		\be_{\tilde{\mu} \times \mu} \left[ \be_{\tilde{\mu} \times \mu} \left[ \ov(x-\varphi_{N}+\gamma_K+\tilde{\gamma}_{0}^{3n},\tilde{\varphi}_{0}^\infty) | \sigma\left(\gamma_0^K, \varphi_0^N\right)\right] | K<\infty \right]
		\leq
		\be_{\tilde{\mu} \times \mu}  \left[W_n | K<\infty \right]\\
		& = W_n.
	\end{align*}
	Taking expectations in \eqref{total problaw} and using the law of total probability, we get that for all $x\neq \mz$
	\begin{align}\label{eq:different starting point}
		&\notag \be_{\tilde{\mu} \times \mu} \left[ \ov(x+\gamma_0^{3n},\varphi_0^\infty) \right]
		\leq
		1 +
		\frac{1}{\eta^2} 
		\be_{\tilde{\mu} \times \mu} \left[\ov(x+\gamma_{K}^{3n},\varphi_{N}^\infty) | K<\infty\right] 
		\bp_{\tilde{\mu}\times \mu} \left( K<\infty \right)\\
		&
		\leq
		1 +
		\frac{1}{\eta^2} 
		W_n
		\bp_{\tilde{\mu}\times \mu} \left( K<\infty \right)
		\leq
		1 +
		\frac{W_n}{4}
	\end{align}
	where the last inequality holds because of \eqref{eq:hitting different point in proof}. This finishes the case $x\neq \mz$.\\
	
	\textbf{Case 2:}
	In the case where $x=\mz$, we define the random variables $K^\prime=K^\prime(\gamma_0^{3n},\varphi_0^\infty)$ and $N^\prime=N^\prime(\gamma_0^{3n},\varphi_0^\infty)$ by 
	\begin{align*}
		& K^\prime(\gamma_0^{3n},\varphi_0^\infty) = \inf \left\{z\in \{1,\ldots,3n\} : \ z=0 \mod 3 \text{ and } \gamma_{\tilde{z}}=\varphi_\ell \text{ for some $\tilde{z}\leq z, \ell\in \N_{>0}$}\right\},\\
		&
		N^\prime(\gamma_0^{3n},\varphi_0^\infty) = \inf \left\{\ell\in \N_{>0} : \ \ell=0 \mod 3 \text{ and } \gamma_{z}=\varphi_{\tilde{\ell}} \text{ for some $\tilde{\ell}\leq \ell, z \in \{1,\ldots,3n\}$}\right\}.
	\end{align*}
	The only difference to the random variables $K$ and $N$ defined above is that we exclude the trivial case $z=0, \ell=0$, for which $\varphi_0=\mz=\gamma_0$. Also, note that $K^\prime < \infty$ if and only if $N^\prime < \infty$.
	On the event where $\gamma_1=\varphi_1=w$ for some $w\in \Z^d$, the paths $\gamma_0^3$ and $\varphi_3^\infty$, respectively $\varphi_0^3$ and $\gamma_3^\infty$ can have no joint edges as
	\begin{equation*}
		\langle \gamma_4 , e_1 \rangle 
		\geq 
		\langle \gamma_1 , e_1 \rangle + 3Q
		=
		\langle \varphi_1 , e_1 \rangle + 3Q
		>
		\langle \varphi_1 , e_1 \rangle + 2Q
		\geq
		\langle \varphi_3 , e_1 \rangle ,
	\end{equation*}
	and the same holds also with the roles of $\varphi$ and $\gamma$ reversed. Thus we get that
	\begin{align*}
		&\mathbbm{1}_{\{\gamma_1 = w =\varphi_1\}} \ov\left(\gamma_0^{3n}, \varphi_0^\infty\right)
		=
		\mathbbm{1}_{\{\gamma_1 = w =\varphi_1\}}  \ov\left(\gamma_0^3, \varphi_0^3\right)
		\ov\left(\gamma_3^{3n}, \varphi_3^\infty\right)\\
		&
		=
		\mathbbm{1}_{\{\gamma_1 = w =\varphi_1\}} \frac{1}{p_w} \ov\left(\gamma_1^3, \varphi_1^3\right)
		\ov\left(\gamma_3^{3n}, \varphi_3^\infty\right)
		\leq
		\mathbbm{1}_{\{\gamma_1 = w =\varphi_1\}} \frac{1}{p_w} \frac{1}{\eta^2}
		\ov\left(\gamma_3^{3n}, \varphi_3^\infty\right)
	\end{align*}
	Also, note that the event $\gamma_1 = w =\varphi_1$ implies that $K^\prime=N^\prime=3$.
	On the event where $K^\prime < \infty$, the same argument as in \hyperlink{targetcase1}{Case 1} above shows that $\gamma_0^{K^\prime}$ can have no joint edges with $\varphi_{N^\prime}^\infty$, and that $\varphi_0^{N^\prime}$ can have no joint edges with $\gamma_{K^\prime}^\infty$. 
	On the event where $K^\prime < \infty$, but $\varphi_1 \neq \gamma_1$, the two paths $\gamma_0^{K^\prime}$ and $\varphi_0^{N^\prime}$ can only intersect on translates of $\{u^+, u^-, v^+, v^-\}$. In particular
	\begin{align*}
		& \mathbbm{1}_{\{\gamma_1 \neq \varphi_1\}} \mathbbm{1}_{\{K^\prime < \infty\}} \ov\left(\gamma_0^{3n}, \varphi_0^\infty\right)
		=
		\mathbbm{1}_{\{\gamma_1 \neq \varphi_1\}} \mathbbm{1}_{\{K^\prime < \infty\}} \ov\left(\gamma_0^{K^\prime}, \varphi_0^{N^\prime}\right) \ov\left(\gamma_{K^\prime}^{3n}, \varphi_{N^\prime}^\infty\right)\\
		&
		\leq
		\mathbbm{1}_{\{\gamma_1 \neq \varphi_1\}} \mathbbm{1}_{\{K^\prime < \infty\}} \frac{1}{\eta^2} \ov\left(\gamma_{K^\prime}^{3n}, \varphi_{N^\prime}^\infty\right)
		\leq
		\mathbbm{1}_{\{K^\prime < \infty\}} \frac{1}{\eta^2} \ov\left(\gamma_{K^\prime}^{3n}, \varphi_{N^\prime}^\infty\right)
	\end{align*}
	Going through all of these cases ($\gamma_1 = w =\varphi_1$ for some $w\in A_M$, $K^\prime < \infty$ but $\gamma_1\neq \varphi_1$, and $K^\prime=\infty$), we see that the inequality
	\begin{align}\label{eq:overlap bound}
		\ov(\gamma_0^{3n},\varphi_0^\infty) \leq &
		1 + \sum_{w\in A_M} \frac{1}{\eta^2 p_w} \mathbbm{1}_{\{\gamma_1 = w =\varphi_1\}} \mathbbm{1}_{\{K^\prime < \infty\}}  \ov(\gamma_{K^\prime}^{3n},\varphi_{N^\prime}^\infty)
		+ \frac{1}{\eta^2}  \ov \left(\gamma_{K^\prime}^{3n},\varphi_{N^\prime}^\infty\right) \mathbbm{1}_{\{K^\prime<\infty\}}
	\end{align}
	holds. From this formula, we see that it is of central importance to control the probability $\p(K^\prime<\infty)$. In Lemma \ref{lemma:hitting} below, we show that
	\begin{align}\label{eq:hitting same point in proof}
		\bp_{\tilde{\mu}\times \mu} \left( K^\prime \left(\gamma_0^{3n},\varphi_0^\infty\right) < \infty \right) \leq & \bp_{\tilde{\mu}\times \mu} \left(\exists k,m\in \N: \gamma_k=\varphi_m \right) \leq 25004\eps^{0.05} \leq \frac{1}{4}  .
	\end{align}
	where the last inequality follows from the assumption on $\eps$ \eqref{eq:Mconditions}. The assumption \eqref{eq:assumptioneps} of Lemma \ref{lemma:hitting} is satisfied in the situation considered because
	\begin{equation*}
		\bp_{\tilde{\mu}\times \mu} \left(\langle \varphi_1, e_1 \rangle = z\right) 
		= 
		\bp_{\tilde{\mu}\times \mu} \left(\langle \gamma_1, e_1 \rangle = z\right) = \frac{\sum_{y\in A_M: \langle y,e_1\rangle =z} p_y }{\sum_{y\in A_M} p_y} \overset{\eqref{eq:leq 1}}{\leq } \frac{1}{\sum_{y\in A_M} p_y} \overset{\eqref{eq:Mconditions}}{=} \eps \leq 10^{-10}
	\end{equation*}
	for all $z\in \Z$.
	
	If $K^\prime<\infty$, conditioned on the $\sigma$-Algebra $\sigma\left(\gamma_0^{K^\prime},\varphi_0^{N^\prime}\right)$, the sequences $\left(\tilde{\varphi}_i = \varphi_{N^\prime+i} - \varphi_{N^\prime} \right)_{i\in \N_0}$ and $\left(\tilde{\gamma}_i = \gamma_{K^\prime+i} - \gamma_{K^\prime} \right)_{i\in \N_0}$ are still distributed according to measures in $M_1$. Further we have
	\begin{align*}
		& \ov\left(\gamma_{K^\prime}^{3n}, \varphi_{N^\prime}^{\infty} \right)
		=
		\ov\left(-\varphi_{N^\prime}+\gamma_{K^\prime}-\gamma_{K^\prime}+\gamma_{K^\prime}^{3n}, -\varphi_{N^\prime}+\varphi_{N^\prime}^{\infty} \right)\\
		&
		=
		\ov\left(-\varphi_{N^\prime}+\gamma_{K^\prime}+\tilde{\gamma}_{0}^{3n-K^\prime}, \tilde{\varphi}_{0}^{\infty} \right)
		\leq
		\ov\left(-\varphi_{N^\prime}+\gamma_{K^\prime}+\tilde{\gamma}_{0}^{3n}, \tilde{\varphi}_{0}^{\infty} \right)
	\end{align*}
	which implies that
	\begin{align*}
		&\be_{\tilde{\mu} \times \mu} \left[ \ov(\gamma_{K^\prime}^{3n},\varphi_{N^\prime}^\infty) | K^\prime<\infty \right]
		=
		\be_{\tilde{\mu} \times \mu} \left[ \be_{\tilde{\mu} \times \mu} \left[ \ov(\gamma_{K^\prime}^{3n},\varphi_{N^\prime}^\infty) | \sigma(\gamma_0^{K^\prime}, \varphi_0^{N^\prime})\right] \Big| K^\prime<\infty \right]
		\\
		&
		\leq
		\be_{\tilde{\mu} \times \mu} \left[ \be_{\tilde{\mu} \times \mu} \left[ \ov(\varphi_{N^\prime}+\gamma_{K^\prime}+\tilde{\gamma}_{0}^{3n},\tilde{\varphi}_{0}^\infty) | \sigma(\gamma_0^{K^\prime}, \varphi_0^{N^\prime})\right] \Big| K^\prime<\infty \right]
		\leq
		\be_{\tilde{\mu} \times \mu}  \left[W_n \big| K^\prime<\infty \right]\\
		& = W_n.
	\end{align*}
	Taking expectations in \eqref{eq:overlap bound}, and using that $ \mathbbm{1}_{\{\gamma_1 = w =\varphi_1\}} $ is measurable with respect to $\sigma \left(\gamma_0^{K^\prime}, \varphi_0^{N^\prime}\right)$, we get for the second summand in \eqref{eq:overlap bound} that
	\begin{align}\label{eq:first bound}
		&\notag \be_{\tilde{\mu}\times \mu} \left[ \sum_{w\in A_M} \frac{1}{\eta^2 p_w} \mathbbm{1}_{\{\gamma_1 = w =\varphi_1\}} \mathbbm{1}_{\{K^\prime < \infty\}}  \ov(\gamma_{K^\prime}^{3n},\varphi_{N^\prime}^\infty) \right]
		\\
		&\notag
		=
		\sum_{w\in A_M}  \frac{1}{\eta^2 p_w}
		\be_{\tilde{\mu}\times \mu} \left[  \be_{\tilde{\mu}\times \mu} \left[  \mathbbm{1}_{\{\gamma_1 = w =\varphi_1\}}  \ov(\gamma_{K^\prime}^{3n},\varphi_{N^\prime}^\infty) | \sigma\left(\gamma_0^{K^\prime}, \varphi_0^{N^\prime}\right)  \right] \big| K^\prime<\infty \right] \bp_{\tilde{\mu} \times \mu} \left(K^\prime < \infty\right) \\
		&\notag
		=
		\sum_{w\in A_M}  \frac{1}{\eta^2 p_w}
		\be_{\tilde{\mu}\times \mu} \left[  \mathbbm{1}_{\{\gamma_1 = w =\varphi_1\}}  \be_{\tilde{\mu}\times \mu} \left[ \ov(\gamma_{K^\prime}^{3n},\varphi_{N^\prime}^\infty) | \sigma\left(\gamma_0^{K^\prime}, \varphi_0^{N^\prime}\right)  \right] \big| K^\prime<\infty \right] \bp_{\tilde{\mu} \times \mu} \left(K^\prime < \infty\right) \\
		& \notag
		\leq 
		\sum_{w\in A_M}  \frac{1}{\eta^2 p_w}
		\be_{\tilde{\mu}\times \mu} \left[  \mathbbm{1}_{\{\gamma_1 = w =\varphi_1\}}  W_n | K^\prime<\infty \right] \bp_{\tilde{\mu} \times \mu} \left(K^\prime < \infty\right)
		\\
		& \notag
		= 
		\sum_{w\in A_M}  \frac{W_n}{\eta^2 p_w}
		\be_{\tilde{\mu}\times \mu} \left[  \mathbbm{1}_{\{\gamma_1 = w =\varphi_1\}}  \mathbbm{1}_{\{K^\prime < \infty\}}   \right]
		= 
		\sum_{w\in A_M}  \frac{W_n}{\eta^2 p_w}
		\bp_{\tilde{\mu}\times \mu} \left( \gamma_1=w=\varphi_1  \right)
		\\
		&=
		\frac{W_n}{\eta^2}
		\sum_{w\in A_M}  \frac{1}{p_w}
		\psi_M(w)^2
		=
		\frac{W_n}{\eta^2}
		\sum_{w\in A_M} \frac{1}{p_w} \left(\frac{p_w}{\sum_{y\in A_M} p_y}\right)^2 
		= \frac{W_n}{\eta^2 \sum_{y\in A_M} p_y} \leq \frac{W_n}{4},
	\end{align}
	where we used \eqref{eq:Mconditions} for the last inequality.
	For the last summand in \eqref{eq:overlap bound}, a similar argument as in \hyperlink{targetcase1}{Case 1} above shows that
	\begin{align}\label{eq:second bound}
		&\notag \be_{\tilde{\mu} \times \mu} \left[\frac{1}{\eta^2}  \ov \left(\gamma_{K^\prime}^{3n},\varphi_{N^\prime}^\infty\right) \mathbbm{1}_{\{K^\prime < \infty \}}\right]\\
		&\notag
		=
		\be_{\tilde{\mu} \times \mu} \left[\frac{1}{\eta^2} \be_{\tilde{\mu} \times \mu} \left[ \ov \left(\gamma_{K^\prime}^{3n},\varphi_{N^\prime}^\infty\right) | \sigma \left(\gamma_0^{K^\prime}, \varphi_0^{N^\prime}\right) \right]  \Big| K^\prime < \infty \right] \bp_{\tilde{\mu} \times \mu} \left(K^\prime < \infty\right) \\
		&
		\leq
		\be_{\tilde{\mu} \times \mu} \left[\frac{1}{\eta^2} W_n \big| K^\prime < \infty \right]
		\bp_{\tilde{\mu} \times \mu} \left(K^\prime < \infty\right)
		=
		\frac{W_n}{\eta^2}
		\bp_{\tilde{\mu} \times \mu} \left(K^\prime < \infty\right)
		\leq \frac{W_n}{4},
	\end{align}
	where the last inequality follows by Lemma \ref{lemma:hitting} and \eqref{eq:Mconditions}. Taking expectations in \eqref{eq:overlap bound} and using \eqref{eq:first bound} and \eqref{eq:second bound}, we get that
	\begin{equation*}
		\be_{\tilde{\mu} \times \mu} \left[\ov(\gamma_0^{3n},\varphi_0^\infty)\right] \leq 1 + \frac{W_n}{2};
	\end{equation*}
	This finishes the case $x=\mz$.
	
\end{proof}

\begin{lemma}\label{lemma:hitting}
	Let $\gamma=\left(\gamma_k\right)_{k\in \N_0}, \varphi=\left(\varphi_k\right)_{k\in \N_0}$ be distributed according to the measure $\bp_{\tilde{\mu}\times \mu}$ with $\tilde{\mu},\mu \in M_1$. Assume that
	\begin{equation}\label{eq:assumptioneps}
		\bp_{\tilde{\mu}\times \mu} \left(\langle \varphi_1, e_1 \rangle = z\right) = \bp_{\tilde{\mu}\times \mu} \left(\langle \gamma_1, e_1 \rangle = z\right) \leq \eps
	\end{equation}
	for some $\eps\in (0,10^{-10})$ and all $z\in \Z$.
	Then, for all $x\in \Z^d$ with $x\neq \mz$ one has for the random variable $K\left(x+\gamma_0^{3n},\varphi_0^\infty\right)$ defined above that
	\begin{align}\label{eq:hitting different point}
		\bp_{\tilde{\mu}\times \mu}\left(K\left(x+\gamma_0^{3n},\varphi_0^\infty\right) < \infty \right) \leq & \bp_{\tilde{\mu}\times \mu} \left(\exists k,m\in \N_0: x+\gamma_k=\varphi_m \right) \leq 25004\eps^{0.05}  .
	\end{align}
	Furthermore, when both paths start at the origin one has for the random variable $K^\prime\left(\gamma_0^{3n},\varphi_0^\infty\right)$ defined above that
	\begin{align}\label{eq:hitting same point}
		\bp_{\tilde{\mu}\times \mu} \left( K^\prime \left(\gamma_0^{3n},\varphi_0^\infty\right) < \infty \right) \leq & \bp_{\tilde{\mu}\times \mu} \left(\exists k,m\in \N: \gamma_k=\varphi_m \right) \leq 25004\eps^{0.05}  .
	\end{align}
\end{lemma}

\begin{proof}
	We only show \eqref{eq:hitting different point}. The proof of \eqref{eq:hitting same point} follows by the same argument with union bounds over $\N_0$ replaced by union bounds over $\N$.
	We have that
	\begin{align*}
		\bp_{\tilde{\mu}\times \mu} \left(\exists k,m\in \N_0: x+\gamma_k=\varphi_m \right)
		&
		=
		\bp_{\tilde{\mu}\times \mu} \left(\exists k,m\in \N_0: \gamma_k=-x+\varphi_m \right)\\
		\bp_{\tilde{\mu}\times \mu} \left(\exists k,m\in \N_0: -x+\varphi_m =  \gamma_k \right)
		&
		=
		\bp_{\mu \times \tilde{\mu}} \left(\exists k,m\in \N_0: -x+\gamma_k=\varphi_m \right).
	\end{align*}
	So by symmetry in the measures $\mu$ and $\tilde{\mu}$ we can assume that $\langle x, e_1\rangle \geq 0$. Using that $\varphi_m$ and $\gamma_k$ are independent, respectively that $\varphi_1$ and $\varphi_{m}-\varphi_1$ are independent, we get for any $k \in \N_{0}, m \in \N_{>0}$ that
	\begin{align*}
		&\bp_{\tilde{\mu}\times \mu} \left( x+\gamma_k=\varphi_m \right)
		\leq
		\bp_{\tilde{\mu}\times \mu} \left( \langle x+\gamma_k, e_1 \rangle= \langle \varphi_m, e_1 \rangle \right)
		\leq
		\sup_{z\in \Z} \bp_{\tilde{\mu}\times \mu} \left( z= \langle \varphi_m, e_1 \rangle \right)\\
		&
		=
		\sup_{z\in \Z} \bp_{\tilde{\mu}\times \mu} \left( z= \langle \varphi_1, e_1 \rangle + \langle \varphi_{m}-\varphi_1, e_1 \rangle \right)
		\leq
		\sup_{z\in \Z} \bp_{\tilde{\mu}\times \mu} \left( z= \langle \varphi_1, e_1 \rangle \right)
		\leq \eps
	\end{align*}
	by \eqref{eq:assumptioneps}. By symmetry, we also have the same result for $m \in \N_{0}, k\in \N_{>0}$. As $x\neq \mz$, we thus get that for all $k,m \in \N_0$ that 
	\begin{equation*}
		\bp_{\tilde{\mu}\times \mu} \left( x+\gamma_k=\varphi_m \right) \leq \eps.
	\end{equation*}
	By a union bound over $k=0,\ldots,\lfloor\eps^{-3/5} \rfloor$, we get for all $m \in \N_{0}$ that
	\begin{equation}\label{eq:small k}
		\bp_{\tilde{\mu}\times \mu}\left(\varphi_m=x+\gamma_k \text{ for some $k \in \{0,\ldots, \lfloor\eps^{-3/5} \rfloor \}$}\right) \leq 2 \eps^{-3/5} \eps = 2 \eps^{2/5}.
	\end{equation}
	The inner products of the form $\langle \varphi_t, e_1 \rangle, \langle \gamma_t, e_1 \rangle$ are non-decreasing sequences in $t$. So in particular if for some $m \leq \lfloor \eps^{-1/5} \rfloor$ there exists $k > \lfloor\eps^{-3/5} \rfloor$ for which $\varphi_m=x+\gamma_k$, then
	\begin{equation*}
		\langle \gamma_{\lfloor\eps^{-3/5} \rfloor}, e_1 \rangle
		\leq
		\langle \gamma_k, e_1 \rangle 
		\leq
		\langle x+\gamma_k, e_1 \rangle 
		=
		\langle \varphi_m, e_1 \rangle 
		\leq 
		\langle \varphi_{\lfloor \eps^{-1/5} \rfloor}, e_1 \rangle,
	\end{equation*}
	where we also used the assumption $\langle x, e_1 \rangle \geq 0$. Using this fact we get that for all $m \leq \lfloor \eps^{-1/5} \rfloor$
	\begin{align}\label{eq:insert eps}
		&\bp_{\tilde{\mu}\times \mu}\left(\varphi_m=x+\gamma_k \text{ for some $k > \lfloor\eps^{-3/5} \rfloor$}\right)\\
		& \notag
		\leq
		\bp_{\tilde{\mu}\times \mu}\left(\langle \varphi_m, e_1 \rangle = \langle x+\gamma_k, e_1 \rangle \text{ for some $k > \lfloor\eps^{-3/5} \rfloor$}\right)
		\leq 
		\bp_{\tilde{\mu}\times \mu}\left(\langle \varphi_{\lfloor \eps^{-1/5} \rfloor}, e_1 \rangle \geq \langle \gamma_{\lfloor\eps^{-3/5} \rfloor}, e_1 \rangle \right) .
	\end{align}
	The random variables
	\begin{equation*}
		\langle \varphi_{\lfloor \eps^{-1/5} \rfloor}, e_1 \rangle \text{ and }
		\left(\langle \gamma_{(3k+1)\lfloor \eps^{-1/5} \rfloor} - \gamma_{3k\lfloor \eps^{-1/5} \rfloor}, e_1 \rangle\right)_{k\in \N_{0}}
	\end{equation*}
	are independent and identically distributed under the measure $\bp_{\tilde{\mu}\times \mu}$, which follows directly from the definition of the measure $\bp_{\tilde{\mu}\times \mu}$; further, as $\eps \leq 10^{-10}$ we have that
	\begin{equation*}
		\langle \gamma_{\lfloor\eps^{-3/5} \rfloor}, e_1 \rangle \geq \sum_{k=0}^{\lfloor \eps^{-0.3} \rfloor} \langle \gamma_{(3k+1)\lfloor \eps^{-1/5} \rfloor} - \gamma_{3k\lfloor \eps^{-1/5} \rfloor}, e_1 \rangle.
	\end{equation*}
	Using this observation, we continue \eqref{eq:insert eps} by
	\begin{align*}
		& \bp_{\tilde{\mu}\times \mu}\left(\varphi_m=x+\gamma_k \text{ for some $k > \lfloor\eps^{-3/5} \rfloor$}\right)
		\leq 
		\bp_{\tilde{\mu}\times \mu}\left(\langle \varphi_{\lfloor \eps^{-1/5} \rfloor}, e_1 \rangle \geq \langle \gamma_{\lfloor\eps^{-3/5} \rfloor}, e_1 \rangle \right)\\
		&
		\leq
		\bp_{\tilde{\mu}\times \mu}\left(\langle \varphi_{\lfloor \eps^{-1/5} \rfloor}, e_1 \rangle \geq \sum_{k=0}^{\lfloor \eps^{-0.3} \rfloor} \langle \gamma_{(3k+1)\lfloor \eps^{-1/5} \rfloor} - \gamma_{3k\lfloor \eps^{-1/5} \rfloor}, e_1 \rangle \right) \leq \frac{1}{\lfloor \eps^{-0.3} \rfloor +1} \leq \eps^{0.3}.
	\end{align*}
	Here, we used the general principle that for i.i.d. positive random variables $A,B_1,\ldots,B_K$ one has $\p\left(A \geq B_1+\ldots+B_K\right) \leq \p\left(A > \max( B_1,\ldots,B_K) \right) \leq K^{-1}$. Combining the previous display with \eqref{eq:small k}, we get for $m \in \{0,1,\ldots, \lfloor \eps^{-1/5} \rfloor \}$ that
	\begin{equation}\label{eq:small n}
		\bp_{\tilde{\mu}\times \mu}\left(\varphi_m=x+\gamma_k \text{ for some $k \in \N_0$}\right) \leq \eps^{0.3} + 2\eps^{2/5} \leq 3 \eps^{0.3}.
	\end{equation}
	For each fixed $\ell \in \Z$, there are at most three $k\in \N_0$ with $\langle x+\gamma_k, e_1 \rangle = \ell$. Thus, for $m > \lfloor \eps^{-1/5} \rfloor$, we get through conditioning on $\langle \varphi_m, e_1 \rangle$ and $\gamma_0^\infty$ that
	\begin{align}\label{eq:insertisup}
		& \notag \bp_{\tilde{\mu}\times \mu}\left(\varphi_m=x+\gamma_k \text{ for some $k \in \N_0$}\right)\\
		&
		\notag
		=
		\be_{\tilde{\mu} \times \mu} \left[  \bp_{\tilde{\mu}\times \mu}\left(\varphi_m=x+\gamma_k \text{ for some $k \in \N_0$} | \gamma_0^\infty , \langle \varphi_m, e_1 \rangle \right) \right]\\
		&
		\leq
		3 \sup \bp_{\tilde{\mu}\times \mu}\left(\varphi_m=z  | (\varphi_{3j+1}-\varphi_{3j})_{j\in \N_0} \right)
	\end{align}
	where the last supremum is over all $z\in \Z^d$ and $A_M$-valued sequences $ (\varphi_{3j+1}-\varphi_{3j})_{j\in \N_0}$. By construction, the random variables of the form $(\varphi_{3j+2}-\varphi_{3j+1})$ and $(\varphi_{3j+3}-\varphi_{3j+2})$ are independent of $(\varphi_{3j+1}-\varphi_{3j})_{j\in \N_0}$. Thus
	\begin{align}\label{eq:insertisup2}
		&\notag \sup \bp_{\tilde{\mu}\times \mu}\left(\varphi_m=z  | (\varphi_{3j+1}-\varphi_{3j})_{j\in \N_0} \right)\\
		\leq
		&
		\sup_{z\in \Z^d} \bp_{\tilde{\mu}\times \mu}\left( \sum_{j=0}^{\lfloor \frac{m-2}{3} \rfloor} (\varphi_{3j+2}-\varphi_{3j+1})
		+
		\sum_{j=0}^{\lfloor \frac{m-3}{3} \rfloor} (\varphi_{3j+3}-\varphi_{3j+2})  =z  \right).
	\end{align}
	The increments of the form $\varphi_{3j+3}-\varphi_{3j+2}$ take values in $\{v^+, v^-\}$ and the increments of the form $\varphi_{3j+2}-\varphi_{3j+1}$ take values in $\{u^+, u^-\}$
	As the two vectors $(u_2,\ldots,u_d)$ and $(v_2,\ldots,v_d) \in \Z^{d-1}$ are linearly independent, for each $z\in \Z^d$ there can exist at most one pair $(h,\ell)\in \{0,\ldots,\lfloor \frac{m-2}{3} \rfloor\}\times \{0,\ldots, \lfloor \frac{m-3}{3} \rfloor\}$ such that the equation
	\begin{equation*}
		h\cdot u^+ + \left(\Big\lfloor \frac{m-2}{3} \Big\rfloor+1-h\right)\cdot u^- + \ell \cdot v^+ + \left(\Big\lfloor \frac{m-3}{3} \Big\rfloor+1-\ell\right) \cdot v^- = z
	\end{equation*}
	holds. This implies that if, for fixed $z\in \Z^d$, it is known that
	\begin{equation*}
		\sum_{j=0}^{\lfloor \frac{m-2}{3} \rfloor} (\varphi_{3j+2}-\varphi_{3j+1})
		+
		\sum_{j=0}^{\lfloor \frac{m-3}{3} \rfloor} (\varphi_{3j+3}-\varphi_{3j+2})  =z ,
	\end{equation*}
	then one can reconstruct the values of the individual sums $\sum_{j=0}^{\lfloor \frac{m-2}{3} \rfloor} (\varphi_{3j+2}-\varphi_{3j+1})$ and $\sum_{j=0}^{\lfloor \frac{m-3}{3} \rfloor} (\varphi_{3j+3}-\varphi_{3j+2})$.
	As the two collections of random variables  $(\varphi_{3j+3}-\varphi_{3j+2})_{j\in \N_0}$ and $(\varphi_{3j+2}-\varphi_{3j+1})_{j\in \N_0}$ are independent under the measure $\bp_{\tilde{\mu}\times \mu}$, we thus get that
	\begin{align*}
		& \sup_{z\in \Z^d} \bp_{\tilde{\mu}\times \mu}\left( \sum_{j=0}^{\lfloor \frac{m-2}{3} \rfloor} (\varphi_{3j+2}-\varphi_{3j+1})
		+
		\sum_{j=0}^{\lfloor \frac{m-3}{3} \rfloor} (\varphi_{3j+3}-\varphi_{3j+2})  =z  \right)\\
		=&
		\left(\sup_{z\in \Z^d} \bp_{\tilde{\mu}\times \mu}\left( \sum_{j=0}^{\lfloor \frac{m-2}{3} \rfloor} (\varphi_{3j+2}-\varphi_{3j+1})  =z  \right)\right)
		\left(\sup_{z\in \Z^d} \bp_{\tilde{\mu}\times \mu}\left( \sum_{j=0}^{\lfloor \frac{m-3}{3} \rfloor} (\varphi_{3j+3}-\varphi_{3j+2})  =z  \right)\right) .
	\end{align*}
	Using that
	\begin{align*}
		& \sup_{z \in \Z^d} \bp_{\tilde{\mu}\times \mu} \left( \sum_{j=0}^{k} \left(\varphi_{3j+2}-\varphi_{3j+1}\right)  = z \right) \leq 100 (k+1)^{-\frac{\log(2)}{\log(3)}} \text{ and }\\
		& \sup_{z \in \Z^d} \bp_{\tilde{\mu}\times \mu} \left( \sum_{j=0}^{k} \left(\varphi_{3j+3}-\varphi_{3j+2}\right)  = z \right) \leq 100 (k+1)^{-\frac{\log(2)}{\log(3)}}
	\end{align*}
	by construction, we can insert these inequalities into the above line and obtain that
	\begin{align}\label{eq:insertisup3}
		\notag
		& \sup_{z\in \Z^d} \bp_{\tilde{\mu}\times \mu}\left( \sum_{j=0}^{\lfloor \frac{m-2}{3} \rfloor} (\varphi_{3j+2}-\varphi_{3j+1})
		+
		\sum_{j=0}^{\lfloor \frac{m-3}{3} \rfloor} (\varphi_{3j+3}-\varphi_{3j+2})  =z  \right) \\
		&
		\leq 
		\left(100 \Big\lfloor \frac{m+1}{3} \Big\rfloor^{-\frac{\log(2)}{\log(3)}} \right)
		\left(100 \Big\lfloor \frac{m}{3} \Big\rfloor^{-\frac{\log(2)}{\log(3)}} \right)
		\leq \left(400 m^{-\frac{\log(2)}{\log(3)}}\right)^2 \leq 1600m^{-5/4}.
	\end{align}
	Combining the three inequalities \eqref{eq:insertisup}, \eqref{eq:insertisup2}, and \eqref{eq:insertisup3}, we get that
	\begin{align}\label{eq:big n}
		\bp_{\tilde{\mu}\times \mu}\left(\varphi_m=x+\gamma_k \text{ for some $k \in \N_0$}\right) \leq 5000 m^{-5/4}.
	\end{align}
	Remember that this inequality holds for $m > \lfloor \eps^{-1/5} \rfloor$.
	Combining the upper bound for small $m$ ($m\leq \lfloor \eps^{-1/5} \rfloor$) \eqref{eq:small n} and the upper bound for big $m$ ($m> \lfloor \eps^{-1/5} \rfloor$) \eqref{eq:big n}, we get via a union bound that
	\begin{align*}
		& \bp_{\tilde{\mu}\times \mu}\left(\varphi_m=x+\gamma_k \text{ for some $k,m \in \N_0$}\right) \leq \sum_{m=0}^{\lfloor \eps^{-1/5} \rfloor} 3 \eps^{0.3} + \sum_{m=\lfloor \eps^{-1/5} \rfloor + 1}^{\infty} 5000 m^{-5/4}\\
		& \leq 4 \eps^{0.1} + \int_{\eps^{-1/5}/2}^{\infty} 5000s^{-5/4} \md s
		= 4 \eps^{0.1} + 20000 \left(\eps^{-1/5}/2\right)^{-1/4}
		\leq 4 \eps^{0.1} + 25000 \eps^{0.05} \leq 25004 \eps^{0.05}.
	\end{align*}
\end{proof}

\subsection{Proof of Theorem \ref{theo:isotropic}: Measures on random paths}\label{sec:proof isot}

In this section, we construct two sets of measures $M_2$, respectively $M_3$, which are measures on self-avoiding walks starting at the origin in $\Z^d$ for $d\geq 4$, respectively $\Z^3$. The set of measures $M_2$ is used for the proof of Theorem \ref{theo:isotropic} in dimensions $d\geq 4$ and is constructed in Section \ref{subsubsec:random paths 4 or higher}. The set of measures $M_3$ is used for the proof of Theorem \ref{theo:isotropic} in dimension $d = 3$ and is constructed in Section \ref{subsubsec:random paths 3}.

\subsubsection{Random paths in dimensions $d \geq 4$}\label{subsubsec:random paths 4 or higher}

In this section, we prove Theorem \ref{theo:isotropic} for dimensions $d\geq 4$. Throughout this section, we will always assume that $(p_x)_{x\in \Z^d\setminus \{\mz\}}$ satisfies \eqref{eq:irreduc}, \eqref{eq:isotropy}, and $\sum_{x\in \Z^d\setminus\{\mz\}}p_x \in (T(d),\infty)$.
Define the four vectors $w_1,\ldots,w_4 \in \{0,1\}^4$ by
\begin{align*}
	&\langle w_1, e_i \rangle = \begin{cases}
		1 & \text{ if } i \in  \{1, \ldots, \lfloor \frac{d}{4} \rfloor\}\\
		0 & \text{ otherwise}
	\end{cases},\\
	&\langle w_2, e_i \rangle = \begin{cases}
		1 & \text{ if } i \in  \{\lfloor \frac{d}{4} \rfloor + 1, \ldots, 2\lfloor \frac{d}{4} \rfloor\}\\
		0 & \text{ otherwise}
	\end{cases},\\
	&\langle w_3, e_i \rangle = \begin{cases}
		1 & \text{ if } i \in  \{2\lfloor \frac{d}{4} \rfloor + 1, \ldots, 3\lfloor \frac{d}{4} \rfloor\}\\
		0 & \text{ otherwise}
	\end{cases}, \text{ and }\\
	&\langle w_4, e_i \rangle = \begin{cases}
		1 & \text{ if } i \in  \{3\lfloor \frac{d}{4} \rfloor + 1, \ldots, d \}\\
		0 & \text{ otherwise}
	\end{cases}.\\
\end{align*}
Note that $w_1+w_2+w_3+w_4$ is the all-ones vector and that each of $w_1,w_2,w_3,w_4$ contains at least $\lfloor \frac{d}{4} \rfloor \geq \frac{d}{7}$ many `one-entries'. Through these four vectors, we define the four sets $A_1,A_2,A_3,A_4 \subset \Z^d$ by
\begin{align*}
	& A_k = \left\{x\in \Z^d : \langle x, w_k \rangle \neq 0, \langle x, w_4 \rangle \geq 0 \right\} \text{ for } k=1,2,3, \text{ and }\\
	& A_4 =  \left\{x\in \Z^d : \langle x, w_4 \rangle > 0 \right\} .
\end{align*}
We start with a proof of the following elementary claim.

\begin{claim}\label{claim:3.4}
	For each of the sets $A_1,A_2,A_3$, and $A_4$ one has
	\begin{equation}\label{eq:}
		\sum_{x\in A_k} p_x \geq \frac{T(d)}{28} .
	\end{equation}
\end{claim}
\begin{proof}
	Let $Z$ be a $\Z^d$-valued random variable with probability mass function $\p(Z=z)=\frac{p_z}{\sum_{x\in \Z^d} p_x}$. Thus we need to show that $\p(Z\in A_k)\geq \tfrac{1}{28}$ for $k=1,2,3,4$. For $k\in \{1,2,3,4\}$, let $I_k\subset \{1,\ldots,d\}$ be the set of indices $i \in \{1,\ldots,d\}$ for which $\langle e_i, w_k \rangle \neq 0$. As $Z\neq \mz$, there exists $i\in \{1,\ldots,d\}$ with $\langle e_i,Z\rangle \neq 0$. Symmetry in the different coordinates yields
	\begin{equation*}
		\p(\langle e_i, Z \rangle \neq 0 \text{ for some $i\in I_k$}) \geq \frac{|I_k|}{d} \geq \frac{1}{7}.
	\end{equation*}
	We write $\langle Z, w_k \rangle = \sum_{i\in I_k} \sigma_i |\langle e_i , Z \rangle|$ with $\sigma_i = \text{sgn}(\langle e_i , Z \rangle)$. Conditioned on the event that $|\langle e_i , Z \rangle|=\ell$ for some $\ell \in \N_{>0}$, the random variable $\langle e_i , Z \rangle$ is equally likely to be either $+\ell$ or $-\ell$. So given the $\sigma$-algebra $ \cF=\sigma\left( (|\langle e_i , Z \rangle|)_{i\in \{1,\ldots,d\}}  \right)$, the random variables $(\sigma_i)_{i\in \{1,\ldots,d\}}$ are independent and uniformly distributed on $\{-1,+1\}$. In particular, this allows us to apply the results of Lemma \ref{lem1} for the sum $\sum_{i\in I_k} \sigma_i |\langle e_i , Z \rangle|$.
	Applying this to $k=1,2,3$, we get that
	\begin{align*}
		&\p(Z\in A_k) = \p(\langle Z, w_k \rangle \neq 0, \langle Z, w_4 \rangle \geq 0|\langle e_i, Z \rangle \neq 0 \text{ for some $i\in I_k$}) \p(\langle e_i, Z \rangle \neq 0 \text{ for some $i\in I_k$})\\
		& \geq \p(\langle Z, w_k \rangle \neq 0 |\langle e_i, Z \rangle \neq 0 \text{ for some $i\in I_k$})
		\cdot
		\p(\langle Z, w_4 \rangle \geq 0 |\langle e_i, Z \rangle \neq 0 \text{ for some $i\in I_k$}) \cdot \frac{1}{7}\\
		& \geq \frac{1}{2} \cdot \frac{1}{2} \cdot \frac{1}{7} = \frac{1}{28},
	\end{align*}
	where we used \eqref{eq:one quarter bound1} and \eqref{eq:one quarter bound2} for the last inequality. For the case $k=4$, we get that
	\begin{align*}
		& \p(Z\in A_4) = \p( \langle x, w_4 \rangle > 0|\langle e_i, Z \rangle \neq 0 \text{ for some $i\in I_4$}) \p(\langle e_i, Z \rangle \neq 0 \text{ for some $i\in I_4$})\\
		& \geq \frac{1}{4} \cdot \frac{1}{7} = \frac{1}{28},
	\end{align*}
	where we used \eqref{eq:one quarter bound} for the last inequality.
\end{proof}

On the set $A_k$, we define the probability measure $\psi_k$ by
\begin{equation*}
	\psi_k(z)= \frac{p_z}{\sum_{x\in A_k} p_x}.
\end{equation*}
We write $\mu$ for the probability measure on infinite sequences $(\zeta_n)_{n\in \N_0} \in (\Z^d)^{\N_0}$ for which the increments $(\zeta_{k+1}-\zeta_k)_{k\in \N_0}$ are independent random variables such that
\begin{align*}
	\zeta_0=\mz \text{ almost surely and }  \mu \left( \zeta_{k+1}-\zeta_k = z \right) = 
	\begin{cases}
		\psi_{4}(z) & \text{ if $k \in \{0,2,4\} \mod 6$},\\
		\psi_{1}(z) & \text{ if $k = 1 \mod 6$},\\
		\psi_{2}(z) & \text{ if $k = 3 \mod 6$},\\
		\psi_{3}(z) & \text{ if $k = 5 \mod 6$},\\
	\end{cases}
\end{align*}
for $z\in \Z^d$. For such a path $(\zeta_k)_{k\in \N_0}$, one has by construction that $\langle \zeta_{k+1}-\zeta_k, w_4 \rangle \geq 0$ for all $k\in \N_0$, with a strict inequality for even $k$. So in particular this implies that the path $(\zeta_k)_{k\in \N_0}$ is almost surely self-avoiding. Let $M_2$ be the set of $6$ different probability measures which are the distributions of the sequence $(\zeta_k-\zeta_N)_{k\geq N}$ for $N\in \N_0$. Note that this distribution does only depend on $N\mod 6$, so there are only $6$ such distributions. For $z\in \Z^d$ and an infinite path $\varphi=(\varphi_0,\varphi_1,\ldots)$ starting at the origin, we define a path $(z\varphi):\{-1,0,1,\ldots\} \to \Z^d$ by
\begin{equation*}
	(z\varphi)_{-1}^{\infty} =(\varphi_{-1},\varphi_0,\varphi_1,\ldots) \coloneqq (z,\varphi_0,\varphi_1,\ldots).
\end{equation*}
Note that this path is not necessarily self-avoiding as it can be possible that $z=\varphi_1$ or $z=\varphi_0=\mz$. However, in this case, we do {\sl not} count the edge twice in the definition of the overlap, i.e., if $z=\varphi_1$ or $z=\mz$ we have nevertheless that
\begin{equation*}
	\ov(x+\gamma_0^n, (z\varphi)_{-1}^{\infty})= \ov(x+\gamma_0^n, \varphi_{0}^{\infty})
\end{equation*}
for any path $x+\gamma_0^n$.
The necessity of the additional parameter $z\in \Z^d$ for the proof will be discussed below.
We now define the number
\begin{equation*}
	W_{n} \coloneqq \sup_{\hat{\mu},\tilde{\mu}\in M_2, x \in \Z^d, z \in \Z^d} \be_{\hat{\mu} \times \tilde{\mu}}
	\left[\ov(x+\gamma_0^n,(z\varphi)_{-1}^{\infty})\right].
\end{equation*}
By truncation, we can without loss of generality assume that $\inf \{p_x : p_x>0\} >0$, which yields that $W_n<\infty$. In Lemma \ref{lem:Wn def 2} below, we derive a bootstrapping-type inequality for $W_n$, which implies that $W_n \leq \tfrac{3}{2} < \infty$. As this bound is uniform in $n$, we thus get that
\begin{equation*}
	\be_{\mu \times \mu} \left[\ov(\gamma_0^\infty, \varphi_0^\infty)\right] 
	= 
	\sup_{n\in \N} \be_{\mu \times \mu} \left[\ov(\gamma_0^n, \varphi_0^n)\right] \leq \sup_{n\in \N} W_n \leq \frac{3}{2},
\end{equation*}
and Proposition \ref{prop:measure on paths} implies that
\begin{equation*}
	\p(|K_\mz|=\infty) \geq \be_{\mu \times \mu} \left[\ov(\gamma_0^\infty, \varphi_0^\infty)\right]^{-1} \geq \frac{2}{3} > 1-\frac{1}{e}.
\end{equation*}

\subsubsection{Random paths in dimension $d = 3$}\label{subsubsec:random paths 3}

For $i\in \N$, let $(X_1^i, X_2^i, X_3^i)$ be a $\Z^3$-valued random variable supported on the set $A_1=\left\{x\in \Z^3: \langle x,e_1 \rangle > 0, \langle x,e_3 \rangle \geq 0\right\}$ with probability mass function
\begin{equation*}
	\p((X_1^i, X_2^i, X_3^i)=x)= \frac{p_x}{\sum_{y\in A_1} p_y} .
\end{equation*}
Furthermore, let $((X_1^i, X_2^i, X_3^i))_{i\in \N}$ be independent. Let $((Y_1^i, Y_2^i, Y_3^i))_{i\in \N}$ be i.i.d. random variables supported on the set $A_2=\left\{x\in \Z^3: \langle x,e_2 \rangle \neq 0, \langle x,e_3 \rangle \geq 0\right\}$ with probability mass function
\begin{equation*}
	\p((Y_1^i, Y_2^i, Z_3^i)=x)= \frac{p_x}{\sum_{y\in A_2} p_y} ,
\end{equation*}
and let $((Z_1^i, Z_2^i, Z_3^i))_{i\in \N}$ be i.i.d. random variables supported on the set $A_3=\left\{x\in \Z^3: \langle x,e_3 \rangle > 0\right\}$ with probability mass function
\begin{equation*}
	\p((Z_1^i, Z_2^i, Z_3^i)=x)= \frac{p_x}{\sum_{y\in A_3} p_y} .
\end{equation*}
First, we prove the following elementary claim.

\begin{claim}\label{claim:a sets}
	For each of the sets $A_1, A_2$, and $A_3$ one has
	\begin{equation}\label{eq:sets 1-2-3 in dim 3}
		\sum_{x\in A_k} p_x \geq \frac{T(3)}{12}.
	\end{equation}
\end{claim}
\begin{proof}
	We start with the set $A_1$.
	Let $Z$ be a $\Z^3$-valued random variable with probability mass function $\p(Z=z)=\frac{p_z}{\sum_{x\in \Z^3 \setminus \{\mz\}} p_x}$ for $z\in \Z^3\setminus \{\mz\}$. Thus we need to show that $\p(Z\in A_1)\geq \tfrac{1}{12}$. As $Z\neq \mz$ almost surely, we get by symmetry in the different coordinates that
	\begin{equation*}
		\p(\langle e_i, Z \rangle \neq 0 ) \geq \frac{1}{3} \text{ for } i=1,2,3.
	\end{equation*}
	Conditioned on the event that $|\langle e_i , Z \rangle|=\ell$ for some $\ell \in \N_{>0}$, the random variable $\langle e_i , Z \rangle$ is equally likely to be either $+\ell$ or $-\ell$. Conditioned on $\langle Z, e_1 \rangle \neq 0$ and $\langle Z, e_3 \rangle \geq 0$, the inner product $\langle Z, e_1 \rangle$ is a random variable that is symmetric around $0$. Further,  conditioned on $\langle Z, e_1 \rangle \neq 0$, the random variable $\langle Z, e_3 \rangle$ is also symmetric around $0$.
	Thus we have for the set $A_1$ that
	\begin{align*}
		&\p(Z\in A_1) = \p(\langle Z, e_1 \rangle > 0, \langle Z, e_3 \rangle \geq 0) = \\
		&
		=
		\p(\langle Z, e_1 \rangle > 0 | \langle Z, e_1 \rangle \neq 0, \langle Z, e_3 \rangle \geq 0) \cdot
		\p(\langle Z, e_3 \rangle \geq 0 | \langle Z, e_1 \rangle \neq 0)
		\cdot
		\p(\langle Z, e_1 \rangle \neq 0)\\
		&
		\geq \frac{1}{2\cdot 2 \cdot 3} = \frac{1}{12}.
	\end{align*}
	For $k=2,3$ it is easy to check by \eqref{eq:isotropy} that 
	\begin{equation*}
		\sum_{x\in A_k} p_x \geq \sum_{x\in A_1} p_x \geq \frac{T(3)}{12}.
	\end{equation*}
\end{proof}

\noindent
Let $S=(S_m=\sum_{i=1}^{m} \sigma_i)_{m\in \N_0}$ be the stochastic process constructed in Section \ref{sec:unpredictable}. Recall, that the important properties of this process were $\sigma_i \in \{-1,+1\}$ and the upper bound on its predictability profile \eqref{eq:predict 100 old}, i.e., the inequality $\pre_k(S) \leq 100 k^{-\frac{\log(2)}{\log(3)}}$. Define $((\tilde{X}_1^i, \tilde{X}_2^i, \tilde{X}_3^i))_{i\in \N}$ by 
\begin{equation*}
	(\tilde{X}_1^i, \tilde{X}_2^i, \tilde{X}_3^i) = (\sigma_i X_1^i, X_2^i, X_3^i).
\end{equation*}
Furthermore, let $(S_m)_{m\in \N_0}, (({X}_1^i, {X}_2^i, {X}_3^i))_{i\in \N}, (Y_1^i, Y_2^i, Y_3^i))_{i\in \N}, \text{ and } ((Z_1^i, Z_2^i, Z_3^i))_{i\in \N}$ be independent. 
Let $(\zeta_k)_{k \in \N_0}\in (\Z^3)^{\N_0}$ be the stochastic process defined by
\begin{align}\label{eq:zeta def}
	\zeta_0=\mz \text{ and }  \zeta_{k+1}-\zeta_k = 
	\begin{cases}
		(Z_1^i, Z_2^i, Z_3^i) & \text{ if $k = 2i-2$},\\
		(\tilde{X}_1^i, \tilde{X}_2^i, \tilde{X}_3^i) & \text{ if $k=4i-3$},\\
		(Y_1^i, Y_2^i, Y_3^i) & \text{ if $k = 4i-1$}.
	\end{cases}
\end{align}
Note that for each $k\in \N_0$, exactly one of the conditions above ($k=2i-2, k=4i-3$, or $k=4i-1$) is satisfied, such that the path $(\zeta_k)_{k\in \N_0}$ is uniquely defined by the above rule. Write $\mu$ for the distribution of $(\zeta_k)_{k\in \N_0}$. For such a path $(\zeta_k)_{k\in \N_0}$, one has by construction that $\langle \zeta_{k+1}-\zeta_k, e_3 \rangle \geq 0$ for all $k\in \N_0$, with a strict inequality for even $k$. So in particular this implies that the path $(\zeta_k)_{k\in \N_0}$ is almost surely self-avoiding. Let $M_3$ be the set of probability measures which are the distributions of the sequence $(\zeta_k-\zeta_N)_{k\geq N}$ conditioned on $(\zeta_0,\ldots,\zeta_N)$, for $N\in \N_0$.\\

For $z\in \Z^3$ and an infinite path $(\varphi_0,\varphi_1,\ldots)$ starting at the origin, we define the path
\begin{equation*}
	(z\varphi)_{-1}^{\infty}=(\varphi_{-1},\varphi_0,\varphi_1,\ldots) \coloneqq (z,\varphi_0,\varphi_1,\ldots).
\end{equation*}
The same notation was used in Section \ref{subsubsec:random paths 4 or higher} above. See also the discussion below \eqref{eq:ineq} for the necessity of the parameter $z\in \Z^3$.
We now define the number
\begin{equation*}
	W_{n}^\prime \coloneqq \sup_{\hat{\mu},\tilde{\mu}\in M_3, x,z\in \Z^3} \be_{\hat{\mu} \times \tilde{\mu}}
	\left[\ov(x+\gamma_0^n,(z\varphi)_{-1}^{\infty})\right].
\end{equation*}
By truncation, we can without loss of generality assume that $\inf \{p_x : p_x>0\} >0$, which yields that $W_n^\prime <\infty$. We now derive a bootstrapping-type inequality for $W_n^\prime$, which implies that $\sup_n W_n^\prime \leq \frac{3}{2}$. As this bound is uniform in $n$, we thus get that
\begin{equation*}
	\be_{\mu \times \mu} \left[\ov(\gamma_0^\infty, \varphi_0^\infty)\right] 
	= 
	\sup_{n\in \N} \be_{\mu \times \mu} \left[\ov(\gamma_0^n, \varphi_0^n)\right] \leq \sup_{n\in \N} W_n^\prime \leq \frac{3}{2},
\end{equation*}
and Proposition \ref{prop:measure on paths} implies that
\begin{equation*}
	\p(|K_\mz|=\infty) \geq \be_{\mu \times \mu} \left[\ov(\gamma_0^\infty, \varphi_0^\infty)\right]^{-1} \geq \frac{2}{3} > 1-\frac{1}{e}.
\end{equation*}

\subsection{The bootstrapping lemmas}\label{sec:boots}

In this section, we prove the key lemmas (Lemmas \ref{lem:Wn def 2} and \ref{lem:Wn def 3}) which imply that the expected weighted overlap satisfies $\be_{\mu \times \mu} \left[\ov(\gamma_0^\infty, \varphi_0^\infty)\right] \leq \frac{3}{2} < \infty$. For the proofs of these lemmas, we also need some upper bounds on the probability that two random paths intersect. These upper bounds are proven in Section \ref{subsubsec:intersection 4 or higher} for dimensions $d\geq 4$, respectively in Section \ref{subsubsec:intersection 3} for dimension $d=3$.

\begin{lemma}\label{lem:Wn def 2}
	For the quantity $W_n$ defined by
	\begin{equation}\label{eq:Wn defini}
		W_{n} = \sup_{\hat{\mu},\tilde{\mu}\in M_2, x,z\in \Z^d} \be_{\hat{\mu} \times \tilde{\mu}}
		\left[\ov(x+\gamma_0^n,(z\varphi)_{-1}^{\infty})\right]
	\end{equation}
	one has $W_n \leq \frac{3}{2}$ for all $n\in \N$.
\end{lemma}

\begin{lemma}\label{lem:Wn def 3}
	For the quantity $W_n^\prime$ defined by
	\begin{equation}\label{eq:Wn defini3}
		W_{n}^\prime = \sup_{\hat{\mu},\tilde{\mu}\in M_3, x,z\in \Z^3} \be_{\hat{\mu} \times \tilde{\mu}}
		\left[\ov(x+\gamma_0^n,(z\varphi)_{-1}^{\infty})\right]
	\end{equation}
	one has $W_n^\prime \leq \frac{3}{2}$ for all $n\in \N$.
\end{lemma}

The proofs of both Lemmas work exactly the same, so we prove them together.

\begin{proof}[Proof of Lemma \ref{lem:Wn def 2} and Lemma \ref{lem:Wn def 3}]
	Let $j\in \{2,3\}$ and let $\widetilde{W_n}=W_n$ for $j=2$, respectively $\widetilde{W_n}=W_n^\prime$ for $j=3$. Note that $d\geq 4$ for $j=2$, respectively $d=3$ for $j=3$.
	We will show that for any $\hat{\mu},\tilde{\mu}\in M_j, x,z \in \Z^d$ one has
	\begin{equation}\label{eq:bootstrapping ineq}
		\be_{\hat{\mu} \times \tilde{\mu}} \left[\ov(x+\gamma_0^n,(z\varphi)_{-1}^{\infty})\right] \leq 1+ \frac{\widetilde{W_n}}{3},
	\end{equation}
	which implies that $\widetilde{W_n} \leq 1+ \frac{\widetilde{W_n}}{3}$, or $\widetilde{W_n} \leq \frac{3}{2}$; note that one can safely rearrange these inequalities as $\widetilde{W_n} < \infty$. For this, we distinguish three cases: $x \notin \{\mz,z\}, x=z$, and $x=\mz$. We can without loss of generality assume that $z\neq \mz$, because if $z=\mz$ and $v\in \Z^d\setminus \{\mz\}$ is some arbitrary vector, one has by definition of the weighted overlap that
	\begin{equation*}
		\ov(x+\gamma_0^n,(z\varphi)_{-1}^{\infty}) = \ov(x+\gamma_0^n,\varphi_{0}^{\infty}) 
		\leq
		\ov(x+\gamma_0^n,(v\varphi)_{-1}^{\infty}).
	\end{equation*}
	Thus, when proving \eqref{eq:bootstrapping ineq}, we can restrict to the case where $z\neq \mz$.\\
	
	\noindent
	\hypertarget{target1}{\textbf{Case 1:}}
	We start with the case $x\notin \{\mz,z\}$.
	Define the random variables $K=K(x+\gamma_0^n,(z\varphi)_{-1}^{\infty})$ and $N=N(x+\gamma_0^n,(z\varphi)_{-1}^{\infty})$ by
	\begin{align*}
		& K(x+\gamma_0^n,(z\varphi)_{-1}^{\infty}) = \inf \left\{ m\in \{0,\ldots,n\}  : x+\gamma_m=\varphi_\ell \text{ for some } \ell \in \{-1,0,1,2,\ldots\} \right\},\\
		& N(x+\gamma_0^n,(z\varphi)_{-1}^{\infty}) = \inf \left\{ \ell \in \{-1,0,1,2,\ldots\} : x+\gamma_{\{K(x+\gamma_0^n,(z\varphi)_{-1}^{\infty})\}}=\varphi_{\ell}  \right\}.
	\end{align*}
	Note that $K \in \{0,\ldots,n\}\cup \{+\infty\}$ and $N\in \{-1,0,\ldots\}\cup \{+\infty\}$ by definition. 
	Further, again by definition of the random variables $N$ and $K$, on the event where $K < \infty$, we have that
	\begin{align}\label{eq:ineq}
		\ov(x+\gamma_0^n,(z\varphi)_{-1}^{\infty}) = \ov(x+\gamma_{K}^{n},(z\varphi)_{N-1}^\infty),
	\end{align}
	where we define $(z\varphi)_{-2} \coloneqq (z\varphi)_{-1}=z$ for the case where $N=-1$. For the case where $K=+\infty$, we define $x+\gamma_{K}^{n}=\emptyset$, which implies that the overlap equals $+1$. Equality \eqref{eq:ineq} holds, as the inner products $\langle \gamma_j, w_4 \rangle$ and $\langle \varphi_j, w_4 \rangle$ are non-decreasing in $j \in \N_0$. Additionally, $\langle \varphi_j, w_4 \rangle$ needs to be strictly increasing on every second term. As $x+\gamma_K=(z\varphi)_N$, the path $x+\gamma_K^\infty$ can only have joint edges with $(z\varphi)_{N-1}^\infty$, but not with $(z\varphi)_{-1}^{N-1}$. On the other hand, the path segment $x+\gamma_{0}^{K}$ can have no joint edges with $(z\varphi)_{-1}^\infty$ by definition of $K$. However, it is possible that $\{x+\gamma_K,x+\gamma_{K+1}\} = \{\varphi_N,\varphi_{N-1}\}$, which is also the reason why the additional parameter $z\in \Z^d$ in the definition of $\widetilde{W_n}$ (\eqref{eq:Wn defini}, respectively \eqref{eq:Wn defini3}) is necessary. Note that the two paths $(z\varphi)_{-1}^\infty$ and $x+\gamma_0^\infty$ traverse the edge $\{\varphi_N,\varphi_{N-1}\}$ in opposite directions in this case. Defining the sequences $\tilde{\gamma}=(\tilde{\gamma}_m=\gamma_{m+K}-\gamma_K)_{m \in \N_0}$ and $\tilde{\varphi}=(\tilde{\varphi}_\ell = (z\varphi)_{\ell+N}-(z\varphi)_N)_{\ell \in \N_0}$ gives that
	\begin{equation}\label{eq:ineq2}
		\ov(x+\gamma_{K}^{n},(z\varphi)_{N-1}^\infty) \leq \ov \left( \tilde{\gamma}_{0}^{n}, (((z\varphi)_{N-1}-(z\varphi)_{N})\tilde{\varphi})_{-1}^{\infty} \right)
	\end{equation}
	Note that for the case $N=-1$, this does not cause any problems, as $(z\varphi)_{-1}=(z\varphi)_{-2}=z$ and thus
	\begin{equation*}
		\ov \left( \tilde{\gamma}_{0}^{n}, (((z\varphi)_{N-1}-(z\varphi)_{N})\tilde{\varphi})_{-1}^{\infty} \right)
		=
		\ov \left( \tilde{\gamma}_{0}^{n}, (\mz\tilde{\varphi})_{-1}^{\infty} \right)
		=
		\ov \left( \tilde{\gamma}_{0}^{n}, \tilde{\varphi}_{0}^{\infty} \right).
	\end{equation*}
	The random variables $N$ and $K$ are measurable with respect to the $\sigma$-algebra generated by $\gamma_0^K$ and $\varphi_0^N$. Conditioned on the $\sigma$-Algebra $\sigma(\gamma_0^K, \varphi_0^N)$, the sequences $(\tilde{\gamma}_\ell)_{\ell \in \N_0}$ and $(\tilde{\varphi}_\ell)_{\ell \in \N_0}$ are still independent and distributed according to measures in $M_j$. Furthermore, $((z\varphi)_{N-1}-(z\varphi)_{N})$ is measurable with respect to $\sigma(\gamma_0^K, \varphi_0^N)$. This yields that
	\begin{align}\label{eq:ineq3}
		\notag &
		\be_{\hat{\mu} \times \tilde{\mu}}  \left[\ov(\tilde{\gamma}_0^n,(((z\varphi)_{N-1}-(z\varphi)_{N}) \tilde{\varphi})_{-1}^\infty)| \sigma(\gamma_0^K, \varphi_0^N) \right]\\
		&
		\leq
		\sup_{\overline{\mu}, \mu^\prime \in M_j, v \in \Z^d} 
		\be_{\overline{\mu} \times \mu^\prime} \left[\ov\left( \gamma_0^n, (v\varphi)_{-1}^{\infty}\right)\right] \leq \widetilde{W_n}.
	\end{align}
	So in particular, in the light of \eqref{eq:ineq}, \eqref{eq:ineq2}, and \eqref{eq:ineq3}, we get that
	\begin{align*}
		&
		\be_{\hat{\mu} \times \tilde{\mu}} \left[\ov(x+\gamma_{0}^{n},(z\varphi)_{-1}^\infty) | K<\infty\right]
		\leq
		\be_{\hat{\mu} \times \tilde{\mu}} \left[\ov(\tilde{\gamma}_0^n,(((z\varphi)_{N-1}-(z\varphi)_{N}) \tilde{\varphi})_{-1}^\infty) | K<\infty\right]
		\\
		&
		=
		\be_{\hat{\mu} \times \tilde{\mu}} \left[  \be_{\hat{\mu} \times \tilde{\mu}}  \left[\ov(\tilde{\gamma}_0^n,(((z\varphi)_{N-1}-(z\varphi)_{N}) \tilde{\varphi})_{-1}^\infty)| \sigma(\gamma_0^K, \varphi_0^N) \right] | K<\infty\right]
		\leq
		\be_{\hat{\mu} \times \tilde{\mu}} \left[ \widetilde{W_n} \right] = \widetilde{W_n},
	\end{align*}
	which implies that
	\begin{align*}
		&\be_{\hat{\mu} \times \tilde{\mu}} \left[\ov(x+\gamma_0^n,(z\varphi)_{-1}^{\infty})\right]
		\leq
		1 + 
		\be_{\hat{\mu} \times \tilde{\mu}} \left[\ov(x+\gamma_{K}^{\infty},\varphi_{N-1}^\infty) \mathbbm{1}_{\{K<\infty\}}\right]\\
		&
		\leq
		1 + 
		\bp_{\hat{\mu} \times \tilde{\mu}}(K<\infty)
		\be_{\hat{\mu} \times \tilde{\mu}} \left[\ov(x+\gamma_{K}^{\infty},\varphi_{N-1}^\infty) | K<\infty\right] \leq 1+ \frac{\widetilde{W_n}}{4},
	\end{align*}
	where the last inequality follows from Lemma \ref{lemma:hitting d=4} if $j=2, d\geq 4$, respectively Lemma \ref{lemma:hitting d=3} if $j=3, d=3$.\\
	
	\noindent
	\hypertarget{target2}{\textbf{Case 2:}}
	Next, we consider the case $x=z$. Here we have that
	\begin{align}\label{eq:ineq start z}
		\ov(x+\gamma_0^n,(z\varphi)_{-1}^{\infty}) \leq \frac{1}{p_z} \ov(x+\gamma_1^n,\varphi_0^\infty) \mathbbm{1}_{\{\gamma_1 = -z\}} +
		\ov(x+\gamma_0^n,\varphi_0^\infty) .
	\end{align}
	We now bound each of the terms on the right-hand side of \eqref{eq:ineq start z} in expectation, starting with the second summand. Let $v\in \Z^d \setminus \{\mz, x\}$ be deterministically chosen. Then
	\begin{equation*}
		\ov(x+\gamma_0^n,\varphi_0^\infty) \leq \ov(x+\gamma_0^n,(v\varphi)_{-1}^\infty).
	\end{equation*}
	Taking expectations and the analysis of \eqref{eq:bootstrapping ineq} in \hyperlink{target1}{Case 1} above (where we assumed $x\notin \{\mz,z\}$) shows that
	\begin{equation}\label{eq:second summand}
		\be_{\hat{\mu} \times \tilde{\mu}}\left[
		\ov(x+\gamma_0^n,\varphi_0^\infty)\right] \leq
		\be_{\hat{\mu} \times \tilde{\mu}}\left[ \ov(x+\gamma_0^n,(v\varphi)_{-1}^\infty)\right] \leq 1 + \frac{\widetilde{W_n}}{4}.
	\end{equation}
	
	
	For the first summand of \eqref{eq:ineq start z}, we have that
	\begin{align}\label{eq:first summand}
		\notag
		& \be_{\hat{\mu} \times \tilde{\mu}} \left[ \frac{1}{p_z} \ov(x+\gamma_1^n,\varphi_0^\infty) \mathbbm{1}_{\{\gamma_1 = -z\}} \right] 
		=
		\frac{1}{p_z} \bp_{\hat{\mu} \times \tilde{\mu}}(\gamma_1=-z)
		\be_{\hat{\mu} \times \tilde{\mu}} \left[ \ov(x+\gamma_1^n,\varphi_0^\infty) | \gamma_1 = -z \right]\\
		&
		\leq \frac{28 p_z}{p_z T(d)} \be_{\hat{\mu} \times \tilde{\mu}} \left[ \ov(x+\gamma_1^n,\varphi_0^\infty) | \gamma_1 = -z \right]
	\end{align}
	where the last inequality follows from the fact that $\hat{\mu} \in M_j$ and Claim \ref{claim:3.4} for $j=2$, respectively Claim \ref{claim:a sets} for $j=3$. For the second term above $\left( \be_{\hat{\mu} \times \tilde{\mu}} \left[ \ov(x+\gamma_1^n,\varphi_0^\infty) | \gamma_1 = -z \right]\right)$, note that the sequence $\Delta(\gamma_1^n) = \left(\gamma^\prime_k\right)_{k\in \{0,\ldots,n-1\}}$ defined by
	\begin{equation*}
		\gamma^\prime_k = \gamma_{k+1}-\gamma_{1} 
	\end{equation*}
	satisfies
	\begin{equation*}
		\ov(x+\gamma_1^n,\varphi_0^\infty) = \ov(x-z+\gamma^{\prime (n-1)}_{0},\varphi_0^\infty)
		=
		\ov(\gamma^{\prime (n-1)}_{0},\varphi_0^\infty)
		=
		\ov(\Delta(\gamma_1^n),\varphi_0^\infty).
	\end{equation*}
	Conditioned on $\gamma_1 = -z$, the sequence $\left(\gamma^\prime_k = \gamma_{k+1}-\gamma_{1}\right)_{k\in \N_0}$ is still distributed according to a measure in $M_j$. Call this measure $\mu^\prime$. Then
	\begin{align*}
		&\be_{\hat{\mu} \times \tilde{\mu}} \left[ \ov(x+\gamma_1^n,\varphi_0^\infty) | \gamma_1 = -z \right]
		\leq
		\be_{\hat{\mu} \times \tilde{\mu}} \left[ \ov(\Delta(\gamma_1^n),\varphi_0^\infty) | \gamma_1=-z \right]\\
		&
		=
		\be_{\mu^\prime \times \tilde{\mu}} \left[\ov( \gamma_0^{n-1},\varphi_0^\infty) \right] \leq \widetilde{W_n}.
	\end{align*}
	Inserting this into \eqref{eq:first summand}, we get that
	\begin{equation*}
		\be_{\hat{\mu} \times \tilde{\mu}} \left[ \frac{1}{p_z} \ov(x+\gamma_1^n,\varphi_0^\infty) \mathbbm{1}_{\{\gamma_1 = -z\}} \right] \leq \frac{28}{T(d)} \widetilde{W_n} \leq \frac{\widetilde{W_n}}{100}.
	\end{equation*}
	Taking expectations in \eqref{eq:ineq start z}, we get together with \eqref{eq:second summand} that
	\begin{equation*}
		\be_{\hat{\mu} \times \tilde{\mu}} \left[\ov(x+\gamma_0^n,(z\varphi)_{-1}^{\infty})\right] \leq 1+ \frac{\widetilde{W_n}}{4} + \frac{\widetilde{W_n}}{100} \leq 1+ \frac{\widetilde{W_n}}{3}
	\end{equation*}
	for the case $x=z$. \\

	\noindent
	\textbf{Case 3:} Finally, we consider the case where $x=\mz$.
	Here we have that
	\begin{align}
		\notag
		\ov(\gamma_0^n,(z\varphi)_{-1}^{\infty}) & \leq \frac{1}{p_z} \ov(\gamma_1^n,\varphi_0^\infty) \mathbbm{1}_{\{\gamma_1 = z\}}  + 
		\sum_{w\in \Z^d } \ \sum_{v \in \Z^d \setminus \{w\}}
		\ov(\gamma_1^n,\varphi_0^\infty) \mathbbm{1}_{\{\gamma_1 = w,  \varphi_1 = v\}} \\
		\label{line2}
		\ \ \ \
		&
		+ \sum_{w\in \Z^d\setminus\{\mz,z\}} \frac{1}{p_w}   \mathbbm{1}_{\{\gamma_1 = \varphi_1 = w\}}
		\ov(\gamma_1^n,\varphi_1^\infty).
	\end{align}
	We bound each of the three terms on the right-hand side of inequality \eqref{line2} separately in expectation. For the first term, the same analysis as in \hyperlink{target2}{Case 2}, specifically as used to bound the first summand of \eqref{eq:ineq start z}, shows that 
	\begin{equation*}
		\be_{\hat{\mu} \times \tilde{\mu}} \left[ \frac{1}{p_z} \ov(\gamma_1^n,\varphi_0^\infty) \mathbbm{1}_{\{\gamma_1 = z\}} \right] \leq \frac{\widetilde{W_n}}{100}.
	\end{equation*}
	For the second summand of \eqref{line2}, for every $w,v \in \Z^d$ with $w\neq v$ and $ \bp_{\hat{\mu} \times \tilde{\mu}}\left(  \gamma_1 = w,  \varphi_1 = v \right) > 0$ define the sequences $\Delta(\gamma_1^n) \coloneqq (\gamma^\prime_k)_{k\in \{0,\ldots,n-1\}}$ and $\Delta(\varphi) \coloneqq (\varphi^\prime_k)_{k\in \N_0}$ by 
	\begin{equation*}
		\gamma_k^\prime = \gamma_{k+1} - w \ \text{ and } \
		\varphi_k^\prime = \varphi_{k+1} - v.
	\end{equation*}
	The path $-v+\varphi_0^\infty$ traverses exactly the same set of edges as the path $(-v,\varphi^\prime)_{-1}^\infty$. The path $-v+\gamma_1^n$ traverses exactly the same st of edges as the path $-v+w+\gamma_0^{\prime(n-1)}=w-v+\Delta(\gamma_1^n)$.
	So for the weighted overlap, we get that
	\begin{align*}
		\ov(\gamma_1^n, \varphi_0^\infty) &
		= 
		\ov(-v+\gamma_1^n, -v+\varphi_0^\infty)
		= 
		\ov \left( w-v+\gamma_0^{\prime(n-1)}, (-v,\varphi^\prime)_{-1}^\infty \right) \\
		&
		= 
		\ov \Big( w-v+\Delta(\gamma_1^n), (-v,\varphi^\prime)_{-1}^\infty \Big).
	\end{align*}
	and taking expectations gives that
	\begin{align*}
		\be_{\hat{\mu} \times \tilde{\mu}} \left[ \ov(\gamma_1^n, \varphi_0^\infty) | \gamma_1=w, \varphi_1=v\right]
		=
		\be_{\hat{\mu} \times \tilde{\mu}} \left[ \ov \Big( w-v+\Delta(\gamma_1^n), (-v,\varphi^\prime)_{-1}^\infty \Big) | \gamma_1=w, \varphi_1=v\right].
	\end{align*}
	Conditioned on $ \gamma_1=w, \varphi_1=v$, the sequences $\varphi^\prime= \left(\varphi^\prime_k = \varphi_{k+1}-v\right)_{k\in \N_0}$ and $\gamma^\prime = \left(\gamma^\prime_k = \gamma_{k+1}-w\right)_{k\in \N_0}$ are distributed according to measures in $M_j$. Call these distributions $\mu^{\prime}$, respectively $\mu^{\prime\prime}$. As $w-v\neq \mz$ and $w-v \neq -v$ (remember that $\bp_{\hat{\mu} \times \tilde{\mu}} \left(\gamma_1=w, \varphi_1=v\right) > 0$, so $w\neq \mz$) we get by the analysis of \hyperlink{target1}{Case 1} that
	\begin{align*}
		&\be_{\hat{\mu} \times \tilde{\mu}} \left[ \ov \Big( w-v+\Delta(\gamma_1^n), (-v,\varphi^\prime)_{-1}^\infty \Big) | \gamma_1=w, \varphi_1=v\right]\\
		&
		=
		\be_{\mu^{\prime} \times \mu^{\prime\prime}} \left[ \ov \Big( w-v+\gamma_0^{n-1}, (-v,\varphi)_{-1}^\infty \Big)  \right] \leq 1+\frac{\widetilde{W_n}}{4}
	\end{align*}
	for all $w,v\in \Z^d$ with $w\neq v, \bp_{\hat{\mu} \times \tilde{\mu}} \left(\gamma_1=w, \varphi_1=v\right) > 0$. The law of total expectation implies that
	\begin{align*}
		&\be_{\hat{\mu} \times \tilde{\mu}} \left[ \sum_{w\in \Z^d } \ \sum_{v \in \Z^d \setminus \{w\}}
		\ov(\gamma_1^n,\varphi_0^\infty) \mathbbm{1}_{\{\gamma_1 = w,  \varphi_1 = v\}} \right]\\
		&
		= \sum_{w\in \Z^d } \ \sum_{v \in \Z^d \setminus \{w\}}
		\be_{\hat{\mu} \times \tilde{\mu}} \left[ 
		\ov(\gamma_1^n,\varphi_0^\infty) | \gamma_1 = w,  \varphi_1 = v  \right] \bp_{\hat{\mu} \times \tilde{\mu}}\left(  \gamma_1 = w,  \varphi_1 = v \right)
		\leq
		1+\frac{\widetilde{W_n}}{4}.
	\end{align*}
	Finally, we consider the third summand of \eqref{line2}. For $w\in \Z^d$ with $\bp_{\hat{\mu} \times \tilde{\mu}}\left(  \gamma_1 = w =  \varphi_1  \right)>0$, consider the sequences $\Delta(\gamma_1^n)= (\gamma_k^\prime=\gamma_{k+1}-w)_{k\in \{0,\ldots,n-1\}}$ and $\Delta(\varphi) \coloneqq (\varphi^\prime_k = \varphi_{k+1}-w)_{k\in \N_0}$. Then,  on the event where $\varphi_1=\gamma_1=w$, one has
	\begin{equation*}
		\ov(\gamma_1^n,\varphi_1^\infty) = \ov(\Delta(\gamma_1^n),\Delta(\varphi))
	\end{equation*}
	Further, conditioned on $\varphi_1=\gamma_1=w$, the distributions of the sequences $(\gamma_k^\prime=\gamma_{k+1}-w)_{k\in \N_0}$ and $(\varphi^\prime_k = \varphi_{k+1}-w)_{k\in \N_0}$ are in $M_j$. Write $\mu^\prime$, respectively $\mu^{\prime\prime}$ for the distribution of $(\gamma_k^\prime=\gamma_{k+1}-w)_{k\in \N_0}$, respectively $(\varphi^\prime_k = \varphi_{k+1}-w)_{k\in \N_0}$. Then
	\begin{align*}
		&\be_{\hat{\mu} \times \tilde{\mu}}\left[\ov(\gamma_1^n,\varphi_1^\infty) | \gamma_1=\varphi_1=w\right]
		=
		\be_{\hat{\mu} \times \tilde{\mu}}\left[\ov(\Delta(\gamma_1^n),\Delta(\varphi)) | \gamma_1=\varphi_1=w\right]\\
		&
		=
		\be_{\mu^{\prime} \times \mu^{\prime\prime }} \left[\ov(\gamma_0^{n-1},\varphi_0^{n-1})\right]
		\leq
		\widetilde{W_n}
	\end{align*}
	Combining this with the fact that
	\begin{equation*}
		\bp_{\hat{\mu} \times \tilde{\mu}} \left( \gamma_1 = \varphi_1 = w\right)
		=
		\bp_{\hat{\mu} \times \tilde{\mu}} \left(  \varphi_1 = w\right)
		\bp_{\hat{\mu} \times \tilde{\mu}} \left( \gamma_1 = w\right)
		\leq
		\left(\frac{p_w 28}{T(d)}\right) \bp_{\hat{\mu} \times \tilde{\mu}} \left(  \varphi_1 = w\right)
	\end{equation*}
	by Claim \ref{claim:3.4} for $d\geq 4$, respectively Claim \ref{claim:a sets} for $d=3$, we get by the law of total expectation that
	\begin{align*}
		& \be_{\hat{\mu} \times \tilde{\mu}} \left[ \sum_{w\in \Z^d\setminus\{\mz,z\}} \frac{1}{p_w}   \mathbbm{1}_{\{\gamma_1 = \varphi_1 = w\}}
		\ov(\gamma_1^n,\varphi_1^\infty) \right]\\
		&
		\leq
		\sum_{w\in \Z^d\setminus\{\mz,z\}}  \frac{1}{p_w} \bp_{\hat{\mu} \times \tilde{\mu}} \left( \gamma_1 = \varphi_1 = w\right)
		\be_{\hat{\mu} \times \tilde{\mu}} \left[    
		\ov(\gamma_1^n,\varphi_1^\infty) | \gamma_1 = \varphi_1 = w\right]\\
		&
		\leq
		\sum_{w\in \Z^d\setminus\{\mz,z\}}  \frac{1}{p_w} \left(\frac{p_w 28}{T(d)}\right)  \bp_{\hat{\mu} \times \tilde{\mu}} \left(  \varphi_1 = w\right)
		\widetilde{W_n}\\
		&
		= \widetilde{W_n}
		\frac{ 28}{T(d)}
		\sum_{w\in \Z^d\setminus\{\mz,z\}}   \bp_{\hat{\mu} \times \tilde{\mu}} \left(  \varphi_1 = w\right)
		\leq \frac{\widetilde{W_n}}{100}.
	\end{align*}
	Thus we bounded each of the three summands on the right-hand side of \eqref{line2} in expectation. Summing the three inequalities we get that
	\begin{equation*}
		\be_{\hat{\mu} \times \tilde{\mu}} \left[\ov(\gamma_0^n,(z\varphi)_{-1}^{\infty})\right]
		\leq 1+ \frac{\widetilde{W_n}}{4} + \frac{\widetilde{W_n}}{100} + \frac{\widetilde{W_n}}{100}
		\leq 1+ \frac{\widetilde{W_n}}{3}
	\end{equation*}
	which finishes the proof.
\end{proof}

\subsubsection{Intersection probabilities in dimensions $d \geq 4$}\label{subsubsec:intersection 4 or higher}

\begin{lemma}\label{lemma:hitting d=4}
	Let $\left((\gamma_k)_{k\in \N_0},(\varphi_n)_{n\in \N_0}\right)$ be distributed according to the measure $\bp_{\hat{\mu} \times \tilde{\mu}}$ with $\hat{\mu},\tilde{\mu} \in M_2$. Then, for all $x,z \in \Z^d$ one has
	\begin{equation}\label{eq:hitting d=4}
		\bp_{\hat{\mu} \times \tilde{\mu}}(\exists (k,n) \in \N_0 \times \N_0\setminus \{(0,0)\} \text{ with } x+\gamma_k=\varphi_n \text{ or } \exists k > 0 \text{ with } x+\gamma_k = z)\leq \frac{1}{4} .
	\end{equation}
\end{lemma}
\begin{proof}
	The random variables $\langle \gamma_{i+1}-\gamma_{i}, w_j \rangle$ are independent for different $i\in \N_0$. Further, if $k\geq 6N$, for each $j\in \{1,2,3\}$, there are at least $N$ many indices $i\in \{0,\ldots,k-1\}$ for which $(\gamma_{i+1}-\gamma_i)$ had probability mass function $\psi_{j}$, so in particular $\langle \gamma_{i+1}-\gamma_{i}, w_j \rangle \neq 0$ almost surely. So if we write
	\begin{equation*}
		\langle \gamma_k, w_j \rangle = \sum_{i=0}^{k-1} \langle \gamma_{i+1} - \gamma_{i}, w_j \rangle
		=
		\sum_{i=0}^{k-1}  \text{sgn}(\langle \gamma_{i+1} - \gamma_{i}, w_j \rangle) |\langle \gamma_{i+1} - \gamma_{i}, w_j \rangle | ,
	\end{equation*}
	then for each $j\in \{1,2,3\}$, the above sum contains at least $N$ indices $i\in \{0,\ldots,k-1\}$ for which $|\langle \gamma_{i+1} - \gamma_{i}, w_j \rangle | \neq 0$. Conditioned on the $\sigma$-algebra 
	\begin{equation*}
		\cF = \sigma\left(\left(|\langle \gamma_{i+1} - \gamma_{i}, w_j \rangle |\right)_{i\in \{0,\ldots,k-1\}, j\in \{1,2,3\}}, \langle \gamma_k,w_4\rangle \right),
	\end{equation*}
	all the signs $\text{sgn}(\langle \gamma_{i+1} - \gamma_{i}, w_j \rangle)$ that are non-zero are independent and uniformly distributed on $\{-1,+1\}$. In particular, we get by Lemma \ref{lem1} that
	for all $k\geq 6N$ and all $z_1,z_2,z_3,\ell \in \Z$
	\begin{align*}
		&\bp_{\hat{\mu} \times \tilde{\mu}}\left(
		\langle \gamma_k,w_1\rangle = z_1, 
		\langle \gamma_k,w_2\rangle = z_2, 
		\langle \gamma_k,w_3\rangle = z_3 \Big| \langle \gamma_k,w_4\rangle = \ell  \right)\\
		=
		&
		\be_{\hat{\mu} \times \tilde{\mu}}
		\left[\bp_{\hat{\mu} \times \tilde{\mu}}\left(
		\langle \gamma_k,w_1\rangle = z_1, 
		\langle \gamma_k,w_2\rangle = z_2, 
		\langle \gamma_k,w_3\rangle = z_3 | \cF  \right) \Big| \langle \gamma_k,w_4\rangle = \ell \right]\\
		=
		&
		\be_{\hat{\mu} \times \tilde{\mu}}
		\left[
		\prod_{i=1}^{3}
		\bp_{\hat{\mu} \times \tilde{\mu}}\left(
		\langle \gamma_k,w_i\rangle = z_i | \cF  \right)
		\Big| \langle \gamma_k,w_4\rangle = \ell
		\right]
		\leq  N^{-3/2}.
	\end{align*}
	In the last line, we used that the terms of the form $\bp_{\hat{\mu} \times \tilde{\mu}}\left(
	\langle \gamma_k,w_j\rangle = z_j | \cF  \right)$ are independent conditioned on $\cF$, as they only depend on $\text{sgn}\left(\langle \gamma_{i+1} - \gamma_{i}, w_j \rangle \right)$. As at least $N$ many terms in the sum $\langle \gamma_k, w_j \rangle = \sum_{i=0}^{k-1} \langle \gamma_{i+1} - \gamma_{i}, w_j \rangle$ are non-zero, we can thus apply inequality \eqref{eq:sqrt n bound d=1} from Lemma \ref{lem1} to get that $\bp_{\hat{\mu} \times \tilde{\mu}}\left(
	\langle \gamma_k,w_j\rangle = z_j | \cF  \right) \leq N^{-1/2}$ for $j=1,2,3$. Using the above inequality, we can directly deduce that for each $v\in \Z^d, \ell \in \Z$ one has
	\begin{equation}\label{eq:inequality ssup}
		\bp_{\hat{\mu} \times \tilde{\mu}}\left( \gamma_k =v  \big| \langle \gamma_k,w_4\rangle =\ell \right) \leq N^{-3/2}.
	\end{equation}
	Each layer $\{y:\langle y, w_4\rangle=\ell\}$ (for $\ell \in \Z$) can contain at most three points in $(z\varphi)_{-1}^{\infty}$.
	For $k\geq 6N$ with $N\in \N_0$, we thus get by conditioning on $\langle \gamma_k, w_4 \rangle$ that
	\begin{align*}
		& \bp_{\hat{\mu} \times \tilde{\mu}}\left(x+\gamma_k = \varphi_n \text{ for some } n\in \N_0 \text{ or } x+\gamma_k=z \right)\\
		& \leq \sup_{\ell \in \Z} \bp_{\hat{\mu} \times \tilde{\mu}}\left(x+\gamma_k = \varphi_n \text{ for some } n\in \N_0 \text{ or } x+\gamma_k=z \Big| (\varphi_n)_{n\in \N_0}, \langle \gamma_k, w_4 \rangle = \ell\right)\\
		&
		\leq \sup_{\ell \in \Z} \sup_{v_1,v_2,v_3 \in \Z^d}
		\bp_{\hat{\mu} \times \tilde{\mu}}\left(x+\gamma_k \in \{v_1,v_2,v_3\}  \Big| \langle \gamma_k, w_4 \rangle = \ell \right) \overset{\eqref{eq:inequality ssup}}{\leq} 3 N^{-3/2},
	\end{align*}
	and thus we get via a union bound that
	\begin{align}\label{eq:unionbound1}
		&\notag  \bp_{\hat{\mu} \times \tilde{\mu}}\left(\exists k\geq 6 \cdot 10^5 :  x+\gamma_k = \varphi_n \text{ for some } n\in \N_0 \text{ or } x+\gamma_k=z  \right)\\
		& \notag
		\leq \sum_{k=6\cdot 10^5}^{\infty}
		\bp_{\hat{\mu} \times \tilde{\mu}}\left(x+\gamma_k = \varphi_n \text{ for some } n\in \N_0 \text{ or } x+\gamma_k=z  \right)\\
		& \notag
		=
		\sum_{N=10^5}^{\infty} \ \sum_{r=0}^{5}
		\bp_{\hat{\mu} \times \tilde{\mu}}\left(x+\gamma_{6N+r} = \varphi_n \text{ for some } n\in \N_0 \text{ or } x+\gamma_{6N+r}=z  \right)\\
		&
		\leq  \sum_{N=10^5}^{\infty} \sum_{r=0}^{5}
		3N^{-3/2} \leq 18 \int_{99999}^{\infty}s^{-3/2} \md s \leq 0.12.
	\end{align}
	For $k < 6\cdot 10^5$, we get via the same calculation that
	\begin{align}\label{eq:unionbound2}
		\notag & \sum_{k=0}^{6\cdot 10^{5}-1} \bp_{\hat{\mu} \times \tilde{\mu}}\left( x+\gamma_k = \varphi_n \text{ for some } n \geq 6\cdot 10^{16} \right) \\
		&
		\leq  \sum_{k=0}^{6\cdot 10^5-1}  \sum_{j=10^{16}}^{\infty} \sum_{r=0}^{5}
		\bp_{\hat{\mu} \times \tilde{\mu}}\left( x+\gamma_k = \varphi_{6j+r}  \right)
		\leq  \sum_{k=0}^{6\cdot 10^5-1}  \sum_{j=10^{16}}^{\infty} 6
		j^{-3/2} \leq 0.08.
	\end{align}
	Furthermore, using that for all $j\in \{1,\ldots,4\}, w \in \Z^d$ one has $\psi_j(w)\leq \frac{28}{T(d)}$ by Claim \ref{claim:3.4}, we have by elementary properties of convolutions that
	\begin{align}
		\notag \sum_{k=1}^{6 \cdot 10^{5}-1} & \bp_{\hat{\mu} \times \tilde{\mu}} \left( x+\gamma_k = z \right) \leq 6 \cdot 10^{5}\left( \sup_{ w \in \Z^d, k\geq 1} \bp_{\hat{\mu} \times \tilde{\mu}} \left( x+\gamma_k = w \right)\right) \\
		&
		\label{eq:unionbound3}
		\leq 6\cdot  10^{5} \left( \sup_{ w \in \Z^d} \bp_{\hat{\mu} \times \tilde{\mu}} \left( x+\gamma_1 = w \right) \right)
		\leq 6 \cdot 10^{5} \frac{28}{T(d)} \leq 0.01, 
		\\
		\notag \sum_{k=1}^{6\cdot 10^{5}-1} & \  \sum_{n=0}^{6\cdot 10^{16}-1}  \bp_{\hat{\mu} \times \tilde{\mu}} \left( x+\gamma_k = \varphi_n \right) \leq 36\cdot 10^{21} \left(\sup_{w \in \Z^d} \bp_{\hat{\mu} \times \tilde{\mu}} \left( x+\gamma_1 = w \right) \right)\\
		&
		\label{eq:unionbound4}
		\leq 36 \cdot 10^{21} \cdot \frac{ 28}{T(d)} \leq 0.02, \text{ and}\\
		&
		\label{eq:unionbound5}
		\sum_{n=1}^{6\cdot 10^{16}-1}  \bp_{\hat{\mu} \times \tilde{\mu}} \left( x+\gamma_0 = \varphi_n \right) \leq 10^{17} \left(\sup_{w \in \Z^d} \bp_{\hat{\mu} \times \tilde{\mu}} \left( w = \varphi_1 \right) \right)
		\leq
		10^{17} \frac{28}{T(d)}
		\leq 10^{-5}.
	\end{align}
	To finish the proof, note that \eqref{eq:unionbound1} through \eqref{eq:unionbound5} imply \eqref{eq:hitting d=4} via another union bound over all pairs $(k,n) \in \N_0 \times \N_0 \setminus \{(0,0)\}$.
\end{proof}

\subsubsection{Intersection probabilities in dimension $d = 3$}\label{subsubsec:intersection 3}

\begin{lemma}\label{lemma:hitting d=3}
	Let $\left((\gamma_k)_{k\in \N_0},(\varphi_n)_{n\in \N_0}\right)$ be distributed according to the measure $\bp_{\hat{\mu} \times \tilde{\mu}}$ with $\hat{\mu},\tilde{\mu} \in M_3$. Then, for all $x,z \in \Z^3$ one has
	\begin{equation}\label{eq:hitting d=3}
		\bp_{\hat{\mu} \times \tilde{\mu}}(\exists (k,n) \in \N_0 \times \N_0\setminus \{(0,0)\} \text{ with } x+\gamma_k=\varphi_n \text{ or } \exists k > 0 \text{ with } x+\gamma_k = z)\leq \frac{1}{4} .
	\end{equation}
\end{lemma}

Before going to the proof of Lemma \ref{lemma:hitting d=3}, we prove an upper bound on the conditioned predictability profile of $\gamma_k$.

\begin{claim}\label{claim:hear kernel unpred}
	For all $\hat{\mu}, \tilde{\mu} \in M_3, n \in \N, k\geq 4n$ one has
	\begin{equation}\label{eq:k = 4n in dim 3}
		\sup_{w\in \Z^3, \ell \in \N_0}
		\bp_{\hat{\mu} \times \tilde{\mu}}\left(\gamma_k = w | \langle \gamma_k, e_3 \rangle = \ell \right) \leq  100 n^{-1.13}
	\end{equation}
	and thus
	\begin{equation}\label{eq:k general in dim 3}
		\sup_{w\in \Z^3, \ell \in \N_0}
		\bp_{\hat{\mu} \times \tilde{\mu}}\left(\gamma_k = w | \langle \gamma_k, e_3 \rangle = \ell \right) \leq 1 \wedge 100 \left( \Big \lfloor \frac{k}{4} \Big \rfloor \right)^{-1.13} \leq 600 k^{-1.13}.
	\end{equation}
\end{claim}

\begin{proof}
	We only do the proof in the case $k=4n$. The proof for $k\geq 4n$ is similar.
	As $(\gamma_i)_{i\in \N_0}$ is distributed according to a measure $\hat{\mu} \in M_3$, we get that $(\gamma_i)_{i\in \N_0}$ is distributed according to $(\zeta_{i+N}-\zeta_{N})_{i \in \N_0}$, conditionally on $\zeta_0,\ldots,\zeta_N$, where $(\zeta_i)_{i\in \N_0}$ has distribution $\mu$, cf. \eqref{eq:zeta def}. We will do the proof only for the case where $N=0 \mod 4$. The other three cases work analogously.
	Let $\cF$ be the $\sigma$-Algebra defined by
	\begin{align*}
		\cF=\sigma \Big( & \zeta_0,\ldots,\zeta_N, (\zeta_{N+2i+1}-\zeta_{N+2i})_{i\in \N_0} , \langle \zeta_{N+k} -\zeta_{N}, e_3 \rangle,
		\sum_{i=0}^{n-1} \langle \zeta_{N+4i+2}-\zeta_{N+4i+1},e_2 \rangle
		\\
		& \left(\langle \zeta_{N+4i+4}-\zeta_{N+4i+3}, e_1 \rangle \right)_{i\in \{0,\ldots,n-1\}},
		\Big[|\langle \zeta_{N+4i+2}-\zeta_{N+4i+1},e_1 \rangle| \Big]_{i\in \{0,\ldots,n-1\}}
		\Big).
	\end{align*}
	From the definition of $\zeta$ and $\cF$, it follows that
	\begin{align*}
		& \sup_{w\in \Z^3, \ell \in \N_0}
		\bp_{\hat{\mu} \times \tilde{\mu}}\left(\gamma_k = w \ \big| \langle \gamma_k, e_3 \rangle = \ell \right) \leq 
		\sup_{w_1,w_2\in \Z, \omega}
		\bp_{\hat{\mu} \times \tilde{\mu}}\left( \langle \zeta_{N+k}, e_1 \rangle = w_1 , \langle \zeta_{N+k}, e_2 \rangle = w_2 \ \big| \cF \right)(\omega)
	\end{align*}
	and thus we will focus on the predictability of $\zeta_{N+k}$ with respect to the $\sigma$-Algebra $\cF$ from here on.
	We write
	\begin{align*}
		& \zeta_{N+k}-\zeta_k = 
		\sum_{i=N}^{N+k-1} (\zeta_{i+1}-\zeta_N) = 
		\sum_{s=0}^{3}
		\sum_{i=0}^{n-1} (\zeta_{N+4i+1 + s}-\zeta_{N+4i + s}) .
	\end{align*}
	The first and the third summand in the above sum, i.e., the inner sum for $s=0$ ($\sum_{i=0}^{n-1} (\zeta_{N+4i+1}-\zeta_{N+4i})$) and the inner sum for $s=2$ ($\sum_{i=0}^{n-1} (\zeta_{N+4i+3}-\zeta_{N+4i+2})$) are measurable with respect to the $\sigma$-Algebra $\cF$, so we can ignore them when it comes to the predictability of $\zeta_{N+k}$ conditioned on $\cF$.
	Conditioned on the $\sigma$-Algebra $\cF$, for $w_1,w_2\in \Z$ the two events $\left\{\langle \zeta_{N+k}, e_1 \rangle = w_1\right\}$ and $\left\{\langle \zeta_{N+k}, e_2 \rangle = w_2\right\}$ are independent. Indeed, conditioned on $\cF$, the event $\left\{\langle \zeta_{N+k}, e_1 \rangle = w_1\right\}$ only depends on $\sum_{i=0}^{n-1} (\zeta_{N+4i+2}-\zeta_{N+4i+1})$, as all three sums $\sum_{i=0}^{n-1} \langle\zeta_{N+4i+s+1}-\zeta_{N+4i+s}, e_1 \rangle$ for $s\in \{0,2,3\}$ are measurable with respect to $\cF$. Contrary to that, the event $\left\{\langle \zeta_{N+k}, e_2 \rangle = w_2\right\}$ only depends on $\sum_{i=0}^{n-1} (\zeta_{N+4i+4}-\zeta_{N+4i+3})$, as all three sums $\sum_{i=0}^{n-1} \langle\zeta_{N+4i+s+1}-\zeta_{N+4i+s}, e_2 \rangle$ for $s\in \{0,1,2\}$ are measurable with respect to $\cF$. Thus we get that
	\begin{align}\label{eq:predictability independence}
		& \notag \sup_{w\in \Z^3, \ell \in \N_0}
		\bp_{\hat{\mu} \times \tilde{\mu}}\left(\gamma_k = w | \langle \gamma_k, e_3 \rangle = \ell \right) \leq 
		\sup_{w_1,w_2\in \Z, \omega}
		\bp_{\hat{\mu} \times \tilde{\mu}}\left( \langle \zeta_{N+k}, e_1 \rangle = w_1 , \langle \zeta_{N+k}, e_2 \rangle = w_2  | \cF \right)(\omega)\\
		& \notag 
		= \sup_{w_1,w_2\in \Z, \omega}
		\Big(\bp_{\hat{\mu} \times \tilde{\mu}}\left(\langle \zeta_{N+k}, e_1 \rangle = w_1  | \cF \right)(\omega) \cdot 
		\bp_{\hat{\mu} \times \tilde{\mu}}\left( \langle \zeta_{N+k}, e_2 \rangle = w_2  | \cF \right)(\omega)\Big)\\
		&
		\leq \left(\sup_{w_1\in \Z, \omega}
		\bp_{\hat{\mu} \times \tilde{\mu}}\left(\langle \zeta_{N+k}, e_1 \rangle = w_1  | \cF \right)(\omega)\right)
		\left(
		\sup_{w_2\in \Z, \omega}
		\bp_{\hat{\mu} \times \tilde{\mu}}\left( \langle \zeta_{N+k}, e_2 \rangle = w_2  | \cF \right)(\omega)\right).
	\end{align}
	For the first factor in the above product, we get that
	\begin{equation*}
		\sup_{w_1\in \Z, \omega}
		\bp_{\hat{\mu} \times \tilde{\mu}}\left(\langle \zeta_{N+k}, e_1 \rangle = w_1  | \cF \right)(\omega)
		=
		\sup_{w_1\in \Z, \omega}
		\bp_{\hat{\mu} \times \tilde{\mu}}\left( \sum_{i=0}^{n-1} \langle \zeta_{N+4i+2}-\zeta_{N+4i + 1}, e_1 \rangle = w_1  | \cF \right)(\omega)
	\end{equation*}
	By definition, we can write
	\begin{equation*}
		\sum_{i=0}^{n-1} \langle \zeta_{N+4i+2}-\zeta_{N+4i+1},e_1 \rangle
		=
		\sum_{i=0}^{n-1} \sigma_{N/4+i}  \cdot |\langle \zeta_{N+4i+2}-\zeta_{N+4i+1},e_1 \rangle|
	\end{equation*}
	where $\sigma_i \in \{-1,+1\}$ are random variables, such that the process $S=\left(S_m=\sum_{i=0}^{m} \sigma_i\right)$ satisfies \eqref{eq:predict 100 old}. Conditioned on $\cF$, the collection of random variables 
	\begin{equation*}
		\left[|\langle \zeta_{N+4i+2}-\zeta_{N+4i+1},e_1 \rangle| \right]_{i\in \{0,\ldots,n-1\}}
	\end{equation*}
	is a multiset containing $n$ positive numbers. Without any conditioning, the random variables $\left(|\langle \zeta_{N+4i+2}-\zeta_{N+4i+1},e_1 \rangle| \right)_{i\in \{0,\ldots,n-1\}}$ are i.i.d., as
	\begin{equation*}
		|\langle \zeta_{N+4i+2}-\zeta_{N+4i+1},e_1 \rangle| = X_1^{N/4+i}
	\end{equation*}
	by \eqref{eq:zeta def}. After conditioning, the random variables $\left(|\langle \zeta_{N+4i+2}-\zeta_{N+4i+1},e_1 \rangle| \right)_{i\in \{0,\ldots,n-1\}}$ are not independent anymore, but they are exchangeable, as $\cF$ contains only information about $\left[|\langle \zeta_{N+4i+2}-\zeta_{N+4i+1},e_1 \rangle| \right]_{i\in \{0,\ldots,n-1\}}$, $\sum_{i=0}^{n-1} \langle \zeta_{N+4i+2}-\zeta_{N+4i+1},e_3 \rangle$, and $\sum_{i=0}^{n-1} \langle \zeta_{N+4i+2}-\zeta_{N+4i+1},e_2 \rangle$.
	Thus we get by Proposition \ref{prop:unpredictable} that 
	\begin{align}\label{eq:predict factor 1}
		\notag &\sup_{w_1\in \Z, \omega}
		\bp_{\hat{\mu} \times \tilde{\mu}}\left(\langle \zeta_{N+k}, e_1 \rangle = w_1  | \cF \right)(\omega)
		=
		\sup_{w_1\in \Z, \omega}
		\bp_{\hat{\mu} \times \tilde{\mu}}\left( \sum_{i=0}^{n-1} \langle \zeta_{N+4i+2}-\zeta_{N+4i + 1}, e_1 \rangle = w_1  \Big| \cF \right)(\omega)\\
		&
		\leq
		\sup_{w_1\in \Z, \omega}
		\bp_{\hat{\mu} \times \tilde{\mu}}\left( \sum_{i=0}^{n-1} \sigma_{N/4+i} |\langle \zeta_{N+4i+2}-\zeta_{N+4i+1},e_1 \rangle| = w_1  \Big| \cF \right)(\omega)
		\leq 100 n^{-\frac{\log(2)}{\log(3)}}.
	\end{align}
	For the component in $e_2$-direction, we get that
	\begin{equation*}
		\langle \zeta_{N+k}, e_2 \rangle = 
		\langle \zeta_{N}, e_2 \rangle +
		\sum_{s=0}^{2}
		\sum_{i=0}^{n-1} \langle \zeta_{N+4i+s+1}-\zeta_{N+4i+s} , e_2 \rangle
		+ \sum_{i=0}^{n-1} \langle \zeta_{N+4i+4}-\zeta_{N+4i+3} , e_2 \rangle.
	\end{equation*}
	In the above expression, all the terms besides $\sum_{i=0}^{n-1} \langle \zeta_{N+4i+4}-\zeta_{N+4i+3} , e_2 \rangle$ are measurable with respect to $\cF$. Further, the numbers $|\langle \zeta_{N+4i+4}-\zeta_{N+4i+3} , e_2 \rangle|$ are always positive for $i=0,\ldots,n-1$, and the sign of $\langle \zeta_{N+4i+4}-\zeta_{N+4i+3} , e_2 \rangle$ is equally likely to be positive or negative, by construction. Thus we get by Lemma \ref{lem1}, \eqref{eq:sqrt n bound d=1}, that 
	\begin{align}\label{eq:predict factor 2}
		\notag
		&\sup_{w_2\in \Z, \omega}
		\bp_{\hat{\mu} \times \tilde{\mu}}\left( \langle \zeta_{N+k}, e_2 \rangle = w_2  | \cF \right)(\omega) \\
		&
		=
		\sup_{w_2\in \Z, \omega}
		\bp_{\hat{\mu} \times \tilde{\mu}}\left( \sum_{i=0}^{n-1} \langle \zeta_{N+4i+4}-\zeta_{N+4i+3} , e_2 \rangle = w_2  \Big| \cF \right)(\omega) \leq \frac{1}{\sqrt{n}}.
	\end{align}
	Combining \eqref{eq:predict factor 1} and \eqref{eq:predict factor 2} and inserting it in \eqref{eq:predictability independence} shows that
	\begin{equation*}
		\sup_{w\in \Z^3, \ell \in \N_0}
		\bp_{\hat{\mu} \times \tilde{\mu}}\left(\gamma_k = w | \langle \gamma_k, e_3 \rangle = \ell \right)
		\leq 100 n^{-\frac{\log(2)}{\log(3)}} \frac{1}{\sqrt{n}} \leq 100 n^{-1.13} .
	\end{equation*}
\end{proof}

With this, we can now move to the proof of Lemma \ref{lemma:hitting d=3}. We use a similar method as in the proof of Lemma \ref{lemma:hitting d=4}.

\begin{proof}[Proof of Lemma \ref{lemma:hitting d=3}]
	Each layer $\{y:\langle y, e_3\rangle=\ell\}$ can contain at most three points in $(z\varphi)_{-1}^{\infty}$.
	We thus get by Claim \ref{claim:hear kernel unpred} that for all $k\geq 4$
	\begin{align*}
		&\bp_{\hat{\mu} \times \tilde{\mu}}\left(x+\gamma_k = \varphi_n \text{ for some } n\in \N_0 \text{ or } x+\gamma_k=z \right)\\
		\leq
		& \sup_{\ell \in \N_0}
		\bp_{\hat{\mu} \times \tilde{\mu}}\left(x+\gamma_k = \varphi_n \text{ for some } n\in \N_0 \text{ or } x+\gamma_k=z \Big| \langle \gamma_k, e_3 \rangle = \ell, (\varphi_n)_{n\in \N_0} \right)\\
		\leq
		& \sup_{v_1,v_2,v_3 \in \Z^d,\ell \in \N_0}
		\bp_{\hat{\mu} \times \tilde{\mu}}\left(x+\gamma_k \in \{v_1,v_2,v_3\} \Big| \langle \gamma_k, e_3 \rangle = \ell, (\varphi_n)_{n\in \N_0} \right)
		\leq 3\cdot 600 k^{-1.13} = 1800 k^{-1.13}.
	\end{align*}
	Using a union bound over the values $k\geq 10^{40}$ we get that
	\begin{align}\label{eq:unionbound1 d=3}
		&\notag  \bp_{\hat{\mu} \times \tilde{\mu}}\left(\exists k\geq 10^{40} :  x+\gamma_k = \varphi_n \text{ for some } n\in \N_0 \text{ or } x+\gamma_k=z  \right)\\
		& \notag
		\leq \sum_{k=10^{40}}^{\infty}
		\bp_{\hat{\mu} \times \tilde{\mu}}\left(x+\gamma_k = \varphi_n \text{ for some } n\in \N_0 \text{ or } x+\gamma_k=z  \right)\\
		&
		\leq \sum_{k=10^{40}}^{\infty}
		1800 k^{-1.13}\leq 1800 \int_{10^{40}-1}^{\infty}s^{-1.13} \md s \leq \frac{1}{10}.
	\end{align}
	For $k < 10^{40}$, we get via the same calculation that
	\begin{align}\notag \label{eq:unionbound2 d=3}
		&\sum_{k=0}^{10^{40}-1} \bp_{\hat{\mu} \times \tilde{\mu}}\left( x+\gamma_k = \varphi_n \text{ for some } n \geq 10^{350} \right) \leq \sum_{k=0}^{10^{40}-1}  \sum_{j=10^{350}}^{\infty} 1800 j^{-1.13}\\
		& \leq 10^{40}1800 \int_{10^{350}-1}^{\infty} s^{-1.13} \md s \leq \frac{1}{10}.
	\end{align}
	For $k = 1$ and $w\in \Z^3$ one has by definition that
	\begin{equation*}
		\bp_{\hat{\mu} \times \tilde{\mu}} \left( \gamma_1 = w \right) \leq \left(\min_{i=1,2,3} \sum_{x\in A_i} {p_x}\right)^{-1} \leq \frac{12}{T(3)},
	\end{equation*}
	where the last inequality follows from Claim \ref{claim:a sets}.
	Using the definition of the measure $\mu$, one further deduces that 
	\begin{equation*}
		\bp_{\hat{\mu} \times \tilde{\mu}} \left( \gamma_k = w \right)  \leq \frac{12}{T(3)}
	\end{equation*}
	for all $k\geq 1$; For $k\geq 2$ this needs to hold because either $\gamma_1$ or $\gamma_2-\gamma_1$ takes values in the set $A_3$ and is independent of all other increments. Thus we get that
	\begin{align}
		\label{eq:unionbound3 d=3} \sum_{k=1}^{10^{40}-1} & \bp_{\hat{\mu} \times \tilde{\mu}} \left( x+\gamma_k = z \right)
		\leq 10^{40} \frac{12}{T(3)} \leq 0.01, 
		\\
		\label{eq:unionbound4 d=3}
		\sum_{k=1}^{10^{40}-1} \ & \sum_{n=0}^{ 10^{350}-1}  \bp_{\hat{\mu} \times \tilde{\mu}} \left( x+\gamma_k = \varphi_n \right)
		\leq \frac{10^{390}\cdot 12}{T(3)} \leq 0.01, 
		\text{ and}
		\\
		\label{eq:unionbound5 d=3}
		\sum_{n=1}^{ 10^{350}-1} & \bp_{\hat{\mu} \times \tilde{\mu}} \left( x+\gamma_0 = \varphi_n \right) \leq \frac{10^{350}\cdot 12}{T(3)} \leq 10^{-10}.
	\end{align}
	To finish the proof, note that \eqref{eq:unionbound1 d=3} through \eqref{eq:unionbound4 d=3} imply \eqref{eq:hitting d=3} via another union bound over all pairs $(k,n) \in \N_0 \times \N_0 \setminus \{(0,0)\}$.
\end{proof}

\subsection{The expected overlap in dimension $d=2$}\label{sec:d=2}

The main open problem that arises from Theorems \ref{theo:main} and \ref{theo:isotropic} is whether the same statements are also true in dimension $d=2$. We conjecture that both theorems are also true in dimension $d=2$. A first approach might be to look for another measure $\mu$ on self-avoiding paths on $\Z^2$ that satisfies
\begin{equation}\label{eq:finite}
	\be_{\mu \times \mu} \left[\ov(\gamma, \varphi)\right] < \infty.
\end{equation}
However, Lemma \ref{propo:cutsets} below gives a short argument that there is no measure on self-avoiding paths on $\Z^2$ satisfying \eqref{eq:finite}. For this, recall that for an infinite graph $G=(V,E)$ a set $F\subset E$ is called a {\sl cutset} separating $o\in V$ from infinity if every infinite self-avoiding path starting at $o$ uses at least one edge in $F$. For two infinite self-avoiding paths $\gamma = (\gamma_k)_{k\in \N_0}$ and $\varphi = (\varphi_k)_{k\in \N_0}$, recall that $\gamma \cap \varphi$ is the set of edges that are used in both paths. The reason why \eqref{eq:finite} can not hold in dimension $d=2$ is because for $d=2$, the expected size intersection $|\gamma \cap \varphi|$ needs to be infinite, i.e., $\be_{\mu \times \mu} \left[ |\gamma \cap \varphi| \right] = \infty$, see \eqref{eq:infini} below. This basically follows from the Nash-Williams inequality \cite{nash1959random}. An application of Jensen's inequality then shows that $\be_{\mu \times \mu} \left[\ov(\gamma, \varphi)\right] = \infty$.

\begin{lemma}\label{propo:cutsets}
	Let $G=(V,E)$ be a graph such that there exists a sequence of disjoint cutsets $(E_n)_{n\in \N}$ separating $o\in V$ from infinity and such that $\sum_{n=1}^{\infty} \frac{1}{|E_n|} = \infty$. Let $\mu$ be a measure on self-avoiding paths starting at $o$ and suppose that $\gamma=(\gamma_k)_{k\in \N_0}$ and $\varphi=(\varphi_k)_{k\in \N_0}$ are independently sampled from the measure $\mu$. Then, for any $q\in (0,1)$
	\begin{equation*}
		\be_{\mu \times \mu} \left[ q^{-|\gamma \cap \varphi|} \right] = \infty.
	\end{equation*}
\end{lemma}
\begin{proof}
	For $e\in E$, let $b_e$ be the probability that the edge $e$ is used by a path $(\gamma_k)_{k\in \N_0}$ sampled from to the measure $\mu$, i.e.,
	\begin{equation*}
		b_e = \mu\left( \{\gamma : e \in \gamma\} \right) = \bp_{\mu} \left( \exists k\in \N_0 : e=\{\gamma_k,\gamma_{k+1}\} \right).
	\end{equation*}
	As $E_n$ is a cutset, this directly implies that 
	\begin{equation*}
		\sum_{e \in E_n} b_e 
		= \sum_{e \in E_n} \bp_{\mu} \left(e \in \gamma\right)
		=  \be_{\mu} \left[ \sum_{e \in E_n} \mathbbm{1}_{\{e\in \gamma\}}\right] \geq 1 .
	\end{equation*}
	Thus, we get that
	\begin{equation*}
		\sum_{e \in E_n} \bp_{\mu \times \mu} \left(e \in \gamma \cap \varphi \right) = 
		\sum_{e \in E_n} \bp_{\mu} \left(e \in \gamma \right)^2
		=
		\sum_{e \in E_n} b_e^2 \geq \frac{\left(\sum_{e \in E_n} b_e \cdot 1\right)^2}{\sum_{e \in E_n} 1^2} \geq \frac{1}{|E_n|},
	\end{equation*}
	where we used the Cauchy-Schwarz inequality for the second to last inequality. Linearity of expectation implies that
	\begin{equation}\label{eq:infini}
		\be_{\mu \times \mu} \left[ |\gamma \cap \varphi| \right] = \sum_{e \in E} \bp_{\mu \times \mu} \left(e \in \gamma \cap \varphi \right)
		=
		\sum_{e \in E} b_e^2 
		\geq \sum_{n=1}^{\infty} \sum_{e \in E_n} b_e^2
		\geq \sum_{n=1}^{\infty} \frac{1}{|E_n|} = \infty.
	\end{equation}
	The convexity of the function $s\mapsto q^{-s}$ directly shows via Jensen's inequality that
	\begin{equation*}
		\be_{\mu \times \mu} \left[ q^{-|\gamma \cap \varphi|} \right] \geq q^{-\be_{\mu \times \mu} \left[ |\gamma \cap \varphi| \right]} = \infty.
	\end{equation*}
\end{proof}

For the case where $V=\Z^2$, $E\subset \left\{\{x,y\} \subset \Z^2: 0 < \|x-y\| \leq K\right\}$ for some $K\in \N$, and $o=\mz$, one can construct the sequence of cutsets $(E_n)_{n\in \N}$ defined by 
\begin{equation*}
	E_n=\{\{x,y\}\subset E: 3nK < \|x\| \leq 3(n+1)K, 3nK < \|y\| \leq 3(n+1)K  \}.
\end{equation*}
One checks that $(E_n)_{n\in \N}$ is a sequence of disjoint cutsets separating $\mz$ from infinity for which $|E_n|\approx n$ and thus $\sum_{n=1}^{\infty} \frac{1}{|E_n|}=\infty$. Thus, if 
\begin{equation*}
	q \coloneqq \sup_{\{x,y\} \in E} p_{x-y} < 1,
\end{equation*}
Lemma \ref{propo:cutsets} directly implies that
\begin{equation*}
	\be_{\mu \times \mu} \left[ \ov(\gamma,\varphi) \right] 
	\geq
	\be_{\mu \times \mu} \left[ q^{-|\gamma \cap \varphi|} \right] = \infty
\end{equation*}
for any measure $\mu$ on self-avoiding walks on $(\Z^2,E)$.

\section{Applications to the Potts model} \label{sec:potts}

In this chapter, we show analogous results of Theorems \ref{theo:main} and \ref{theo:isotropic} for the ferromagnetic Potts model.
The same was also done by Friedli and de Lima in \cite{friedli2006truncation} and we follow their notation and setup very closely. See also \cite{grimmett2006book,fortuin1972random,fortuinii1972random,fortuiniii1972random,georgii2001random} for more background. In the $q$-states Potts model ($q\in \N_{\geq 2}$), there is a spin $\sigma_x \in \{1,\ldots,q\}$ attached to each vertex $x\in \Z^d$. 
Let $\phi : \Z^d \setminus \{\mz\} \to \left[0,\infty\right)$ be a {\sl ferromagnetic potential}.
For $s\in \{1,\ldots,q\}$, $\Lambda_L = \{-L,\ldots,L\}^d$, and $\sigma \in \{1,\ldots,q\}^{\Lambda_L} \eqqcolon \Omega_{\Lambda_L}$, the Hamiltonian with pure boundary condition $s$ for a potential $\phi$ is defined by
\begin{equation*}
	H^s_{\phi,\Lambda_L}(\sigma) = - \sum_{\{x,y\}\subset \Lambda_L: x \neq y} \phi(x-y)\mathbbm{1}_{\{\sigma_x=\sigma_y\}} - \sum_{x\in \Lambda_L, y \notin \Lambda_L} \phi(x-y)\mathbbm{1}_{\{\sigma_x=s\}}.
\end{equation*}
At inverse temperature $\beta>0$, the Gibbs measure with pure $s$ boundary condition is defined by
\begin{equation*}
	\mu^{\beta,s}_{\phi,\Lambda_L}(\sigma) = \frac{\exp(-\beta  H^s_{\phi,\Lambda_L}(\sigma))}{Z} \text{ for } \sigma \in \Omega_{\Lambda_L},
\end{equation*}
where $Z$ is a normalizing constant so that the measure has total mass $1$. When the potential is not summable, i.e., when $\sum_x \phi(x)=\infty$, the Hamiltonian can (at least formally) not be defined as above, but looking at the minimizers of the energy one gets that $\sigma_x = s$ for all $x\in \Lambda_L$ with probability $1$ under the measure $\mu^{\beta,s}_{\phi,\Lambda_L}$. This trivial behavior was also described by Dyson \cite{dyson1969existence,friedli2006truncation}, as in the non-summable situation {\em ``[...] there is an infinite energy-gap between the ground states and all other states, so that the system is completely ordered at all finite temperatures, and there can be no question of a phase transition."}. Thus, one typically considers the Gibbs measure with respect to a summable potential. The non-summability $\sum_x \phi(x)=\infty$ means an infinite interaction volume, which in turn implies that the interaction has infinite range.
For a ferromagnetic potential $\phi:\Z^d \setminus \{\mz\} \to \left[0,\infty\right)$ with $\sum_{x} \phi(x) < \infty$, the infinite volume Gibbs measure $\mu^{\beta,s}_{\phi}$ is the subsequential limit of $\mu^{\beta,s}_{\phi,\Lambda_L}$ for $L\to \infty$.

For a potential $\phi : \Z^d \setminus \{\mz\} \to \left[0,\infty\right)$, we define the {\sl truncated potential} $\phi_N$ by $\phi_N(x)=\phi(x)\mathbbm{1}_{\{\|x\|\leq N\}}$. For this, we say that a potential $\phi$ is isotropic if the respective analogue of \eqref{eq:isotropy} holds $\phi$, we say that it is symmetric if the analogue of \eqref{eq:symmetry} holds, and we say that it is irreducible if the analogue of \eqref{eq:irreduc} holds. The result that
one can truncate a potential $\phi$ beyond some value $N \in \N$ and still have $\mu^{\beta,s}_{\phi_N}(\sigma_\mz = s) > \frac{1}{q}$ is the Potts analogue of the truncation question for percolation.

The next theorem is an analogue of Theorems \ref{theo:main} and \ref{theo:isotropic} for the Potts model. Part $(A)$ of Theorem \ref{theo:Potts} below is the analogue of Theorem \ref{theo:main}, whereas part $(B)$ is the analogue of Theorem \ref{theo:isotropic}. The proof of this theorem essentially follows from the previous results and the Fortuin–Kasteleyn (FK) representation \cite{fortuin1972random}.

\begin{theorem}\label{theo:Potts}
	\begin{enumerate}
		\item[(A)] 
		Let $\beta>0, d\geq 3$, and let $\phi:\Z^d \setminus \{\mz\} \to \left[0,\infty\right)$ be a potential that is not summable, irreducible, and symmetric. Then there exists $N \in \N$ such that
		\begin{equation*}
			\mu^{\beta,s}_{\phi_N}(\sigma_\mz = s) > \frac{1}{q}.
		\end{equation*}
		Further, 
		\begin{equation*}
			\lim_{N\to \infty} \mu^{\beta,s}_{\phi_N}(\sigma_\mz = s) = 1.
		\end{equation*}
		\item[(B)] 
		Let $d\geq 3$ and let $\phi:\Z^d \setminus \{\mz\} \to \left[0,\infty\right)$ be a potential that is summable, irreducible, and isotropic. Let $\beta > 0$, then
		\begin{equation*}
			\mu^{\beta,s}_{\phi}(\sigma_\mz = s) \geq 1 - \exp \left(1- \frac{\sum_{x \in \Z^d \setminus \{\mz\}} (\beta \phi(x)\wedge 1)}{2qT(d)} \right) ,
		\end{equation*}
		where $T(d)$ was defined in Theorem \ref{theo:isotropic}.
	\end{enumerate}
\end{theorem}

The theorem follows from Proposition \ref{prop:friedli}, which allows a comparison between the $q$-states Potts model and long-range percolation. The proposition is taken from \cite[Proposition 1]{friedli2006truncation}. See also \cite{aizenman1988discontinuity} and \cite[Section 6]{georgii2001random} for more details.

\begin{proposition}\label{prop:friedli}
	Let $\phi: \Z^d \setminus \{\mz\} \to \left[0,\infty\right)$ be a summable potential.
	Define $p_x$ by
	\begin{equation*}
		p_x \coloneqq \frac{1-e^{-2 \beta \phi(x)}}{1+(q-1) e^{-2 \beta \phi(x)}} .
	\end{equation*}
	Consider independent long-range percolation on $\Z^d$ where an edge $\{x,y\}$ is open with probability $p_{x-y}$, and write $\p(|K_\mz|=\infty)$ for the probability that the origin is contained in an infinite cluster.
	Then the long-range $q$-states Potts model with potential $\phi(x)$ at temperature $\beta > 0$ satisfies
	\begin{equation*}
		\mu_{\phi}^{\beta, s}\left(\sigma_{\mz}=s\right) \geq \frac{1}{q}+\frac{q-1}{q} \p(|K_\mz|=\infty) .
	\end{equation*}
\end{proposition}

Using this proposition, we can deduce Theorem \ref{theo:Potts}.

\begin{proof}[Proof of Theorem \ref{theo:Potts}]
	Define $(p_x)_{x\in \Z^d\setminus \{\mz\}}$ by
	\begin{equation*}
		p_x \coloneqq   \frac{1-e^{-2 \beta \phi(x)}}{1+(q-1) e^{-2 \beta \phi(x)}} 
	\end{equation*}
	and for $N\in \N_0$ define $(p_x^N)_{x\in \Z^d\setminus \{\mz\}}$ by
	\begin{equation*}
		p_x^N \coloneqq  \frac{1-e^{-2 \beta \phi_N(x)}}{1+(q-1) e^{-2 \beta \phi_N(x)}} = p_x \cdot \mathbbm{1}_{\{\|x\|\leq N\}}.
	\end{equation*}
	We write $\p_N$ for the measure of independent long-range percolation on $\Z^d$ where an edge $\{x,y\}$ is open with probability $p_{x-y}^N$. Assuming that $\sum_{x\in \Z^d \setminus \{\mz\}} \phi(x) = \infty$, one also gets that $\sum_{x\in \Z^d \setminus \{\mz\}} p_x = \infty$. Theorem \ref{theo:main} implies that
	\begin{equation*}
		\lim_{N\to \infty} \p_N(|K_\mz| = \infty) = 1. 
	\end{equation*}
	From this, one can directly deduce part $(A)$ of Theorem \ref{theo:Potts} by Proposition \ref{prop:friedli}. For part $(B)$ of Theorem \ref{theo:Potts}, observe that
	\begin{align*}
		p_x = \frac{1-e^{-2 \beta \phi(x)}}{1+(q-1) e^{-2 \beta \phi(x)}} 
		\geq
		\frac{1-e^{-2 \beta \phi(x)}}{q} \geq \frac{\tfrac{1}{2}(2\beta \phi(x)\wedge 1)}{q}
		\geq
		\frac{\beta \phi(x)\wedge 1}{2q}.
	\end{align*}
	In the second inequality, we used the elementary inequality $1-e^{-s}\geq \frac{1}{2} (s\wedge 1)$ for $s\geq 0$. Using Proposition \ref{prop:friedli} and Theorem \ref{theo:isotropic}, we thus get that
	\begin{align*}
		& \mu_{\phi}^{\beta, s}\left(\sigma_{\mz}=s\right) \geq \frac{1}{q}+\frac{q-1}{q} \p(|K_\mz|=\infty)
		\geq \p(|K_\mz|=\infty) 
		\overset{\eqref{eq:isotrop finite cluster low prob}}{\geq } 1 - \exp \left(1- \frac{\sum_{x\in \Z^d \setminus \{\mz\} p_x}}{T(d)} \right)   \\
		& \geq 1 - \exp \left(1- \frac{\sum_{x \in \Z^d \setminus \{\mz\}} (\beta \phi(x)\wedge 1)}{2qT(d)} \right) .
	\end{align*}
\end{proof}

\paragraph*{Acknowledgements.}

I thank Tom Hutchcroft for very useful discussions. I thank Marek Biskup and Aernout van Enter for helpful comments on an earlier draft. I thank an anonymous referee for useful comments and suggestions.


\begin{thebibliography}{1}
	
	\footnotesize
	
	\bibitem{aizenman1988discontinuity}
	Michael Aizenman, Jennifer~T. Chayes, Lincoln Chayes, and Charles~M. Newman.
	\newblock Discontinuity of the magnetization in one-dimensional $1/|x-y|^2$
	Ising and Potts models.
	\newblock {\em Journal of Statistical Physics}, 50:1--40, 1988.
	
	\bibitem{alves2017note}
	Caio Alves, Marcelo~R. Hil{\'a}rio, Bernardo~N. B. de~Lima, and Daniel Valesin.
	\newblock A note on truncated long-range percolation with heavy tails on
	oriented graphs.
	\newblock {\em Journal of Statistical Physics}, 169:972--980, 2017.
	
	\bibitem{baumler2023continuity}
	Johannes B{\"a}umler.
	\newblock Continuity of the critical value for long-range percolation.
	\newblock {\em arXiv preprint arXiv:2312.04099}, 2023.
	
	\bibitem{benjamini1998unpredictable}
	Itai Benjamini, Robin Pemantle, and Yuval Peres.
	\newblock Unpredictable paths and percolation.
	\newblock {\em The Annals of Probability}, 26(3):1198--1211, 1998.
	
	\bibitem{berger2002transience}
	Noam Berger.
	\newblock Transience, recurrence and critical behavior for long-range
	percolation.
	\newblock {\em Communications in mathematical physics}, 226(3):531--558, 2002.
	
	\bibitem{burton1989density}
	Robert~M. Burton and Michael Keane.
	\newblock Density and uniqueness in percolation.
	\newblock {\em Communications in mathematical physics}, 121:501--505, 1989.
	
	\bibitem{campos2020truncation}
	Alberto~M. Campos and Bernardo~N. B. de~Lima.
	\newblock Truncation of long-range percolation model with square non-summable
	interactions.
	\newblock {\em arXiv preprint arXiv:2009.13671}, 2020.
	
	\bibitem{cox1983oriented}
	J.~Theodore Cox and Richard Durrett.
	\newblock Oriented percolation in dimensions $d\geq 4$: bounds and asymptotic
	formulas.
	\newblock In {\em Mathematical Proceedings of the Cambridge Philosophical
		Society}, volume~93, pages 151--162. Cambridge University Press, 1983.
	
	\bibitem{dembin2022almost}
	Barbara Dembin and Vincent Tassion.
	\newblock Almost sharp sharpness for Poisson Boolean percolation.
	\newblock {\em arXiv preprint arXiv:2209.00999}, 2022.
	
	\bibitem{duminil2016new}
	Hugo Duminil-Copin and Vincent Tassion.
	\newblock A new proof of the sharpness of the phase transition for Bernoulli
	percolation and the Ising model.
	\newblock {\em Communications in Mathematical Physics}, 343(2) (2016) 725--745.
	
	\bibitem{duminil2017new}
	Hugo Duminil-Copin and Vincent Tassion.
	\newblock A new proof of the sharpness of the phase transition for Bernoulli
	percolation on $\mathbb{Z}^d$.
	\newblock {\em L’Enseignement math{\'e}matique}, 62(1) (2017) 199--206.
	
	\bibitem{dyson1969existence}
	Freeman~J. Dyson.
	\newblock Existence of a phase-transition in a one-dimensional Ising ferromagnet.
	\newblock {\em Communications in Mathematical Physics}, 12:91–107, 1969.
	
	\bibitem{easo2023critical}
	Philip Easo and Tom Hutchcroft.
	\newblock The critical percolation probability is local.
	\newblock {\em arXiv preprint arXiv:2310.10983}, 2023.
	
	\bibitem{van2016truncated}
	Aernout~C. D. van Enter, Bernardo~N. B. de~Lima, and Daniel Valesin.
	\newblock Truncated long-range percolation on oriented graphs.
	\newblock {\em Journal of Statistical Physics}, 164:166--173, 2016.
	
	\bibitem{fortuin1972random}
	Cornelius~Marius Fortuin and Piet~W. Kasteleyn.
	\newblock On the random-cluster model: I. introduction and relation to other
	models.
	\newblock {\em Physica}, 57(4):536--564, 1972.
	
	\bibitem{fortuiniii1972random}
	Cornelis~Marius Fortuin.
	\newblock On the random-cluster model: II. the percolation model.
	\newblock {\em Physica}, 58(3):393--418, 1972.
	
	\bibitem{fortuinii1972random}
	Cornelis~Marius Fortuin.
	\newblock On the random-cluster model: III. the simple random-cluster model.
	\newblock {\em Physica}, 59(4):545--570, 1972.
	
	\bibitem{friedli2004longrange}
	Sacha Friedli, Bernardo N. B. de~Lima, and Vladas Sidoravicius.
	\newblock {On Long Range Percolation with Heavy Tails}.
	\newblock {\em Electronic Communications in Probability}, 9:175 -- 177,
	2004.
	
	\bibitem{friedli2006truncation}
	Sacha~Friedli and Bernardo N. B. de~Lima.
	\newblock On the truncation of systems with non-summable interactions.
	\newblock {\em Journal of Statistical Physics}, 122(6):1215--1236, 2006.
	
	\bibitem{georgii2001random}
	Hans-Otto Georgii, Olle H{\"a}ggstr{\"o}m, and Christian Maes.
	\newblock The random geometry of equilibrium phases.
	\newblock In {\em Phase transitions and critical phenomena}, volume~18, pages
	1--142. Elsevier, 2001.
	
	\bibitem{gomes2021upper}
	Pablo~Almeida Gomes, Alan Pereira, and Remy Sanchis.
	\newblock Upper bounds for critical probabilities in bernoulli percolation
	models.
	\newblock {\em arXiv preprint arXiv:2106.10388}, 2021.
	
	\bibitem{grimmett2006book}
	Geoffrey R.~Grimmett.
	\newblock The Random-Cluster Model,
	\newblock {\em volume 333 of Grundlehren der Mathematischen Wissenschaften}. Springer, 2006.
	
	\bibitem{grimmett1984connectedness}
	Geoffrey R.~Grimmett, Micheal~Keane, and John M.~Marstrand.
	\newblock On the connectedness of a random graph.
	\newblock In {\em Mathematical Proceedings of the Cambridge Philosophical
		Society}, volume~96, pages 151--166. Cambridge University Press, 1984.
	
	\bibitem{grimmett1990supercritical}
	Geoffrey~R. Grimmett and John~M. Marstrand.
	\newblock The supercritical phase of percolation is well behaved.
	\newblock {\em Proceedings of the Royal Society of London. Series A:
		Mathematical and Physical Sciences}, 430(1879):439--457, 1990.
	
	\bibitem{haggstrom1998nearest}
	Olle H{\"a}ggstr{\"o}m and Elchanan Mossel.
	\newblock Nearest-neighbor walks with low predictability profile and
	percolation in $2+\eps$ dimensions.
	\newblock {\em The Annals of Probability}, 26(3):1212--1231, 1998.
	
	
	
	\bibitem{heydenreich2017progress}
	Markus Heydenreich and Remco Van~der Hofstad.
	\newblock {\em Progress in high-dimensional percolation and random graphs}.
	\newblock Springer, 2017.	
	
	
	\bibitem{hoffman1998unpredictable}
	Christopher Hoffman.
	\newblock Unpredictable nearest neighbor processes.
	\newblock {\em Annals of probability}, pages 1781--1787, 1998.
	
	
	\bibitem{kalikow1988random}
	Steven Kalikow and Benjamin Weiss.
	\newblock When are random graphs connected.
	\newblock {\em Israel journal of mathematics}, 62:257--268, 1988.
	
	
	\bibitem{kesten1990asymptotics}
	Harry Kesten.
	\newblock Asymptotics in high dimensions for percolation.
	\newblock {\em Disorder in physical systems}, pages 219--240, 1990.
	
	\bibitem{klenke2013probability}
	Achim Klenke.
	\newblock {\em Probability theory: a comprehensive course}.
	\newblock Springer Science \& Business Media, 2013.
	
	\bibitem{de2008truncated}
	Bernardo N. B. de~Lima and Art{\"e}m Sapozhnikov.
	\newblock On the truncated long-range percolation on $\Z^2$.
	\newblock {\em Journal of applied probability}, 45(1):287--291, 2008.
	
	
	\bibitem{meester1996continuity}
	Ronald Meester and Jeffrey~E. Steif.
	\newblock On the continuity of the critical value for long range percolation in
	the exponential case.
	\newblock {\em Communications in mathematical physics}, 180(2):483--504, 1996.
	
	\bibitem{menshikov2001note}
	Mikhail~Menshikov, Vladas~Sidoravicius, and Marina~Vachkovskaia.
	\newblock A note on two-dimensional truncated long-range percolation.
	\newblock {\em Advances in Applied Probability}, 33(4):912--929, 2001.
	
	\bibitem{monch2023inhomogeneous}
	Christian M{\"o}nch.
	\newblock Inhomogeneous long-range percolation in the weak decay regime.
	\newblock {\em arXiv preprint arXiv:2303.02027}, 2023.
	
	\bibitem{nash1959random}
	C.~St.~J.~A. Nash-Williams.
	\newblock Random walk and electric currents in networks.
	\newblock In {\em Mathematical Proceedings of the Cambridge Philosophical
		Society}, volume~55, pages 181--194. Cambridge University Press, 1959.
	
	\bibitem{penrose1993spread}
	Mathew~D. Penrose.
	\newblock On the spread-out limit for bond and continuum percolation.
	\newblock {\em The Annals of Applied Probability}, 3(1):253--276, 1993.
	
	\bibitem{robbins1955remark}
	Herbert Robbins.
	\newblock A Remark on Stirling's Formula.
	\newblock {\em The American Mathematical Monthly}, 62(1):26--29, 1955.
	
	\bibitem{sidoravicius1999truncated}
	Vladas~Sidoravicius, Donatas~Surgailis, and Maria E.~Vares.
	\newblock On the truncated anisotropic long-range percolation on $\Z^2$.
	\newblock {\em Stochastic processes and their applications}, 81(2):337--349,
	1999.
	
	\bibitem{spanos2024spread}
	Panagiotis Spanos and Matthew Tointon.
	\newblock Spread-out percolation on transitive graphs of polynomial growth.
	\newblock {\em arXiv preprint arXiv:2404.17262}, 2024.
	
	
\end{thebibliography}
\end{document}